\documentclass{m2an}

\usepackage[T1]{fontenc}
\usepackage{amsmath,amsthm,amsfonts,amssymb,bm}
\usepackage{mathtools}
\usepackage{enumitem}\setitemize{leftmargin=5mm}
\usepackage{accents}
\usepackage{color}
\usepackage{booktabs}
\usepackage[usenames,dvipsnames,svgnames,table]{xcolor}
\usepackage{cite}
\usepackage{soul}
\usepackage{hyperref}
\hypersetup{
   colorlinks=true,
   urlcolor=green!70!black,
   citecolor=blue!70!black,
   linkcolor=green!70!black
}

\usepackage{verbatim}
\usepackage{aligned-overset}
\usepackage{csvsimple}

\usepackage{graphicx}
\usepackage{fancyhdr}

\usepackage{todonotes}
\usepackage{caption}
\usepackage{subcaption}

\usepackage{tikz}
\usetikzlibrary{cd,patterns,calc,fit,backgrounds}
\usepackage{pgfplots,pgfplotstable}
\usepgfplotslibrary{colorbrewer,groupplots}
\usepackage{caption}

\pgfplotsset{
cycle list/Set1-5,
cycle multiindex* list={
mark list*\nextlist
Set1-5\nextlist
},
}

\usepackage[capitalize]{cleveref}
\usepackage{stmaryrd} 
\usepackage{pgf,tikz}
\usetikzlibrary{spy,matrix, calc, arrows,shapes,decorations}
\usepackage{tkz-euclide}
\usetikzlibrary{positioning,arrows,calc,decorations.markings,math,arrows.meta}
\usepackage{pgfplots} 
\usepgfplotslibrary{groupplots}
\usepackage{csvsimple} 
\usepackage{siunitx} 
\usepackage{tabularx}
\usepackage{booktabs}
\usepackage{multirow}
\usepackage[normalem]{ulem}
\useunder{\uline}{\ul}{}

\pgfplotscreateplotcyclelist{barcolors}{%
teal,fill=teal,fill opacity=0.7,thick,postaction={pattern=crosshatch dots,pattern color=teal}\\
red!70!white,fill=red!70!white,fill opacity=0.7,thick,postaction={pattern=crosshatch dots, pattern color=red!70!white}\\
ForestGreen,fill=ForestGreen,fill opacity=0.7,thick,postaction={pattern=crosshatch dots, pattern color=ForestGreen}\\
cyan!60!black,fill=cyan!60!black,fill opacity=0.7,thick,postaction={pattern=crosshatch dots, pattern color=cyan!60!black}\\
orange,fill=orange,fill opacity = 0.7,thick,postaction={pattern=crosshatch dots, pattern color=orange}\\
}

\pgfplotscreateplotcyclelist{MachNum}{%
 teal,every mark/.append style={solid,fill=teal!80!black},mark=*,mark size=2.5pt\\
 teal,every mark/.append style={solid,fill=teal!80!black},mark=*,mark size=2.5pt\\
 ForestGreen,every mark/.append style={solid,fill=ForestGreen!80!black},mark=triangle*,mark size=2.5pt\\
 ForestGreen,every mark/.append style={solid,fill=ForestGreen!80!black},mark=triangle*,mark size=2.5pt\\
}

\pgfplotscreateplotcyclelist{LiftvsSIP}{%
 teal,every mark/.append style={solid,fill=teal!80!black},mark=*,mark size=2.5pt,opacity=.8\\
 orange,every mark/.append style={solid,fill=orange!80!black},mark=square*,mark size=2.5pt,opacity=.8\\
 ForestGreen,every mark/.append style={solid,fill=ForestGreen!80!black},mark=triangle*,mark size=2.5pt,opacity=.8\\
 red!70!white,every mark/.append style={solid,fill=red!80!black},mark=pentagon*,mark size=2.5pt\\
}

\pgfplotscreateplotcyclelist{ConvPlot}{%
 teal,every mark/.append style={solid,fill=teal!80!black},mark=*,mark size=2.5pt\\
 red!70!white,every mark/.append style={solid,fill=red!80!black},mark=pentagon*,mark size=2.5pt\\
 ForestGreen,every mark/.append style={solid,fill=ForestGreen!80!black},mark=triangle*,mark size=2.5pt\\
 cyan!60!black,every mark/.append style={solid,fill=cyan!80!black},mark=diamond*,mark size=2.5pt\\
 orange,every mark/.append style={solid,fill=orange!80!black},mark=square*,mark size=2.5pt\\
 ForestGreen,every mark/.append style={solid,fill=ForestGreen!80!black},mark=triangle*,mark size=2.5pt\\
 cyan!60!black,every mark/.append style={solid,fill=cyan!80!black},mark=diamond*,mark size=2.5pt\\
}

\pgfplotscreateplotcyclelist{unfmixed1}{%
 teal,every mark/.append style={solid,fill=teal!80!black},mark=*,mark size=2.5pt\\
 red!70!white,every mark/.append style={solid,fill=red!80!black},mark=pentagon*,mark size=2.5pt\\
 ForestGreen,every mark/.append style={solid,fill=ForestGreen!80!black},mark=triangle*,mark size=2.5pt\\
cyan!60!black,every mark/.append style={solid,fill=cyan!80!black},mark=diamond*,mark size=2.5pt\\
orange,every mark/.append style={solid,fill=orange!80!black},mark=square*,mark size=2.5pt\\
}

\pgfplotscreateplotcyclelist{unfmixed2}{%
teal,every mark/.append style={solid,fill=teal!80!black},mark=*,mark size=2.5pt\\
orange,every mark/.append style={solid,fill=orange!80!black},mark=square*,mark size=2.5pt\\
}

\pgfplotscreateplotcyclelist{unfmixed3}{%
 orange,every mark/.append style={solid,fill=orange!80!black},mark=square*,mark size=2.5pt\\
 teal,every mark/.append style={solid,fill=teal!80!black},mark=*,mark size=2.5pt\\
 cyan!60!black,every mark/.append style={solid,fill=cyan!80!black},mark=diamond*,mark size=2.5pt\\
 red!70!white,every mark/.append style={solid,fill=red!80!black},mark=pentagon*,mark size=2.5pt\\
 ForestGreen,every mark/.append style={solid,fill=ForestGreen!80!black},mark=triangle*,mark size=2.5pt\\
}

\pgfplotscreateplotcyclelist{paulcolors}{%
violet,every mark/.append style={solid,fill=violet},mark=square*,very thick,mark size=3pt\\%
teal,every mark/.append style={solid,fill=teal},mark=*,very thick,mark size=3pt\\%
orange,every mark/.append style={solid,fill=orange},mark=diamond*,very thick,mark size=3pt\\%
violet,densely dashed,every mark/.append style={solid,fill=violet},mark=square*,very thick,mark size=3pt\\%
teal,densely dashed,every mark/.append style={solid,fill=teal},mark=*,very thick,mark size=3pt\\%
orange,densely dashed,every mark/.append style={solid,fill=orange},mark=diamond*,very thick,mark size=3pt\\%
violet,dash dot,every mark/.append style={solid,fill=violet},mark=square*,very thick,mark size=3pt\\%
teal,dash dot,every mark/.append style={solid,fill=teal},mark=*,very thick,mark size=3pt\\%
orange,dash dot,every mark/.append style={solid,fill=orange},mark=diamond*,very thick,mark size=3pt\\%
}
\pgfplotscreateplotcyclelist{paulcolors2}{%
teal,every mark/.append style={solid,fill=teal},mark=*,very thick,mark size=3pt\\%
orange,every mark/.append style={solid,fill=orange},mark=diamond*,very thick,mark size=3pt\\%
teal,densely dashed,every mark/.append style={solid,fill=teal},mark=*,very thick,mark size=3pt\\%
orange,densely dashed,every mark/.append style={solid,fill=orange},mark=diamond*,very thick,mark size=3pt\\%
teal,dash dot,every mark/.append style={solid,fill=teal},mark=*,very thick,mark size=3pt\\%
orange,dash dot,every mark/.append style={solid,fill=orange},mark=diamond*,very thick,mark size=3pt\\%
}
\pgfplotscreateplotcyclelist{paulcolors1}{%
teal,every mark/.append style={solid,fill=teal},mark=*,very thick,mark size=3pt\\%
}
\pgfplotscreateplotcyclelist{paulcolorsfill}{%
teal,fill=teal,fill opacity=0.5,thick\\
orange,fill=orange,fill opacity=0.5,thick\\
violet,fill=violet,fill opacity=0.5,thick\\
}

\pgfplotsset{
    discard if not/.style 2 args={
        x filter/.append code={
            \edef\tempa{\thisrow{#1}}
            \edef\tempb{#2}
            \ifx\tempa\tempb
            \else
                
            \fi
        }
    }
}
\pgfplotsset{compat=1.16}

\pgfplotsset{tick label style={font=\small},label style={font=\small},legend style={font=\small},}
\pgfplotsset{ width=.49\linewidth}

\usepackage{listings}
\usepackage[plain]{algorithm}

\definecolor{maroon}{cmyk}{0, 0.87, 0.68, 0.32}
\definecolor{halfgray}{gray}{0.55}
\definecolor{ipython_frame}{RGB}{207, 207, 207}
\definecolor{ipython_bg}{RGB}{247, 247, 247}
\definecolor{ipython_red}{RGB}{186, 33, 33}
\definecolor{ipython_green}{RGB}{0, 128, 0}
\definecolor{ipython_cyan}{RGB}{64, 128, 128}
\definecolor{ipython_purple}{RGB}{170, 34, 255}

\lstdefinelanguage{iPython}{
    morekeywords=[2]{abs,all,any,basestring,bin,bool,bytearray,callable,chr,classmethod,cmp,compile,complex,delattr,dict,dir,divmod,enumerate,eval,execfile,file,filter,float,format,frozenset,getattr,globals,hasattr,hash,help,hex,id,input,int,isinstance,issubclass,iter,len,list,locals,long,map,max,memoryview,min,next,object,oct,open,ord,pow,property,range,raw_input,reduce,reload,repr,reversed,round,set,setattr,slice,sorted,staticmethod,str,sum,super,tuple,type,unichr,unicode,vars,xrange,zip,apply,buffer,coerce,intern},%
    sensitive=true,%
    morecomment=[l]\#,%
    morestring=[b]',%
    morestring=[b]",%
    morestring=[s]{'''}{'''},
    morestring=[s]{"""}{"""},
    morestring=[s]{r'}{'},
    morestring=[s]{r"}{"},%
    morestring=[s]{r'''}{'''},%
    morestring=[s]{r"""}{"""},%
    morestring=[s]{u'}{'},
    morestring=[s]{u"}{"},%
    morestring=[s]{u'''}{'''},%
    morestring=[s]{u"""}{"""},%
    literate=
    {á}{{\'a}}1 {é}{{\'e}}1 {í}{{\'i}}1 {ó}{{\'o}}1 {ú}{{\'u}}1
    {Á}{{\'A}}1 {É}{{\'E}}1 {Í}{{\'I}}1 {Ó}{{\'O}}1 {Ú}{{\'U}}1
    {à}{{\`a}}1 {è}{{\`e}}1 {ì}{{\`i}}1 {ò}{{\`o}}1 {ù}{{\`u}}1
    {À}{{\`A}}1 {È}{{\'E}}1 {Ì}{{\`I}}1 {Ò}{{\`O}}1 {Ù}{{\`U}}1
    {ä}{{\"a}}1 {ë}{{\"e}}1 {ï}{{\"i}}1 {ö}{{\"o}}1 {ü}{{\"u}}1
    {Ä}{{\"A}}1 {Ë}{{\"E}}1 {Ï}{{\"I}}1 {Ö}{{\"O}}1 {Ü}{{\"U}}1
    {â}{{\^a}}1 {ê}{{\^e}}1 {î}{{\^i}}1 {ô}{{\^o}}1 {û}{{\^u}}1
    {Â}{{\^A}}1 {Ê}{{\^E}}1 {Î}{{\^I}}1 {Ô}{{\^O}}1 {Û}{{\^U}}1
    {œ}{{\oe}}1 {Œ}{{\OE}}1 {æ}{{\ae}}1 {Æ}{{\AE}}1 {ß}{{\ss}}1
    {ç}{{\c c}}1 {Ç}{{\c C}}1 {ø}{{\o}}1 {å}{{\r a}}1 {Å}{{\r A}}1
    {€}{{\EUR}}1 {£}{{\pounds}}1
    {^}{{{\color{ipython_purple}\^{}}}}1
    {=}{{{\color{ipython_purple}=}}}1
    {+}{{{\color{ipython_purple}+}}}1
    {*}{{{\color{ipython_purple}$^\ast$}}}1
    {/}{{{\color{ipython_purple}/}}}1
    {+=}{{{+=}}}1
    {-=}{{{-=}}}1
    {*=}{{{$^\ast$=}}}1
    {/=}{{{/=}}}1,
    literate=
    *{-}{{{\color{ipython_purple}-}}}1
     {?}{{{\color{ipython_purple}?}}}1,
    identifierstyle=\color{black}\ttfamily,
    commentstyle=\color{ipython_cyan}\ttfamily,
    stringstyle=\color{ipython_red}\ttfamily,
    keepspaces=true,
    showspaces=false,
    showstringspaces=false,
    rulecolor=\color{ipython_frame},
    framexleftmargin=0mm,
    numbers=left,
    numberstyle=\tiny\color{halfgray},
    numbersep=1mm,
    xleftmargin=1mm,
    basicstyle=\scriptsize,
    keywordstyle=\color{ipython_green}\ttfamily,
}

\lstdefinestyle{trefftzy}{
    language=iPython,
    emptylines=1,
    breaklines=true,
    basicstyle=\footnotesize\ttfamily\color{black},    
    moredelim=**[is][\color{teal}]{<}{>},
    moredelim=**[is][\color{purple}]{'}{'},
}

\newcommand{\inner}[1]{\langle #1 \rangle}
\newcommand{\innerr}[1]{\langle \rho#1 \rangle}
\newcommand{\innercsr}[1]{\langle c_s^2 \rho#1 \rangle}
\newcommand{\csr}{c_s^2 \rho}

\newcommand\restr[2]{{ \left.\kern-\nulldelimiterspace #1 \vphantom{\big|} \right|_{#2} }}
\newcommand{\dom}{\mathcal{O}} 
\newcommand{\nv}{\boldsymbol{\nu}}
\newcommand{\bflow}{\mathbf{b}}

\newcommand{\angvel}{\Omega}
\newcommand{\conv}{\partial_{\bflow}} 
\newcommand{\opd}{(\omega+i\conv+i\angvel\times)}
\newcommand{\dopd}{(\omega+i\ddiffb+i\angvel\times)}
\newcommand{\dopdSmall}{(\omega\!+\!i\ddiffb\!\!+\!i\angvel\times)}
\renewcommand{\u}{\mathbf{u}}
\renewcommand{\v}{\mathbf{v}}
\newcommand{\w}{\mathbf{w}}
\newcommand{\z}{\mathbf{z}}
\newcommand{\q}{\mathbf{q}} 

\renewcommand{\div}{\operatorname{div}}

\DeclareMathOperator{\hess}{Hess}
\DeclareMathOperator{\Id}{\mathrm{id}}
\DeclareMathOperator{\tr}{\mathrm{tr}} 
\DeclareMathOperator{\trF}{\mathrm{tr} \vert_{\Fn}}

\newcommand{\calT}{\mathcal{T}}

\newcommand{\IN}{\mathbb{N}}

\newcommand{\X}{\mathbb{X}}
\newcommand{\Xn}{\mathbb{X}_n}
\newcommand{\XT}{\X_{\mathcal{T}_n}}
\newcommand{\XF}{\X_{\mathcal{F}_n}}

\newcommand{\bH}{{\mathbf H}}

\newcommand{\bD}{\mathbf{D}}

\newcommand{\bRl}{\mathbf{R}^{\!l}}

\newcommand{\bW}{\mathbf{W}}

\newcommand{\bL}{\mathbf{L}}

\newcommand{\bu}{\mathbf{u}}

\newcommand{\bff}{\mathbf{f}}

\newcommand{\jump}[1]{[\![ #1 ]\!]}
\newcommand{\hdgjump}[1]{[\![\underline{#1} ]\!]}

\newcommand{\Tn}{\calT_n}

\newcommand{\N}{\mathbb{N}}

\newcommand{\seq}[1]{(#1_n)_{n \in \mathbb{N}}}
\newcommand{\spl}{\langle}
\newcommand{\spr}{\rangle}

\newcommand{\diffb}{\partial_{\bflow}}
\newcommand{\ddiffb}{\bD_{\bflow}^n}
\newcommand{\ddiv}{\div_{\nv}^n}
\newcommand{\Fn}{\mathcal{F}_n}
\newcommand{\Sn}{\mathcal{S}_n}

\newcommand{\jnorm}[2]{[ \! [\![\underline{#1}]\!] \! ]_{\partial \Tn,1\hspace*{-0.3mm}/2,\nv}^{#2}}

\newcommand{\matii}{\underline{\underline{m}}}

\renewcommand\Re{\operatorname{Re}}
\renewcommand\Im{\operatorname{Im}}

\newcommand{\divPqM}{( \div + P_{L^2_0} \q \cdot \! + M )}
\newcommand{\divPq}{( \div + P_{L^2_0} \q \cdot )}
\newcommand{\divpqM}{( \div + \pi_n^l P_{L^2_0} \q \cdot \! + M )}

\newcommand{\ddivpqMn}{(\ddiv + \pi_n^l P_{L^2_0} \q \cdot \! + M_n)}

\newcommand{\PVol}[1]{(#1)_{\tau}}
\newcommand{\PFac}[1]{(#1)_{F}}

\newcommand{\pinX}{\underline{\pi}_n}
\newcommand{\pinXd}{\underline{\tilde{\pi}}_n^d}

\newcommand{\h}{\mathfrak{h}}

\newcommand{\adiv}{a^{\div}}
\newcommand{\aconv}{a^{\diffb}}
\newcommand{\arem}{a^{r}}
\newcommand{\adivn}{\hyperref[eq:an:components:adivn]{a^{\div}_n}}
\newcommand{\aconvn}{\hyperref[eq:an:components:aconvn]{a^{\diffb}_n}}
\newcommand{\aremn}{\hyperref[eq:an:components:aremn]{a^{r}_n}}
\newcommand{\adivq}{a^{(\div+\q\cdot)}}
\newcommand{\aremq}{a^{(r-\q\cdot)}}

\newcommand{\Sol}{S_n}

\newcommand{\numQ}[1]{\num[
  fixed-exponent = 3,
  round-precision = 12,
  round-mode=figures,
  group-minimum-digits = 4,
  group-separator = {\,},
  drop-zero-decimal = true,
  table-format = 5.0]{#1}}
\usepackage{collcell}
\newcolumntype{N}{>{\collectcell\numQ}r<{\endcollectcell}}

\definecolor{mhcol}{rgb}{0,0,128}

\definecolor{editscol}{rgb}{0.6, 0.2, 0.8}

\newcommand\converges{\hyperlink{converge}{converges}}
\newcommand\convergesl{\hyperlink{converge}{$\rightarrow$}}
\newcommand\convergence{\hyperlink{converge}{convergence}}
\newcommand\stable{\hyperlink{stable}{stable}}
\newcommand\stability{\hyperlink{stable}{stability}}
\newcommand\regular{\hyperlink{regular}{regular}}
\newcommand\regularity{\hyperlink{regular}{regularity}}
\newcommand\compactseq{\hyperlink{compactseq}{compact}}
\newcommand\compactop{\hyperlink{compactop}{compact}}

\newcommand\approximates{\hyperlink{approximate}{approximates}}

\newcommand{\weaklyTcoercive}{\hyperlink{weakTcoercivity}{weakly T-coercive}}
\newcommand{\DAS}{\hyperlink{DAS}{DAS}}

\newcommand{\DBconst}{\hyperlink{DBconst}{C_{(20)}}}

\newcommand{\AddDBconst}{C_{\text{dt},\pi}}

\newcommand{\BnI}{B_n^{(1)}}
\newcommand{\BnII}{B_n^{(2)}}
\newcommand{\KnI}{K_n^{(1)}}
\newcommand{\KnII}{K_n^{(2)}}
\newcommand{\anI}{a_n^{(1)}}
\newcommand{\anII}{a_n^{(2)}}

\definecolor{darkolivegreen}{rgb}{0.33, 0.42, 0.18}
\colorlet{v}{darkolivegreen}
\colorlet{w}{blue!60!black}
\colorlet{vv}{darkolivegreen}
\colorlet{ww}{blue!60!black}
\colorlet{vw}{orange!60!black}
\colorlet{wv}{violet!90!black}

\newcommand\colorv[1]{\textcolor{v}{#1}}
\newcommand\colorw[1]{\textcolor{w}{#1}}
\newcommand\colorvv[1]{\textcolor{vv}{#1}}
\newcommand\colorww[1]{\textcolor{ww}{#1}}
\newcommand\colorvw[1]{\textcolor{vw}{#1}}
\newcommand\colorwv[1]{\textcolor{wv}{#1}}

\newcommand\markI{\colorv{\text{(I)}}}
\newcommand\markII{\text{(II)}}
\newcommand\markIII{\colorw{\text{(III)}}}
\newcommand\markIV{\colorw{\text{(IV)}}}

\newcommand{\changed}[1]{{\color{black}{#1}}}

\theoremstyle{plain}
\newtheorem{theorem}{Theorem}
\newtheorem{lemma}[theorem]{Lemma}

\theoremstyle{definition}
\newtheorem{definition}[theorem]{Definition}
\newtheorem{remark}[theorem]{Remark}
\newtheorem{assumption}[theorem]{Assumption}

\begin{document}
\title[HDG discretizations for Galbrun's equation]{Hybrid discontinuous Galerkin discretizations for the damped time-harmonic Galbrun's equation}

\author{Martin Halla}\address{Institute for Applied and Numerical Mathematics, Karlsruhe Institute of Technology, Englerstra{\ss}e 2, 76131 Karlsruhe, Germany; \email{martin.halla@kit.edu}}
\author{Christoph Lehrenfeld}\address{Institute for Numerical and Applied Mathematics, University of Göttingen, Lotzestr. 16-18, 37083 Göttingen, Germany; \email{\{lehrenfeld, t.beeck\}@math.uni-goettingen.de}}
\author{Tim van Beeck}\sameaddress{2}

\date{\today}




%
\begin{abstract}
   In this article, we study the damped time-harmonic Galbrun's equation which models solar and stellar oscillations. We introduce and analyze hybrid discontinuous Galerkin discretizations (HDG) that are stable and optimally convergent for all polynomial degrees greater than or equal to one. The proposed methods are robust with respect to the drastic changes in the magnitude of the coefficients that naturally occur in stars. Our analysis is based on the concept of discrete approximation schemes and weak T-compatibility, which exploits the weakly T-coercive structure of the equation. Compared to the $H^1$-conforming discretization of [Halla, Lehrenfeld, Stocker, 2022], our method offers improved stability and robustness. Furthermore, it significantly reduces the computational costs compared to the $H(\div)$-conforming DG discretization of [Halla, 2023], which has similar stability properties. These advantages make the proposed HDG methods well-suited for astrophysical simulations.
\end{abstract}
%
%

\maketitle
\section{Introduction}
Helioseismology studies the interior of the Sun through acoustic oscillations measured at the surface \cite{GBS10}. Reconstructing physical quantities in the interior, such as the density, the sound speed, or subsurface flows, requires solving a passive imagining problem.  To tackle this problem, approaches such as \emph{helioseismic holography} \cite{LB97,BH24} rely on an accurate and computationally efficient solution of the forward problem. To model solar and stellar oscillations, we consider the damped time-harmonic \emph{Galbrun's equation}: Find $\u : \dom \to \mathbb{C}^d$ such that 
\begin{subequations}\label{eq:Galbrun}
\begin{align}
   \begin{aligned}
      -\rho\opd^2\u 
      &- \nabla\left(\rho c_s^2\div \u\right)
      + (\div \u) \nabla p -\nabla(\nabla p\cdot \u)\\
      & + (\hess(p)-\rho\hess(\phi))\u 
      + \gamma \rho (-i \omega) \u
      = \bff \quad \text{ in } \dom,
   \end{aligned}& \\ \label{eq:GalbrunBc}
   \nv \cdot \u = 0 \quad \text{ on } \partial &\dom,
\end{align}
\end{subequations}
where $\dom \subset \mathbb{R}^d$, $d = 2,3$, is a bounded Lipschitz domain.
We denote by $\rho$ the density, by $c_s$ the sound speed, by $p$ the pressure, by $\phi$ the gravitational potential, and by $\bff$ the source term. Furthermore, $\angvel \in \mathbb{R}^d$ is the angular velocity of the frame of reference, $\omega$ is the frequency and $\bflow$ is the background velocity. The operator $\diffb \coloneqq \sum_{l = 1}^d \bflow_l \partial_{x_l}$ is the directional derivative in the direction of $\bflow$, $\hess(\cdot)$ is the Hessian, and by $\nv$, we denote the exterior unit normal vector on $\partial \dom$. The damping is modeled with the term $-i \gamma \rho \omega \u$, where $\gamma$ is a scalar damping coefficient. \changed{While \eqref{eq:Galbrun} is sometimes reduced to a scalar equation in helioseismology \cite{AA2017,  BFFGP20}, or convected Helmholtz equations \cite{BRT23} are considered instead, the full vectorial equation is necessary to capture key physical phenomena such as the emergence of $f$- and $g$-modes. Thus, it is highly relevant to study the full equations \eqref{eq:Galbrun}.

In this article, we focus on the boundary condition \eqref{eq:GalbrunBc}. Accounting for wave propagation in the atmosphere, however, requires radiating boundary conditions have to be considered, which have been investigated in \cite{BFFGP20,HLP2021,FHPG24} for the scalar problem and in \cite{BFFGP21,PFFBG24} for the vectorial problem.}

Galbrun's equation was first derived in \cite{Galbrun31} and is a linearization of the nonlinear Euler equations with the Lagrangian perturbation of displacement as unknown. Without the additional rotational terms as in \eqref{eq:Galbrun}, it is commonly applied in aeroacoustics \cite{90yrsGalbrun}. For the well-posed analysis in the time domain, we refer to \cite{HB21}. \changed{One approach in aeroacoustics is a stabilized formulation leveraging an additional transport equation for the vorticity, as analyzed in \cite{BMMPP12} for the time-harmonic setting.} 

In asteroseismology, the strong damping of waves allows for a more direct analysis as achieved in \cite{HH21}, where the well-posedness of problem \eqref{eq:Galbrun} has been shown assuming that the Mach number $\Vert c_s^{-1} \bflow \Vert_{\bL^\infty}$ of the background flow $\bflow$ is bounded suitably. The main ingredient of the proof is a generalized Helmholtz decomposition and a weak T-coercivity argument. Here, we call a problem weakly T-coercive if it is a compact perturbation of a T-coercive problem.
The T-coercivity technique \cite{BCS02,BCZ10,CC13} relies on the explicit construction of an operator that realizes the inf-sup condition.  
This approach has been successfully applied to a variety of problems, including Helmholtz-like problems \cite{Cia12,HN15,Halla21,vBZ24,FvBZ25} and problems with sign-changing coefficients \cite{BCZ10,DCC12,DCC14,HH024}. 

The key to develop stable discretizations of (weakly) T-coercive problems is to transfer the construction of the T-operator to the discrete level in a stable manner. In particular, the stability of the discretization is obtained when the constructions fulfill a \emph{T-compatibility} condition \cite{Halla21,HLS22H1}.

The construction and analysis of reliable finite element schemes for \eqref{eq:Galbrun} was initiated in \cite{HLS22H1}, where suitable $H^1$-conforming discretizations were considered, primarily to circumvent the challenges associated with analyzing non-conforming methods. Since the stability of the discrete divergence operator is essential for stable discretizations of \eqref{eq:Galbrun} \cite{AHLS22,HLS22H1}, $H^1$-conforming discretizations of \eqref{eq:Galbrun} face similar restrictions as those encountered for finite element discretizations for the Stokes problem. As with the Scott-Vogelius \cite{SV85} pair, the polynomial degree has to be sufficiently large (e.g., $k \ge 4$ in 2d and $k \ge 8$ in 3d) and/or special meshes (e.g., barycentric refinements) have to be used. In addition, the assumed bound on the Mach number lacked robustness with respect to changes in magnitude of the physical parameters.

The challenges posed by non-conforming discretizations were then overcome in \cite{H23Hdiv} for $H(\div)$-conforming and in \cite{Thesis_vB23} for fully discontinuous Galerkin (DG) finite elements. Notably, these schemes are stable for all polynomial degrees $k \ge 1$, and the assumed bound on the Mach number remains robust even in the presence of highly heterogeneous physical parameters. To achieve this, the directional derivative $\diffb$ is stabilized through a \emph{lifting operator} \cite{BR97,BMMPR00}, which ensures stability without the need to choose a suitable penalty parameter. 

However, these developments were primarily motivated by theoretical considerations, and they lack computational efficiency, in particular because the lifting operator drastically increases the computational costs in a DG  setting (see \cref{rem:numex:LiftingCosts}). Thus, we propose hybrid discontinuous Galerkin (HDG) discretizations of \eqref{eq:Galbrun} in the current work. The key idea of \emph{hybridization} \cite{CGL09,Leh10} is to introduce additional facet unknowns, which increases the total number of degrees of freedom but reduces the number of global couplings. Due to the resulting structure of the linear system, static condensation can be applied to eliminate the volume unknowns, leading to a significant reduction in the computational costs. Furthermore, in the hybrid setting, relying on a lifting operator to stabilize $\diffb$ is feasible, since it is a local operator. Our analysis extends the work of \cite{H23Hdiv,Thesis_vB23} to the hybrid setting, covering both the fully non-conforming and the $H(\div)$-conforming case. We show stability and quasi-optimality for all polynomial degrees $k \ge 1$ (\changed{where $k$ is the polynomial degree of the volume unknown}), and the required boundedness assumption on the Mach number is robust with respect to the physical parameters. 
Moreover, the proposed methods significantly reduce the computational costs, making them well-suited for large-scale, efficient, and accurate simulations of solar oscillations. \\

\textbf{Structure of paper.} In \cref{sec:framework}, we repeat the abstract framework which we use to analyze the proposed discretizations of \eqref{eq:Galbrun}. In particular, we recall the concepts of weak T-coercivity, discrete approximation schemes, and weak T-compatibility which provide sufficient criteria for the convergence of approximations. In \cref{sec:prelimContinuousWF}, we introduce the continuous weak formulation of \eqref{eq:Galbrun}. 
In \cref{sec:method}, we introduce hybrid discontinuous Galerkin methods for \eqref{eq:Galbrun} and show that the discretizations are discrete approximation schemes which allows us to apply the framework introduced in \cref{sec:framework}. Afterwards, we utilize the weak T-compatibility criteria to prove the stability and convergence of the proposed discretization in \cref{sec:convergenceAnalysis}  and conclude with numerical experiments in \cref{sec:numerics}.

\section{Abstract framework}\label{sec:framework}
This section recalls the abstract tools which we will use to analyze the proposed discretizations of \eqref{eq:Galbrun}. For more details and proofs we refer to \cite{HLS22H1,H23Hdiv,Thesis_vB23}. In \cref{subsec:weakTcoercivity}, we discuss the concept of weak T-coercivity which essentially asks for an operator to be a compact perturbation of a bijective operator, cf.~\cref{def:weakTcoercivity} for a precise definition. Afterwards, we study the approximation of weakly T-coercive operators in \cref{subsec:weakTcompatibility}. In particular, we introduce the (much broader) framework of \emph{discrete approximation schemes} and discuss sufficient conditions for the convergence of discrete approximations of weakly T-coercive operators.  

\subsection{Weak T-coercivity}\label{subsec:weakTcoercivity}
For two Hilbert spaces $(X,\spl \cdot, \cdot \spr_X)$ and $(Y,\spl \cdot, \cdot \spr_Y)$, we denote by $L(X,Y)$ the space of bounded linear operators from $X$ to $Y$. In particular, we set $L(X) \coloneqq L(X,X)$. Through the Riesz-isomorphism, there exists a one-to-one relation between bounded sesquilinear forms $a(\cdot,\cdot)$ on $X \times X$ and bounded linear operators $A \in L(X)$ via $\spl Au, u' \spr_X \coloneqq a(u,u')$ for all $u, u' \in X$. Thus, we discuss the following concepts for linear operators $A \in L(X)$, but also associate them with the corresponding sesquilinear form. 

Recall that an operator $A \in L(X)$ is called \emph{coercive} if it holds that
$$\inf_{u \in X \setminus \{ 0 \}} \frac{\vert \spl Au, u \spr_X \vert}{\Vert u \Vert^2_X} > 0.$$
This condition is equivalent, cf.~\cite[Lem.~C.58]{EG_FE2} to the existence of $\xi \in \mathbb{C}$, $\vert \xi \vert = 1$, such that 
\begin{equation} \label{eq:coercivity:C}
\inf_{u \in X \setminus \{ 0 \}} \frac{\Re( \xi \spl Au, u \spr_X)}{\Vert u \Vert^2_X} > 0 .
\end{equation}
The well-known Lax-Milgram lemma states that bounded coercive operators are bijective. More generally, a bounded operator $A \in L(X)$ is bijective if and only if the adjoint operator $A^{\ast} \in L(X)$ is injective and the \emph{inf-sup condition} holds:
\begin{equation*}
   \inf_{u \in X  \setminus \{ 0 \}} \sup_{v \in X  \setminus \{ 0 \}} \frac{\vert \spl Au, v \spr_X \vert}{\Vert u \Vert_X \Vert v \Vert_X} > 0.
\end{equation*}

Equivalently, we can prove \emph{T-coercivity} \cite{Cia12}, which asks for the existence of a bijective operator $T \in L(X)$ such that $T^\ast A$ (or $AT$) is coercive. We recall the following generalization of T-coercivity. 

\begin{definition}[Weak T-coercivity]\label{def:weakTcoercivity}
   \hypertarget{weakTcoercivity}We call an operator $A \in L(X)$ \emph{weakly T-coercive} if there exists a bijective operator $T \in L(X)$ and a compact operator $K \in L(X)$ such that $A T + K$ is coercive.
\end{definition}

In other words, an operator is \weaklyTcoercive~if it is a compact perturbation of a T-coercive, operator. Thus, \weaklyTcoercive~operators are Fredholm with index zero and therefore bijective if and only if they are injective.

\subsection{Discrete approximation schemes and weak T-compatibility}\label{subsec:weakTcompatibility}
We want to study the approximation of \weaklyTcoercive~operators in a general setting. To this end, we discuss the notion of \emph{weak T-compatibility} \cite{Halla21,HLS22H1} which is build upon the framework of \emph{discrete approximation schemes} \cite{Stummel_I,Vainikko}. For a more extensive review of these concepts, we refer to \cite[Chap. 2]{Thesis_vB23}. In the following, let $X$ be a Hilbert space and $\seq{X}$ be a sequence of finite dimensional Hilbert spaces, which are not necessarily subspaces of $X$. Instead, we assume that there exists a sequence of bounded linear operators $\seq{p}$, $p_n \in L(X,X_n)$, such that $\lim_{n \rightarrow \infty} \Vert p_n u \Vert_{X_n} = \Vert u \Vert_X$ for all $u \in X$. Finally, let $A \in L(X)$ be a bounded linear operator and $\seq{A}$, $A_n \in L(X_n)$, be a sequence of bounded linear operators. 

\begin{definition}\label{def:DAS}
   In the setting from above, we define the following concepts:
   \begin{enumerate}[label=(\roman*)]
      \item \hypertarget{converge}A sequence $\seq{u}$, $u_n \in X_n$, is said to \emph{converge} to $u \in X$, if $\lim_{n \rightarrow \infty} \Vert p_n u - u_n \Vert_{X_n} = 0$. 
      \item \hypertarget{compactseq}A sequence $\seq{u}$, $u_n \in X_n$, is called \emph{compact}, if for every subsequence $\IN' \subset \IN$ there exists a subsubsequence $\IN'' \subset \IN'$ and $u \in X$ such that $(u_n)_{n \in \IN''}$ converges to $u$.
      \item \hypertarget{approximate}A sequence of operators $\seq{A}$, $A_n \in L(X_n)$, \emph{approximates} (also called \emph{asymptotic consistency}) an operator $A \in L(X)$, if $\lim_{n \rightarrow \infty} \Vert A_n p_n u - p_n A u \Vert_{X_n} = 0$ for all $u \in X$.
      \item \hypertarget{compactop}A sequence of operators $\seq{A}$, $A_n \in L(X_n)$, is called \emph{compact}, if for every bounded sequence $\seq{u}$, $u_n \in X_n$, $\Vert u_n \Vert_{X_n} \le C$, the sequence $\seq{A_n u}$ is compact.
      \item \hypertarget{stable}A sequence of operators $\seq{A}$, $A_n \in L(X_n)$, is called \emph{stable}, if there exist constants $C > 0$, $n_0 > 0$, such that $A_n$ is invertible and $\Vert A_n^{-1} \Vert_{L(X_n)} \le C$ for all $n > n_0$. 
      \item \hypertarget{regular}A sequence of operators $\seq{A}$, $A_n \in L(X_n)$, is said to be \emph{regular}, if $\Vert u_n \Vert_{X_n} \le C$ and the compactness of $\seq{A_n u}$ imply the compactness of $\seq{u}$ itself.
   \end{enumerate} 
   \hypertarget{DAS}We call the triple $(X_n,p_n,A_n)$ a \emph{discrete approximation scheme} (DAS) of $(X,A)$ if we have that $\lim_{n \rightarrow \infty} \Vert p_n u \Vert_{X_n} = \Vert u \Vert_X$ for all $u \in X$ and $A_n$ approximates $A$. 
\end{definition}

A conforming Galerkin scheme $(X_n, p_n, A_n)$, where $X_n \subset X$ fulfills an approximation property, $p_n \in L(X,X_n)$ is the orthogonal projection onto $X_n$, and $A_n \coloneqq p_n A \vert_{X_n}$, is always a \DAS~of $(X,A)$. 
Our main goal is to show that the sequence of approximations $\seq{u}$, $u_n \in X_n$, \converges~to the continuous solution $u \in X$, so we are interested in the \stability~of the sequence $\seq{A}$. The following result shows that we can focus on the \regularity~of the sequence $\seq{A}$ instead. 

\begin{lemma}[Lem.~1 \& 2 of \cite{HLS22H1}]\label[lemma]{lem:RegularStable}
   Let $A \in L(X)$ be bijective and $(X_n,p_n, A_n)$ be a \DAS~of $(X,A)$. If $\seq{A}$ is \regular, then $\seq{A}$ is \stable. Further, if $u \in X$ solves $Au=f$ and $u_n \in X_n$ are solutions to $A_n u_n = f_n$ where $\lim_{n \rightarrow \infty} \Vert p_n f - f_n \Vert_{X_n} = 0$, then $\seq{u}$ \converges~to $u$.
\end{lemma}

The following theorem gives sufficient conditions for the \regularity~of approximations of \weaklyTcoercive~operators and therefore the \stability~of the approximation. It is the key motivation for the analysis presented in \cref{sec:convergenceAnalysis}. We note that if $A \in L(X)$ is \weaklyTcoercive~, then there exists a bijective operator $B \in L(X)$ and a compact operator $K \in L(X)$ such that $AT = B + K$. 

\begin{theorem}[Thm.~3 of \cite{HLS22H1}]\label{thm:weakTcompatibility}
   Assume that there exists a constant $C>0$, sequences $\seq{A}$, $\seq{T}$, $\seq{B}$, $\seq{K}$ and $B,T \in L(X)$ such that for each $n \in \IN$ it holds that $A_n,T_n,B_n,K_n \in L(X_n)$, $\Vert T_n \Vert_{L(X_n)}$, $\Vert T_n^{-1} \Vert_{L(X_n)}$, $\Vert B_n \Vert_{L(X_n)}$, $\Vert B_n^{-1} \Vert_{L(X_n)} \le C$, $B$ is bijective, $\seq{K}$ is \compactop~and 
   \begin{equation*}
      \lim_{n \rightarrow \infty} \Vert T_n p_n u - p_n T u \Vert_{X_n} = 0, \qquad \lim_{n \rightarrow \infty} \Vert B_n p_n u - p_n B u \Vert_{X_n} = 0, \qquad \forall u \in X,
   \end{equation*} 
   \begin{equation*}
      A_n T_n = B_n + K_n. 
   \end{equation*}
   Then the sequence $\seq{A}$ is \regular. 
\end{theorem}

To summarize, \cref{thm:weakTcompatibility} yields the \stability~of a discrete approximation scheme, provided that we can transfer the \weaklyTcoercive~structure of the continuous operator $A \in L(X)$ to the discrete level in a \stable~manner. 

\section{Preliminaries \& continuous formulation}\label{sec:prelimContinuousWF}
\changed{In this section, we discuss notations and assumptions on the coefficients in \Cref{subsec:method:preliminaries} and introduce the continuous formulation of Galbrun's equation in \Cref{subsec:continuousFormulation}}.
\subsection{Preliminaries}\label{subsec:method:preliminaries}
For simplicity, we assume that $\dom \subset \mathbb{R}^d$ is a convex Lipschitz polyhedron and consider $\dom$ to be the default domain of all function spaces. Thus, we write for example $L^2 \coloneqq L^2(\dom)$. Further, we denote by $\spl \cdot, \cdot \spr$ the standard $L^2$-scalar product. For any space $X$ of scalar valued functions, we denote its vectorial version $\bm{X} \coloneqq [X]^d$ using the boldface notation. With abuse of notation, we also use the notation $\spl \cdot, \cdot \spr$ for the vector valued $\bL^2$-scalar product. Furthermore, we define the space 
\begin{equation}\label{eq:defH1nu0}
   \bH^1_{\nv 0} \coloneqq \{ \v \in \bH^1 : \v \cdot \nv = 0 \text{ on } \partial \dom \}.   
\end{equation}

Let $\omega \in \mathbb{R} \setminus \{ 0 \}$, $\Omega \in \mathbb{R}^d$, and $c_s, \rho \in W^{1,\infty}(\dom,\mathbb{R})$, $\gamma \in L^\infty(\dom,\mathbb{R})$ be measurable and bounded from above and below. For any function, we denote by $\underline{\cdot}$ and $\overline{\cdot}$ its minimal and maximal value in the domain under consideration. Thus, the boundedness assumptions on the coefficients translate to 
\begin{equation}\label{eq:CoeffsBounded}
   \underline{c_s} \le c_s(\bm{x}) \le \overline{c_s}, \quad \underline{\rho} \le \rho(\bm{x}) \le \overline{\rho}, \quad \underline{\gamma} \le \gamma(\bm{x}) \le \overline{\gamma} \quad \text{ for all } \bm{x} \in \dom, 
\end{equation}
for constants $\underline{c_s}, \overline{c_s}, \underline{\rho}, \overline{\rho}, \underline{\gamma}, \overline{\gamma} > 0$. 
Finally, let the background flow $\bflow \in \bm{W}^{1,\infty}(\dom,\mathbb{R}^d)$ be compactly supported in $\dom$.
We assume that the background flow conserves mass in the sense that $\div(\rho \bflow) = 0$. In particular, the former assumptions imply that $\div(\rho \bflow) \in L^2$ and $\bflow \cdot \nv = 0$ on $\partial \dom$.
Let the pressure and gravitational potential $p,\phi \in W^{2,\infty}(\dom,\mathbb{R})$.

Throughout the manuscript, we use the notation $A \lesssim B$, if there exists a constant $C > 0$ such that $A \le CB$, where the constant $C$ may be different at each occurrence. The constant $C$ may depend on the domain $\dom$ and the physical parameters, but not on the index $n \in \mathbb{N}$ and functions involved in $A$ and $B$. In particular, the constant $C$ is not allowed to depend on the ratio $\frac{\underline{c_s}^2 \underline{\rho}}{\overline{c_s}^2 \overline{\rho}}$. 

\subsection{Continuous weak formulation}\label{subsec:continuousFormulation}
The analysis of the continuous weak formulation of \eqref{eq:Galbrun} is presented in \cite{HH21}. We include a brief discussion here for referece throughout the article and refer to \cite{HH21} for further details. We define 
\begin{equation}
      \X \coloneqq \{ \u \in \bL^2 : \div \u \in L^2, \diffb \u \in \bL^2, \nv \cdot \u = 0 \text{ on } \partial \dom \}. 
\end{equation} 
Deviating from \cite{HH21}, we consider the associated inner product on $\X$ to be weighted:
\begin{equation*}
   \inner{\u, \u'}_{\X} \coloneqq \innercsr{\div \u, \div \u'} + \inner{\u, \u'} + \innerr{\diffb \u, \diffb \u'}.
\end{equation*}
Due to the smoothness assumptions $c_s, \rho \in W^{1,\infty}(\dom,\mathbb{R})$ and the boundedness assumptions \eqref{eq:CoeffsBounded}, the weighted inner product is equivalent to the canonical inner product on $\X$ and the proof that $\X$ is a Hilbert space follows with the same argumentation as in \cite[Lem.~2.1]{HH21}. The smoothness assumption $\bflow \in \bm{W}^{1,\infty}$ and the compactness of $\text{supp } \bflow \subset \dom$ ensure that the embedding $\bm{C}^\infty_0 \subset \X$ is dense \cite[Thm.~6]{HLS22H1}.

For $\u,\u' \in \X$, we define the following sesquilinear forms
\begin{subequations}\label{eq:a:components}
   \begin{align}
      \adiv(\u,\u') &\coloneqq \inner{\csr \div \u,\div \u'}_{L^2} + \inner{\div \u, \nabla p \cdot \u'}_{L^2} + \inner{\nabla p \cdot \u, \div \u'}_{L^2}, \\
      \aconv(\u,\u') &\coloneqq \inner{\rho \opd \u, \opd \u'}_{\bL^2}, \\
      \arem(\u,\u') &\coloneqq \inner{(\hess(p) - \rho \hess(\phi)) \u, \u'}_{\bL^2} - i \omega \inner{\gamma \rho \u, \u'}_{\bL^2}.
   \end{align}
\end{subequations}
Then, we define the sesquilinear form $a \colon \X \times \X \to \mathbb{C}$ by
\begin{equation}\label{eq:a}
      a(\u,\u') \coloneqq \adiv(\u,\u') - \aconv(\u,\u') + \arem(\u,\u')
\end{equation}
and denote by $A \in L(\X)$ the associated operator. Assuming mass conservation $\div(\rho \bflow) = 0$, the variational formulation of \eqref{eq:Galbrun} is given by
\begin{equation}\label{eq:cont:weakForm}
   \text{find } \u \in \X \text{ such that } a(\u,\u') = \spl \bff, \u' \spr \text{ for all } \u' \in \X, 
\end{equation}
cf.~\cite[Sec. 2.3]{HH21}. If the Mach number $\Vert c_s^{-1} \bflow \Vert_{\bL^\infty}$ is bounded suitably, it can be shown that the operator $A$ is \weaklyTcoercive~and injective such that problem \eqref{eq:cont:weakForm} is well-posed \cite[Thm.~3.11]{HH21}. Defining $\q \coloneqq c_s^{-2} \rho^{-1} \nabla p$ the sesquilinear form $a(\cdot,\cdot)$ can be written as 
\begin{align}
   a(\u,\u')\! = &\spl c_s^2 \rho (\div + \q \cdot) \u, (\div + \q \cdot) \u' \spr \!  - \! \spl \rho (\omega+\! i\conv \!+ \! i\angvel\times) \u, (\omega+ \! i\conv \!+ \! i\angvel\times)  \u' \spr \nonumber \\
   &+ \spl (\hess(p) - \rho \hess(\phi)-c_s^2\rho \q \otimes \q) \u, \u' \spr - i \omega \spl \gamma \rho \u, \u' \spr. \label{eq:aCont:q}
\end{align}
This representation will be useful for the discussion of the well-posedness of the continuous and the discrete problem in \cref{sec:convergenceAnalysis}. Similar to \eqref{eq:a:components}, we define 
\begin{subequations}\label{eq:aq:components}
   \begin{align}
      \adivq (\u,\u') &\coloneqq \inner{c_s^2 \rho (\div + \q \cdot) \u, (\div + \q \cdot) \u'} \\
      \aremq (\u,\u') &\coloneqq \inner{(\hess(p) - \rho \hess(\phi)-c_s^2\rho \q \otimes \q) \u, \u'} - i \omega \inner{\gamma \rho \u, \u'}, 
   \end{align}
\end{subequations}
such that $a(\u,\u') = \adivq(\u,\u') - \aconv(\u,\u') + \aremq(\u,\u')$. In particular, considering $a(\u,\u)$ for $\u \in \X \cap \ker \{ \div + \q \}$ reveals that the sesquilinear form $a(\cdot,\cdot)$ is not coercive.

\section{Hybrid discontinuous Galerkin discretizations}\label{sec:method}
In this section, we introduce the considered discretizations of \eqref{eq:Galbrun}. \changed{After introducing the discrete spaces in \cref{subsec:method:discretization}, we introduce the discrete weak formulation of Galbrun's equation in \cref{subsec:discreteWF}. In preparation for the analysis, we introduce projection operators in \cref{subsec:projectionOperators} and show in \cref{subsec:method:DAS} that the proposed discretization is indeed a \DAS, allowing us to apply the framework from \cref{sec:framework}}.

\subsection{Discrete spaces}\label{subsec:method:discretization}
Let $\seq{\mathcal{T}}$ be a sequence of shape regular, simplicial triangulations of the domain $\dom$.
Let $\mathcal{F}_n$ be the collection of all faces of the triangulation $\Tn$, and let $\partial \Tn$ be the collection of all element boundaries $\partial \tau$ of elements $\tau \in \Tn$. Notice the subtle difference between $\Fn$ and $\partial \Tn$; for instance, summing over all element boundaries counts each interior facet \emph{twice}.

For an element $\tau \in \Tn$ or a face $F \in \Fn$, we denote by $h_{\tau}$ and $h_{F}$ their diameters, respectively, and we set $h_{\partial \tau} = \max_{F \in \partial \tau} h_F$. For a unified presentation, we define a function $\h \vert_{\sigma} \coloneqq h_{\sigma}$, $\sigma \in \Sn$, where $\Sn \in \{ \Tn, \partial \Tn, \Fn \}$. Finally, let $h_n \coloneqq \max_{\tau \in \Tn} h_{\tau}$ be the maximal mesh size.

For a generic Hilbert space $\mathbb{S}$, we denote by $\mathbb{S}(\Tn)$ its broken version on $\Tn$. In particular, we denote by $\mathbb{P}^k(\Tn)$ and $\mathbb{P}^k(\Fn)$ the spaces of piecewise polynomials up to degree $k$ on $\Tn$ and $\Fn$.

On broken spaces $\mathbb{S}(\Sn)$, where $\Sn \in \{ \Tn, \partial \Tn, \Fn \}$, we use the abbreviations:
\begin{equation*}
   \spl \cdot, \cdot \spr_{\mathbb{S}(\Sn)} \coloneqq \sum_{\sigma \in \Sn} \spl \cdot, \cdot  \spr_{\mathbb{S}(\sigma)}, \qquad \Vert \cdot \Vert_{\mathbb{S}(\Sn)}^2 \coloneqq \sum_{\sigma \in \Sn} \Vert \cdot \Vert_{\mathbb{S}(\sigma)}^2.
\end{equation*}
In particular, we set $\spl \cdot, \cdot \spr_{\Sn} \coloneqq \spl \cdot, \cdot \spr_{L^2(\Sn)}$. With abuse of notation, we will also use this notation for the respective broken vector-valued scalar products, i.e.~with $L^2$ replaced by $\bL^2$. We introduce the discrete space   
\begin{equation*}
   \Xn \coloneqq \XT \times \XF,
\end{equation*} 
where $\XT$ and $\XF$ are discrete polynomial spaces defined on $\Tn$ and $\Fn$, respectively. The default choices are $\XT = [\mathbb{P}^k(\Tn)]^d$ and $\XF = [\mathbb{P}^k(\Fn)]^d$, $k \in \mathbb{N}$, yielding a fully non-conforming HDG discretization. However, the forthcoming analysis also covers different choices, for example $H(\div)$-conforming spaces, cf.~\cref{rem:HdivHDG}.
The purpose of the facet space is to enable static condensation, cf.~\cref{rem:staticcond}.

For functions $\u_n \in \Xn$, we write $\u_n = (\u_{\tau}, \u_{F})$, where $\u_{\tau} \in \XT$ is the volume and $\u_F \in \XF$ is the facet component of $\u_n$. On occasion, we make use of the projection operators onto the volume or facet components defined by $\PVol{\cdot} \colon \Xn \to \XT, (\u_{\tau},\u_{F}) \mapsto \u_{\tau}$ and $\PFac{\cdot} \colon \Xn \to \XF, (\u_{\tau},\u_{F}) \mapsto \u_{F}$.

We define the following HDG-jump operators element-wise on $\tau \in \Tn$
\begin{equation}
   \hdgjump{\u_n} \coloneqq  \u_{\tau} - \u_F, \quad \hdgjump{\u_n}_{\nv} \coloneqq \nv \cdot \hdgjump{\u_n}, \quad \hdgjump{\u_n}_{\bflow} \coloneqq (\bflow \cdot \nv) \hdgjump{\u_n}, 
\end{equation}
where we interpret $\u_{\tau}$ in a trace sense. Further, we define $\jump{\u_{\tau}}_F \coloneqq \u_{\tau} \vert_{\tau_1} - \u_{\tau} \vert_{\tau_2}$, $F \in \Fn$, $\tau_1,\tau_2 \in \Tn$, $\tau_1 \cap \tau_2 = F$, to be the usual DG-jump operator (distinguished by the absence of the underline) on $\XT$. Here, we assume a unique numbering of the aligned elements for each facet to fix the sign of the jump.

\changed{\begin{remark}[Hybridization as the enabler of static condensation]\label[remark]{rem:staticcond}
A principal purpose of hybridization is to restructure the coupling pattern of a DG formulation. In standard DG, interior (element) unknowns of neighboring elements couple directly through face terms. By introducing facet (trace) unknowns, the HDG scheme reroutes all inter-element communication through the mesh skeleton: interior unknowns couple only to unknowns on the same element and its facets, while global couplings are carried solely by the facet unknowns. This change in coupling structure enables the local elimination of interior element unknowns and leaves a global system posed only in terms of skeleton unknowns. The difference in coupling patterns is illustrated in \Cref{fig:DGvsHDG}.

\begin{figure}[!htbp]
   \centering
   \begin{subfigure}[b]{0.49\textwidth}
      \centering
      \includegraphics[width=0.8\textwidth]{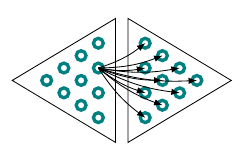}
      \subcaption{DG}
   \end{subfigure}
   \begin{subfigure}[b]{0.49\textwidth}
      \centering
      \includegraphics[width=0.8\textwidth]{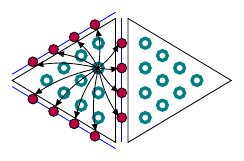}
      \subcaption{HDG}
   \end{subfigure}
   \caption{We compare the coupling of unknowns for DG and HDG methods. For illustration purposes, we depict the situation in the scalar setting with polynomial degree $k = 3$.}
   \label{fig:DGvsHDG}
\end{figure}

With the product space $\Xn = \XT \times \XF$ and the splitting $\u_n=(\u_\tau,\u_F)$ introduced above, the HDG bilinear form derived below yields a linear system of the block form
\[
\begin{bmatrix}
A_{\tau\tau} & A_{\tau F} \\
A_{F\tau} & A_{FF}
\end{bmatrix}
\begin{bmatrix}
\u_\tau \\ \u_F
\end{bmatrix}
=
\begin{bmatrix}
b_\tau \\ b_F
\end{bmatrix}.
\]
By construction, $A_{\tau\tau}$ is block-diagonal across elements (each block involves only the degrees of freedom on a single element), so that a Schur-complement strategy eliminating the element unknowns is possible and efficient. The dominant computational cost then lies in solving the reduced Schur-complement system for the skeleton unknowns; once this global system is solved, the element unknowns are recovered by independent local back-substitutions on each element. 

This two-stage procedure is referred to as \emph{static condensation}; see also \cite{CGL09,KirbySherwinCockburn2012} for further details on static condensation with HDG. It is well suited for high-performance computing because the local element solves are embarrassingly parallel and the reduced global system involves substantially fewer unknowns than the full DG system.
\end{remark}}

\begin{remark}[Alternative choices for $\XT$ and $\XF$]\label[remark]{rem:HdivHDG}
   As discussed above, we may choose $\XT$ as the Brezzi-Douglas-Marini (BDM) space \cite{BDM87} defined as 
   \begin{equation*}
      \XT \coloneqq \mathbb{BDM}_{k} \coloneqq \{ \u_{\tau} \in [\mathbb{P}^k(\Tn)]^d : \jump{\u_{\tau} \cdot \nv}_{F} = 0 \ \forall F \in \Fn \} \subset H(\div, \dom),
   \end{equation*}
   where we recall that $\jump{\cdot}_{F}$, $F \in \Fn$, denotes the usual DG-jump operator. To obtain an $H(\div)$-conforming HDG discretization \cite{LS16}, we set $\XF$ to be the space of tangential polynomials on the skeleton $\Fn$:
   \begin{equation*}
      \XF \coloneqq \mathbb{P}^{k,\text{tang}}(\Fn) =  \{ \u_F \in [\mathbb{P}^k(\Fn)]^d : P_{\nv} \u_F = 0 \}.  
   \end{equation*}
   For these choices, we redefine the HDG-jump operators as $\hdgjump{\u_n} \coloneqq P_{\nv}^{\perp} \u_{\tau} - \u_F$ such that $\hdgjump{\u_n}_{\nv} = 0$.
   To optimize the computational efficiency, we can use relaxed $H(\div)$-conforming spaces as introduced in \cite{LLS18,LLS19}. We define 
   \begin{equation*}
      \mathbb{BDM}_{k}^{-} \coloneqq \{ \u_\tau \in [\mathbb{P}^k(\Tn)]^d : \Pi_F^{k-1} \jump{\u_\tau \cdot \nv}_F = 0 \ \forall F \in \Fn \},
   \end{equation*}
   where $\Pi_F^{k-1}$ is the $L^2$-projection onto $[\mathcal{P}^{k-1}(F)]^d$. While functions in $\mathbb{BDM}_{k}^{-}$ are not normal-continuous in the highest order, this relaxation reduces the number of coupling degrees of freedoms improving the sparsity pattern of the system matrices.
   Note that in the case of (relaxed) $H(\div)$-conforming HDG discretizations the skeleton unknowns in the static condensation procedure are the facet unknowns in $\XF$ as well as the facet degrees of freedom in $\XT$ that are responsible for the normal-continuity.
\end{remark}

\subsection{Discrete weak formulation \& stabilization mechanisms}\label{subsec:discreteWF}

Let $l \in \N$. For all $\u_n \in \Xn$, we define the weighted lifting operators of degree $l$ as $\bRl \u_n \in [\mathbb{P}^l(\Tn)]^d$ and $R^l \u_n \in \mathbb{P}^l(\Tn)$ solving
\begin{subequations}
   \begin{align}
      \spl \rho \bRl \u_n, \bm{\psi}_n \spr_{\Tn} &= - \spl \rho \hdgjump{\u_n}_{\bflow}, \bm{\psi}_n \spr_{\bL^2(\partial \Tn)} \quad \text{ for all } \bm{\psi}_n \in [\mathbb{P}^l(\Tn)]^d, \label{eq:VecLift} \\
      \spl \csr R^l \u_n, \psi_n \spr_{\Tn} &= - \spl \csr \hdgjump{\u_n}_{\nv}, \psi_n \spr_{L^2(\partial \Tn)} \quad \text{ for all } \psi_n \in \mathbb{P}^l(\Tn). \label{eq:ScalLift}
   \end{align}
\end{subequations}
Due to the Cauchy-Schwarz and the discrete trace inequality, we have that
\begin{equation}\label{eq:liftingBoundedness}
   \begin{aligned}
      \Vert \rho^{1/2} \bRl\u_n \Vert_{\bL^2}^2 &\le C_{dt}^2 \Vert \rho^{1/2} \h^{-1/2} \hdgjump{\u_n}_{\bflow} \Vert^2_{\bL^2(\partial \Tn)},\\
      \Vert (\csr)^{1/2} R^l\u_n \Vert^2_{L^2} &\le C_{dt}^2  \Vert (\csr)^{1/2} \h^{-1/2} \hdgjump{\u_n}_{\nv}\Vert^2_{L^2(\partial \Tn)}.
   \end{aligned}
\end{equation}
The lifting operators allow us to define the following discrete versions of the differential operators $\ddiffb$ and $\ddiv$ on $\Xn$. For $\bu_n\in\Xn$ and $\tau \in \Tn$, we define
\begin{equation}\label{eq:discreteDiffs}
   (\ddiffb \u_n) \vert_{\tau} \coloneqq (\diffb \u_\tau)  + \bRl \u_n \quad \text{and} \quad (\ddiv \u_n) \vert_{\tau} \coloneqq (\div \u_{\tau})  + R^l\u_n. 
\end{equation}
These operators can be interpreted as distributional versions of their continuous counterpart on the broken polynomial space $\Xn$, cf.~\cite{BE08, BO09, DPE10}. To stabilize $\diffb$ on the discrete level, we replace $\diffb$ by $\ddiffb$ in the discrete sesquilinear form. This treatment stems from a Bassi-Rebay lifting technique \cite{BR97,BMMPR00} and enables us to obtain a stable method without an additional stabilization term. \changed{In particular, this technique allows us to avoid further assumptions on the magnitude of the Mach number. We note that the same technique is commonly used in hybrid high-order (HHO) methods \cite{DPEL14,DPE15,CEP20} to define reconstructions of differential operators.}

In contrast, the terms involving the divergence operator do not depend on the background flow $\bflow$. 
Thus, we use a classical symmetric interior penalty (SIP) technique for those terms and note that the discrete divergence operator is mainly introduced to obtain a unified notation with $\ddiffb$, for which the lifting technique is essential.
To be precise, we introduce the following discrete sesquilinear forms: 
\begin{subequations}\label{eq:an:components}
   \begin{align}
      \adivn(\u_n,\u_n') &\coloneqq \inner{\csr \ddiv \u_n, \ddiv \u_n'}_{\Tn} +  s_n(\u_n,\u_n') \label{eq:an:components:adivn}\\
      &\qquad + \inner{\ddiv \u_n, \nabla p \cdot \u_{\tau}'}_{\Tn} + \inner{\nabla p \cdot \u_{\tau}, \ddiv \u_n'}_{\Tn} \nonumber \\ 
      \aconvn(\u_n,\u_n') &\coloneqq \inner{\rho \dopd \u_n, \dopd \u_n'}_{\Tn}\label{eq:an:components:aconvn}\\
      \aremn(\u_n,\u_n') &\coloneqq \inner{ (\hess(p) - \rho \hess(\phi)) \u_{\tau}, \u_{\tau}'}_{\Tn} - i \omega \inner{ \gamma \rho \u_{\tau}, \u_{\tau}'}_{\Tn}
      \label{eq:an:components:aremn}
   \end{align}
\end{subequations}
where the stabilization term $s_n(\cdot,\cdot)$ is defined for $\alpha > 0$ as
\begin{equation*}
   s_n(\u_n,\u_n') \coloneqq -\spl c_s^2 \rho \alpha \h^{-1} \hdgjump{\u_n}_{\nv}, \hdgjump{\u_n'}_{\nv} \spr_{\partial \Tn} - \spl c_s^2 \rho R^l \u_n, R^l \u_n' \spr_{\Tn}.
\end{equation*}
The unusual minus in front of the first term is required to show stability in \cref{thm:weakTcompatibilityFulfilled}. In particular, the construction of the $T_n$ in \cref{subsec:CA:Tops} flips the sign in front of the normal jump, which makes the terms stemming from $s_n(\cdot,\cdot)$ positive.

Altogether, we define the discrete version of \eqref{eq:a} through
\begin{equation}\label{eq:an}
   a_n(\u_n,\u_n') \coloneqq \adivn(\u_n,\u_n') - \aconvn(\u_n,\u_n') + \aremn(\u_n,\u_n').
\end{equation}
We denote by $A_n \in L(\X_n)$ the operator associated to the sesquilinear form $a_n(\cdot,\cdot)$. The use of the discrete divergence operator $\ddiv$ in combination with the stabilization term $s_n(\cdot,\cdot)$ indeed yields a SIP formulation for $\adivn(\cdot,\cdot)$, since
\begin{align*}
   \innercsr{&\ddiv \u_n, \ddiv \u_n'}_{\Tn} \! + \! s_n(\u_n, \u_n') \\
   = &\innercsr{\div \u_{\tau}, \div \u_{\tau}}_{\Tn} - \innercsr{\hdgjump{\u_n}_{\nv}, \div \u_{\tau}'}_{\partial \Tn} \\
   -&\innercsr{\div \u_{\tau}, \hdgjump{\u_n'}_{\nv}}_{\partial \Tn} \! + \! \innercsr{\alpha \h^{-1} \hdgjump{\u_n}_{\nv}, \hdgjump{\u_n'}_{\nv}}_{\partial \Tn}.
\end{align*}
Altogether, we consider the discrete problem 
\begin{equation}\label{eq:discr:weakForm}
   \text{find } \u_n \in \X_n \text{ such that } a_n(\u_n,\u_n') = \spl \bff, \u_n' \spr \text{ for all } \u_n' \in \X_n.
\end{equation}

For functions $\u_n, \u_n' \in \Xn$, we define the following scalar product
\begin{equation}\label{def:scalarProductXn}
   \begin{aligned}
      \inner{\u_n, \u_n'}_{\Xn} \coloneqq &\innercsr{\ddiv \u_n, \ddiv \u_n'}_{\Tn} + \inner{\u_\tau, \u_\tau'}_{\Tn} \\
      &+ \innerr{\ddiffb \u_n, \ddiffb \u_n'}_{\Tn} + \innercsr{\h^{-1} \hdgjump{\u_n}_{\nv}, \hdgjump{\u_n'}_{\nv}}_{\partial \Tn},
   \end{aligned}
\end{equation}
and denote by $\Vert \cdot \Vert_{\Xn} = \sqrt{\spl \cdot, \cdot \spr_{\Xn}}$ the induced norm. The terms involving the normal jump $\hdgjump{\cdot}_{\nv}$ are added to control the terms of $a_n(\cdot,\cdot)$ arising from the SIP stabilization of the $\spl \div \cdot,\div \cdot \spr_{\Tn}$-term.

\begin{remark}[Stabilization in the setting of \cref{rem:HdivHDG}]
   If we choose an $H(\div)$-conforming HDG discretization as discussed in \cref{rem:HdivHDG}, we have that $R^l \u_n = 0$ for all $\u_n \in \Xn$, and therefore we have that $\ddiv = \div$ and $s_n(\cdot,\cdot) = 0$. For the choice of $\XT = \mathbb{BDM}_{k}^{-}$ and $\XF = \mathbb{P}^{k,\text{tang}}(\Fn)$, we still obtain $R^l \u_n = 0$ if the lifting degree satisfies $l \le k-1$. 
\end{remark}

\subsection{Projection operators}\label{subsec:projectionOperators}
In preparation, we define the following bounded interpolation operators. For $s > 1/2$, let 
\begin{itemize}
   \item $\pi_n^d : \bH^s \to [\mathbb{P}^k(\Tn)]^d \cap H(\div)$ be an $H(\div)$-conforming interpolation operator, 
   \item $\pi_n^l : L^2 \to \mathbb{P}^k(\Tn)$ be the (scalar) $L^2$-interpolation operator, 
   \item $\pi_n : \bH^s \to \XT$ be an interpolation operator into $\XT$, e.g.~the standard $\bL^2$-interpolation operator if $\XT = [\mathbb{P}^k(\Tn)]^d$ or an $H(\div)$-conforming interpolation operator if $\XT = [\mathbb{P}^k(\Tn)]^d \cap H(\div)$.
\end{itemize}
We assume that the interpolation operators $\pi_n^d$ and $\pi_n^l$ fulfill the commutation property $\div \pi_n^d = \pi_n^l \div$ and that the following estimates hold:
for all $\v \in \bH^r(\tau)$, $r \in [1,k+1]$, $m \in [0,r]$, $\tau \in \Tn$, we have that
\begin{subequations}\label{eq:pi:approximationEstimates}
   \begin{align}
      \vert \v - \tilde{\pi}_n \v \vert_{\bH^m(\tau)} &\le C_{\text{apr}} h_{\tau}^{r-m} \vert \v \vert_{\bH^r(\tau)}, \\
      \Vert \v - \tilde{\pi}_n \v \Vert_{\bL^2(\partial \tau)} &\le C_{\text{ab}} h_{\tau}^{r-1/2} \vert \v \vert_{\bH^r(\tau)},
   \end{align}
\end{subequations}
where $\tilde{\pi}_n \in \{ \pi_n^d, \pi_n \}$. For standard constructions satisfying these assumptions, we refer to \cite{EG_FE1}.

We extend the interpolation operator $\pi_n$ from the DG space $\XT$ to the HDG space $\Xn$ through 
\begin{equation*}
   \pinX : \bH^s \to \Xn, \u_{\tau} \mapsto (\pi_n \u_{\tau}, \trF (\pi_n \u_{\tau})),
\end{equation*}
where the trace operator $\trF$ of a discontinuous function is defined through averaging. For $\u_{\tau} \in \XT$, $F \in \Fn$ and $\tau_{1},\tau_{2} \in \Tn$ such that $\tau_{1} \cap \tau_{2} = F$, we set 
\begin{equation*}
   ( \trF \u_{\tau}) \vert_F \coloneqq \frac{1}{2} ( \tr \u_{\tau} \vert_{\tau_1} + \tr \u_{\tau} \vert_{\tau_2}).
\end{equation*}

We define a specific extension of the $H(\div)$-conforming interpolation operator $\pi_n^d$ to $\Xn$ for future use. For any vector-valued function $\u$, let $P_{\nv} \u \coloneqq \nv (\nv \cdot \u)$ denote the normal projection and $P_{\nv}^{\perp} \coloneqq \Id - P_{\nv}$ be the tangential projection. For $s \ge 0$, we set
\begin{equation}\label{eq:def:pinXd}
   \pinXd : \bH^{1+s} \to \XT \cap H(\div) \times \XF, \u \mapsto (\pi_n^d \u, P_{\nv}(\pi_n^d \u) + P_{\nv}^{\perp}(\pi_n^{\Fn}\u)),
\end{equation}
where $\pi_n^{\Fn}$ is the $L^2$-projection onto $\XF$. The properties of the interpolation operator $\pinXd$ are studied in \cref{lem:pinXd:properties}.

\subsection{Interpretation as DAS}\label{subsec:method:DAS}
To apply the abstract framework discussed in \cref{sec:framework}, we have to show that the proposed discretization can be interpreted as a \DAS.
As a first step, we have to define a suitable quasi projection $p_n \in L(\X,\Xn)$. 
Since the trace of functions in $\X$ is not necessarily well-defined, the evaluation of the discrete operator $\ddiffb$ and therefore the evaluation of $\spl \cdot, \cdot \spr_{\Xn}$ is not well-defined for functions in $\X$. Nevertheless, we want to define the quasi projection $p_n$ in the spirit of an orthogonal projection. 
Thus, for $\u \in \X$, we define $p_n \u \in \X_n$ to be the solution to
\begin{equation}\label{eq:def:pn}
   \inner{p_n \u, \u_n'}_{\X_n} \! = \! \innercsr{\div \u, \ddiv \u_n'}_{L^2} + \inner{\u, \u_{\tau}'}_{\bL^2} + \innerr{\diffb \u, \ddiffb \u_n'}_{\bL^2} \! \!
   \quad \forall \u_n' \in \Xn.
\end{equation}
If the function $\u$ has enough regularity to allow for a well-defined trace, for instance if $\u \in \X \cap \bH^{1}$, the discrete operators $\ddiv$ and $\ddiffb$ can be applied to the pair $\underline{\u} \coloneqq (\u, \tr \u)$ and thus $\underline{\u}$ may be plugged into $\spl \cdot, \cdot \spr_{\Xn}$, cf.~\eqref{def:scalarProductXn}. Then, \eqref{eq:def:pn} can be written as $\spl p_n \u, \u_n' \spr_{\Xn} = \spl \u, \u_n' \spr_{\Xn}$.

The Cauchy-Schwarz inequality yields that $p_n \in L(\X,\X_n)$ with $\Vert p_n \Vert_{L(\X,\X_n)} \le 1$. Furthermore, for all $\u_n' \in \X_n$ the following Galerkin orthogonality property holds:
\begin{equation}\label{eq:GalerkinOrthogonality}
   \begin{aligned}
      0 = &\innercsr{(\div \u - \ddiv p_n \u), \ddiv \u_n'}_{\Tn} + \inner{\u - \PVol{p_n \u}, \u_n'}_{\Tn} \\
      &+ \innerr{(\diffb \u - \ddiffb p_n \u), \ddiffb \u_n'}_{\Tn} - \innercsr{\h^{-1} \hdgjump{p_n \u}_{\nv}, \jump{\u_n'}_{\nv}}_{\partial \Tn}
   \end{aligned}
\end{equation}

As discussed above, the norm $\Vert \cdot \Vert_{\Xn}$ cannot be evaluated for functions in $\X$. To circumvent this issue, we introduce a distance function $d_n : \X \times \Xn \to \mathbb{R}^{+}_0$, 
\begin{equation}
   \begin{aligned}
      d_n(\u,\u_n')^2 \coloneqq &\Vert (\csr)^{1/2}(\div \u - \ddiv \u_n) \Vert_{L^2}^2 + \Vert \u - \u_{\tau} \Vert_{\bL^2}^2 \\
      &+ \Vert \rho^{1/2} (\diffb \u - \ddiffb \u_n) \Vert_{\bL^2}^2 + \jnorm{\u_h}{2}, 
   \end{aligned}
\end{equation}
where we define the following jump norm on $\Xn$:
\begin{equation}\label{eq:def:jumpNorm}
   \jnorm{\u_h}{2} \coloneqq \Vert (\csr)^{1/2} \h^{-1/2} \hdgjump{\u_n}_{\nv} \Vert_{L^2(\partial \Tn)}^2. 
\end{equation} 
If the trace is well-defined, e.g.~for $\u \in \X \cap \bH^{1}$, then we obtain $d_n(\u,\u_n) = \Vert \underline{\u} - \u_n \Vert_{\Xn}$ for $\underline{\u} = (\u, \tr \u)$. Furthermore, the distance function $d_n(\cdot,\cdot)$ fulfills the following triangle inequalities
\begin{equation}\label{eq:dn:triangle}
   d_n(\u,\u_n) \le d_n(\tilde{\u},\u_n) + \Vert \tilde{\u} - \u \Vert_{\X}, \quad d_n(\u,\u_n) \le d_n(\u,\tilde{\u}_n) + \Vert \tilde{\u}_n - \u_n \Vert_{\Xn}, 
\end{equation}
for all $\u, \tilde{\u} \in \X$ and $\u_n, \tilde{\u}_n \in \Xn$. 

To show that the triple $(\Xn,p_n,A_n)$ is a \DAS~of $(\X,A)$, we will show in \cref{lem:pnuToU} below that $\lim_{n \to \infty} \Vert p_n \u \Vert_{\Xn} = \Vert \u \Vert_{\X}$ for all $\u \in \X$ and in \cref{thm:AnToA} that $A_n$ \approximates~$A$. As a preparation, we proceed to analyze the projection operators defined in the previous section. We remark that most results follow with the same argumentation as in the $H(\div)$-conforming DG \cite{H23Hdiv} and the fully discontinuous DG \cite{Thesis_vB23} case.

\begin{lemma}\label[lemma]{lem:pinXd:properties}
   Let $\u \in \bH^{1+s}$, $0 \le s \le k$, and $\pinXd$ be defined by \eqref{eq:def:pinXd}. Then
   \begin{equation}\label{eq:pinXd:jumps}
      \hdgjump{\pinXd \u}_{\nv} = 0 \qquad \text{and} \qquad \hdgjump{\pinXd \u} = P_{\nv}^{\perp}(\pi_n^d \u - \pi_n^{\Fn}\u).
   \end{equation}
   Furthermore, there exists a constant $C > 0$ such that 
   \begin{equation}\label{eq:pinXd:boundedness}
	   \Vert \underline{\u} - \pinXd \u \Vert_{\Xn} \le C h^s_n \Vert \u \Vert_{\bH^{1+s}},
   \end{equation}
   where $\underline{\u} = (\u, \tr \u)$.
\end{lemma}

\begin{proof}
   The properties \eqref{eq:pinXd:jumps} hold by definition of $\pinXd$. Since $\hdgjump{\pinXd \u}_{\nv} = 0$, it follows that $R^l \pinXd \u = 0$ and therefore by definition of the discrete divergence operator $\ddiv$, we have that $\ddiv \pinXd \u = \div \PVol{\pinXd \u}$. The triangle inequality yields that 
   \begin{equation*}
      \Vert \underline{\u} - \pinXd \u \Vert_{\Xn} \le \Vert \u - \PVol{\pinXd \u} \Vert_{\X(\Tn)} + \Vert \rho^{1/2} \bRl (\underline{\u} -\pinXd \u) \Vert_{\Tn}. 
   \end{equation*}
   For the first term, the commutation property $\div \pi_n^d = \pi_n^l \div$ and the approximation properties of $\pi_n^d$ and $\pi_n^l$ imply that $\Vert \u - \! \PVol{\pinXd \u} \! \Vert_{\X(\Tn)} \! \lesssim \! h^s_n \Vert \u \Vert_{\bH^{1+s}}$. 
   
   For the second term, we calculate with \eqref{eq:liftingBoundedness} and $\hdgjump{\underline{\u}} = 0$ that 
   \begin{equation}\label{eq:pinXd:RBoundedness}
      \begin{aligned}
         \Vert \rho^{1/2} \bRl (\underline{\u} - \pinXd \u) \Vert_{\Tn} &\lesssim \Vert \rho^{1/2} \h^{-1/2} \hdgjump{\pinXd \u}_{\bflow} \Vert_{\bL^2(\partial \Tn)} \\
		   &\lesssim \Vert \rho^{1/2} \h^{-1/2} (\bflow \cdot \nv) P_{\nv}^{\perp}(\pi_n^d \u - \pi_n^{\Fn}\u) \Vert_{\bL^2(\partial \Tn)}\\
		   &\lesssim \Vert \rho^{1/2} \h^{-1/2} (\bflow \cdot \nv)(\pi_n^d \u - \u) \Vert_{\bL^2(\partial \Tn)} \\
         &\quad + \Vert \rho^{1/2} \h^{-1/2} (\bflow \cdot \nv)(\pi_n^{\Fn} \u - \u) \Vert_{\bL^2(\partial \Tn)}\\
         &\lesssim h^s_n \Vert \u \Vert_{\bH^{1+s}},
      \end{aligned}
   \end{equation}
   where we use the definition of $\pinXd$ in the second line, the boundedness of $P_{\nv}^{\perp}$ in the third and the approximation properties of $\pi_n^d$ and $\pi_n^{\Fn}$ in the last line. Altogether, we conclude that there exists a constant $C > 0$ such that $\Vert \underline{\u} - \pinXd \u \Vert_{\Xn} \le C \Vert \u \Vert_{\bH^{1+s}}$ which proves \eqref{eq:pinXd:boundedness}.
\end{proof}

\begin{lemma}\label[lemma]{lem:pnBoundedBypinX}
   For all $\u \in \bH^{1}_{\nv 0}$, it holds that $d_n(\u,p_n \u) \le d_n(\u,\pinX \u)$.
\end{lemma}

\begin{proof}
   Follows from an application of \eqref{eq:GalerkinOrthogonality} and the Cauchy-Schwarz inequality.
\end{proof}

\begin{lemma}\label[lemma]{lem:pinXRates}
   For all $\u \in \bH^1_{\nv 0} \cap \bH^{1+s}$, $0 < s \le k$, there exists a constant $C > 0$ independent of $h$ such that $d_n(\u,\pinX \u) \le C h_n^s \Vert \u \Vert_{\bH^{1+s}}$.
\end{lemma}

\begin{proof}
   Follows from the boundedness of the lifting operators, cf.~\eqref{eq:liftingBoundedness} and the approximation properties of $\pi_n$, cf.~\eqref{eq:pi:approximationEstimates}.
\end{proof}

\begin{lemma}\label[lemma]{lem:dnpnuToZero}
   For each $\u \in \X$, it holds that $\lim_{n \to \infty} d_n(\u,p_n \u) = 0$.
\end{lemma}

\begin{proof}
   Due to the density of $\bm{C}^\infty_0$ in $\X$ \cite[Thm.~6]{HLS22H1}, we can choose $\tilde{\u} \in \bm{C}^\infty_0$ such that $\Vert \u - \tilde{\u} \Vert_{\X} < \epsilon$ for any $\epsilon > 0$. Then, we can estimate with \eqref{eq:dn:triangle} that 
   \begin{equation*}
      d_n(\u,p_n \u) \le d_n(\tilde{\u},p_n \tilde{\u}) + \Vert \tilde{\u} - \u \Vert_{\X}  + \Vert p_n (\tilde{\u} - \u) \Vert_{\Xn} \le d_n(\tilde{\u},p_n \tilde{\u}) + 2 \epsilon.
   \end{equation*}
   Thus, the claim follows from the previous \cref{lem:pnBoundedBypinX} and \cref{lem:pinXRates}.
\end{proof}

\begin{lemma}\label[lemma]{lem:dnupinXdtoZero}
   For all $\u \in \bH^1_{\nv 0}$, it holds that 
   \begin{equation}
      \lim_{n \to \infty} d_n(\u,\pinX \u) = 0 \quad \text{and} \quad \lim_{n \to \infty} d_n(\u,\pinXd \u) = 0.
   \end{equation} 
\end{lemma}

\begin{proof}
   By construction of $\pinXd$ and the approximation properties of $\pi_n^d$, we have for $\u \in \bH^{1+s}$ that
   \begin{equation*}
	   d_n(\u,\pinXd \u) \lesssim \Vert \u - \pi_n^d \u \Vert_{\X(\Tn)} + \Vert \rho^{1/2} \bRl \pinXd \u \Vert_{\Tn} \lesssim h^s \Vert \u \Vert_{\bH^{1+s}}. 
   \end{equation*}
   \cref{lem:pinXRates} yields the same estimate for $\pinX$. 
   The proof of the claim then follows with similar density arguments as in the proof of \cref{lem:dnpnuToZero} with the additional technicality of constructing a smooth approximation that respects the boundary condition $\nv \cdot \tilde{\u} = 0$ on $\partial \dom$. For technical details, we refer to the proof of \cite[Lem.~6]{H23Hdiv}.
\end{proof}

\begin{lemma}\label[lemma]{lem:pnuToU}
   For all $\u \in \X$ it holds that $\lim_{n \rightarrow \infty} \Vert p_n \u \Vert_{\Xn} = \Vert \u \Vert_{\X}$.
\end{lemma}

\begin{proof}
   With \eqref{eq:def:pn}, we compute  
   \begin{equation*}
      \begin{aligned}
         \Vert p_n \u \Vert_{\Xn}^2 &= \inner{p_n \u, p_n \u}_{\Xn} \\
         &= \innercsr{\div \u, \ddiv p_n \u}_{L^2} + \inner{\u, \PVol{p_n \u}}_{\bL^2} + \innerr{\diffb \u, \ddiffb p_n \u}_{\bL^2} \\
         &= \Vert \u \Vert_{\X} + \innercsr{\div \u, \ddiv p_n \u  - \div \u}_{L^2} + \inner{\u, \PVol{p_n \u} - \u}_{\bL^2} \\
         &\quad + \innerr{\diffb \u, \ddiffb p_n \u - \diffb \u}_{\bL^2}.
      \end{aligned} 
   \end{equation*}
   Since $\lim_{n \to \infty} d_n(\u,p_n \u) = 0$ by \cref{lem:dnpnuToZero}, the claim follows from the estimate 
   \begin{align*}
      \vert  &\innercsr{\div \u, \ddiv p_n \u  - \div \u}_{L^2} + \inner{\u, \PVol{p_n \u}  - \u}_{\bL^2} + \innerr{\diffb \u, \ddiffb p_n \u - \diffb \u}_{\bL^2} \vert \\
      \le &\Vert \u \Vert_{\X} d_n(\u,p_n \u). 
   \end{align*}
\end{proof}

Recall that $A_n \in L(\Xn)$ and $A \in L(\X)$ are the linear operators associated with the sesquilinear forms $a_n(\cdot,\cdot)$ defined by \eqref{eq:an} and $a(\cdot,\cdot)$ defined by \eqref{eq:a}. In preparation to show that $A_n$ \approximates~$A$, we prove the following compactness result.

\begin{lemma}\label[lemma]{lem:weakConvergence}
   Let $\seq{\u} \subset \Xn$ be such that $\sup_{n \in \N} \Vert \u_n \Vert_{\Xn} < \infty$. Then there exists $\u \in \X$ and a subsequence $\N' \subset \N$ such that $\PVol{\u_n} \overset{\bL^2}{\rightharpoonup} \u$, $\csr \ddiv \u_n \overset{L^2}{\rightharpoonup} \csr \div \u$ and $\rho \ddiffb \u_n \overset{\bL^2}{\rightharpoonup} \rho \diffb \u$, $n \in \N'$.
\end{lemma}

\begin{proof}
   We modify standard arguments from the DG-case, see e.g.~\cite{BE08, BO09, DPE10} and \cite{H23Hdiv}, to the HDG setting, see also \cite{KCR23}. \\ 
   By assumption, the sequences $\PVol{\u_n}$, $\rho \ddiffb \u_n$ and $\csr \ddiv \u_n$ are bounded in $\bL^2$ and $L^2$, respectively. Thus, there exist a subsequence $\N' \subset \N$ and elements $\u, \bm{g} \in \bL^2$, $q \in L^2$ such that $\PVol{\u_n} \overset{\bL^2}{\rightharpoonup} \u$, $\csr \ddiv \u_n \overset{L^2}{\rightharpoonup} q$ and $\rho \ddiffb \u_n \overset{\bL^2}{\rightharpoonup} \bm{g}$. It remains to show that $\bm{g} = \rho \diffb \u$ and $q = \csr \div \u$. We only show the former, as the latter follows with a similar argumentation for the scalar lifting operator $R^l$, see also \cite[Lem.~6.7]{Thesis_vB23}. Let $\bm{\psi} \in \bm{C}^\infty_0$ and $\bm{\psi}_n$ be the lowest order standard $\bH^1$-interpolant of $\bm{\psi}$ on $\Tn$. Then, we compute with element-wise partial integration
   \begin{equation}\label{eq:weakConvergence:partialInt}
      \begin{aligned}
         - \inner{&\PVol{\u_n}, \rho \diffb \bm{\psi}}_{\Tn}  \\
         &= \innerr{\diffb \PVol{\u_n}, \bm{\psi}}_{\Tn} - \innerr{(\bflow \cdot \nv) \PVol{\u_n}, \bm{\psi}}_{\partial \Tn} + \underbrace{\innerr{(\bflow \cdot \nv) \PFac{\u_n}, \bm{\psi}}_{\partial \Tn}}_{=0} \\[-0.25cm]
         &= \innerr{\diffb \PVol{\u_n}, \bm{\psi}}_{\Tn} - \innerr{\hdgjump{\u_n}_{\bflow}, \bm{\psi}}_{\partial \Tn} \\
         \overset{\eqref{eq:VecLift}}&{=} \innerr{\diffb \PVol{\u_n}, \bm{\psi}}_{\Tn} - \innerr{\hdgjump{\u_n}_{\bflow}, \bm{\psi} - \bm{\psi}_n}_{\partial \Tn} + \innerr{\bRl \u_n, \bm{\psi}_n}_{\Tn} \\
         \overset{\eqref{eq:discreteDiffs}}&{=} \innerr{\diffb \PVol{\u_n}, \bm{\psi} - \bm{\psi}_n}_{\Tn} - \innerr{\hdgjump{\u_n}_{\bflow}, \bm{\psi} - \bm{\psi}_n}_{\partial \Tn} + \innerr{\ddiffb \u_n, \bm{\psi}_n}_{\Tn} \\
         &= -\inner{\PVol{\u_n}, \rho \diffb  \left( \bm{\psi}-\bm{\psi}_n  \right)}_{\Tn} + \innerr{\ddiffb \u_n, \bm{\psi}_n}_{\Tn},
      \end{aligned}
   \end{equation}
   where we recall that $\div(\rho \bflow) = 0$ by assumption. Since $\Vert \bm{\psi} - \bm{\psi}_n \Vert_{\bH^1} \lesssim h_n \Vert \bm{\psi} \Vert_{\bH^2}$ and $\Vert \u_n \Vert_{\Xn} \lesssim 1$, it follows that $ \lim_{n \to \infty} \innerr{\ddiffb \u_n, \bm{\psi}_n}_{\Tn} = \lim_{n \to \infty} - \inner{\PVol{\u_n}, \rho \diffb \bm{\psi}_n}_{\Tn}$. Thus, we obtain
   \begin{align*}
      \inner{\bm{g}, \bm{\psi}} &= \lim_{n \rightarrow \infty} \innerr{\ddiffb \u_n, \bm{\psi}} = \lim_{n \to \infty} \left( \innerr{\ddiffb \u_n, \bm{\psi} - \bm{\psi_n}}_{\Tn} + \innerr{\ddiffb \u_n, \bm{\psi}_n}_{\Tn} \right) \\
      \overset{\eqref{eq:weakConvergence:partialInt}}&{=} \lim_{n \rightarrow \infty} - \inner{\PVol{\u_n}, \rho \diffb \bm{\psi}}_{\Tn} = - \inner{\u, \rho \diffb \bm{\psi}}_{\Tn}.
   \end{align*}
   Consequently, it holds that $\bm{g} = \rho \diffb \u$ and with similar arguments $q = \csr \div \u$. 
\end{proof} 

\begin{theorem}\label{thm:AnToA}
   The operator $A_n \in L(\Xn)$ \approximates~the operator $A \in L(\X)$, i.e.~for each $\u \in \X$, it holds that $$\lim_{n \to \infty} \Vert (A_n p_n - p_n A) \u \Vert_{\Xn} = 0.$$
\end{theorem}

\begin{proof}
   As $\Vert \u_n \Vert_{\Xn} = \sup_{\u_n \in \Xn, \Vert \u_n' \Vert_{\Xn} = 1} \vert \inner{\u_n,\u_n'}_{\Xn} \vert$, we can choose for any $\u \in \X$ a sequence $\seq{\u} \subset \Xn$ such that $\Vert \u_n \Vert_{\Xn} = 1$ and 
   \begin{equation*}
      \Vert (A_n p_n - p_n A) \u \Vert_{\Xn} \le \vert \spl (A_n p_n - p_n A) \u, \u_n \spr_{\Xn} \vert + 1/n.  
   \end{equation*}
   For any subsequence $\N' \subset \N$, we can choose a subsubsequence $\N'' \subset \N'$ as in \cref{lem:weakConvergence} such that for a $\u' \in \X$ it holds that  
   \begin{align*}
      \lim_{n \in \N''} \inner{p_n &A \u, \u_n}_{\Xn}\! \overset{\eqref{eq:def:pn}}= \! \lim_{n \in \N''} \left( \! \inner{\div A \u, \csr \ddiv \u_n}_{L^2} \! + \! \inner{A \u, \u_n}_{\bL^2} \! + \! \inner{\diffb A \u, \rho \ddiffb \u_n}_{\bL^2} \! \right) \\
      &= \inner{\div A \u, \csr \div \u'}_{L^2} + \inner{A \u, \u'}_{\bL^2} + \inner{\diffb A \u, \rho \diffb \u'}_{\bL^2} = \spl A \u, \u' \spr_{\X}.
   \end{align*}
   Furthermore, recalling \eqref{eq:an:components} and \eqref{eq:an}, we have that 
   \begin{equation*}
      \inner{A_n p_n \u, \u_n}_{\Xn} = a_n(p_n \u, \u_n) = \adivn(p_n \u,\u_n) - \aconvn(p_n \u, \u_n) + \aremn(p_n \u, \u_n)
   \end{equation*}
   As discussed before, the trace of functions in $\X$ is not well-defined, but the normal trace is. Thus, we can plug $\u \in \X$ into the forms $a_n^{y}(\cdot,\cdot)$, $y \in \{ \div,r\}$, but not into $\aconvn(\cdot,\cdot)$. For $y \in \{ \div,r\}$, we recall the definitions from \eqref{eq:a:components} and calculate
   \begin{equation*}
      \lim_{n \in \N''} a_n^y(p_n \u, \u_n) = \lim_{n \in \N''}\left( a_n^y(\u,\u_n)+ a_n^y(p_n \u - \u,\u_n) \right) = a^y(\u,\u'),
   \end{equation*} 
   where the last equality follows from $\vert a_n^y(p_n \u - \u,\u_n) \vert \lesssim d_n(\u,p_n \u)$, \cref{lem:dnpnuToZero} and \cref{lem:weakConvergence}. 
   For the remaining term $\aconvn(p_n \u, \u_n)$, we calculate 
   \begin{align*}
      \aconvn(p_n \u, \u_n) &= \inner{\rho \dopd p_n \u, \dopd \u_n}_{\Tn} \\
      &= \inner{\rho \opd \u, \dopd \u_n}_{\Tn} \\
      &\qquad + \inner{\rho(\omega + i\Omega \times)(\PVol{p_n \u} - \u), \dopd \u_n}_{\Tn} \\
      &\qquad + \inner{\rho(\ddiffb p_n \u - \diffb \u), \dopd \u_n}_{\Tn}
   \end{align*}
   While the first term converges to $\aconv(\u,\u')$ due to \cref{lem:weakConvergence}, the second and third terms are again bounded by $d_n(\u,p_n \u)$ and thus converge to zero by \cref{lem:dnpnuToZero}. 
   Altogether, we obtain that 
   \begin{equation*}
      \lim_{n \in \N''} \inner{A_n p_n \u,\u_n}_{\Xn} = \adiv(\u,\u') - \aconv(\u,\u') + \arem(\u,\u') = \inner{A \u, \u'}_{\X},
   \end{equation*}
   and therefore $\lim_{n \in \N''} \Vert (A_n p_n - p_n A ) \u \Vert_{\Xn} = 0$, which completes the proof.
\end{proof}

Thus, we have shown in \cref{lem:pnuToU} that $\lim_{n \to \infty} \Vert p_n \u \Vert_{\Xn} = \Vert \u \Vert_{\X}$ and in \cref{thm:AnToA} that $A_n$ \approximates~$A$. Consequently, we conclude that the triple $(\Xn,p_n,A_n)$ is a \DAS~of $(\X,A)$ which allows us to apply the results from \cref{thm:weakTcompatibility} to analyze the discrete problem \eqref{eq:discr:weakForm}.

\section{Convergence Analysis}\label{sec:convergenceAnalysis}
The main goal of this section is to show that the sequence of approximations $\seq{A}$ is \stable~and that the sequence of discrete solutions $\seq{\u}$ \converges~to the solution of the continuous problem with optimal order. To this end, we want to use \cref{thm:weakTcompatibility} to show that the sequence $\seq{A}$ is \regular~and apply \cref{lem:RegularStable} to obtain \stability~and \convergence. In \cref{subsec:CA:Tops}, we introduce T-operators $T$ and $T_n$ on the continuous and the discrete level. Afterwards, we show in \cref{subsec:CA:TnAnalysis} that $T_n$ satisfies the assumptions from \cref{thm:weakTcompatibility}. In \cref{subsec:CA:weakTcompatibility} we show that the remaining requirements from \cref{thm:weakTcompatibility} are satisfied and in \cref{subsec:CA:convergenceEstimates} we conclude the analysis of the discrete problem \eqref{eq:discr:weakForm}. The roadmap for the analysis is shown in \cref{fig:overview}.

\begin{figure}[!htbp]
   \begin{center}
       \scalebox{0.75}{\begin{tikzpicture}[
    line/.style={draw, -latex'},
    blockb/.style={rectangle, draw, fill=gray!20, text width=14em, text centered, minimum height=2.5em,rounded corners},
    blockc/.style={rectangle, draw, fill=gray!20, text width=11em, text centered, minimum height=2.5em,rounded corners},
    block/.style={rectangle, draw, fill=gray!20, text width=8em, text centered, minimum height=5.5em,rounded corners}
]

\node[blockb, fill=none] (TnT) {
Def. of $T_n$\&$T$  
(via \eqref{eq:Xdecomp:cont}\&\eqref{eq:Xdecomp:discr})
};

\node[blockb, fill=none,below=0.25cm of TnT] (TnStable) {
$\seq{T}$ \stable~  
(\cref{lem:TnStable})
};

\node[blockc, fill=none,right=1cm of $(TnT.south east)!0.5!(TnStable.north east)$] (TnApprox) {
    $T_n$ \approximates~$T$ \\[1ex]
(\cref{lem:TnToT})
};

\node[draw=none,fill=none,left=0.25cm of $(TnT.west)!0.5!(TnStable.west)$,rotate=90, text centered, anchor=south] (Secs1) {
    \!\!\!\!\!\!Sect. \ref{subsec:CA:Tops}\& \ref{subsec:CA:TnAnalysis}\!\!\!\!\!\!
};

\node[blockb, fill=none,below=2cm of TnStable] (DbEst) {
Estimate $\ddiffb$ 
(\cref{lem:ddiffbEstimate})
};

\node[blockb, fill=none,below=0.25cm of DbEst] (ddivEst) {
Estimate $\ddiv$  
(\cref{lem:ddivEstimate})
};

\node[blockc, fill=none,right=1cm of $(DbEst.south east)!0.5!(ddivEst.north east)$] (wT) {
\!\!\!$A_n T_n = B_n + K_n$ \!\!\!\\[0.1ex]
\!\!\!$\seq{B}$ \stable, \!\!\!\\[0.1ex]
\!\!\!$\seq{B}$ \approximates~$B$, \!\!\!\\[0.1ex]
\!\!\!$\seq{K}$ \compactop, \!\!\!\\[0.1ex] ~
\!\!\!(\cref{thm:weakTcompatibilityFulfilled}) \!\!\!
};


\node[draw=none,fill=none,left=0.25cm of $(DbEst.west)!0.5!(ddivEst.west)$,rotate=90, text centered, anchor=south] (Secs2) {
Section \ref{subsec:CA:weakTcompatibility}
};

\path (TnStable) -- (DbEst) node[midway] (help) {};

\node[blockc, fill=white, below=0cm of $(wT)!0.67!(TnApprox)$] (regular) {
    $\seq{A}$ \regular\\[1ex]
    (by \cref{thm:weakTcompatibility})
};

\node[block, fill=none, right=1.5cm of regular] (stable) {
$\seq{A}$ \stable \\[1ex]    
\& $\seq{\u}$\convergesl$\u$ \\[1ex] 
(\cref{thm:convergence})
};

\node[draw=none,fill=none,above=0.25cm of $(stable.north)$,rotate=0, text centered, anchor=south] (Secs3) {
    \cref{subsec:CA:convergenceEstimates}
};    


\node[draw=none,fill=none] at (Secs1 -| stable) (ghost1){
    \phantom{\cref{subsec:CA:convergenceEstimates}}
};

\begin{scope}[on background layer]
\end{scope}


\node[block,draw=none, text width=20em, text centered, minimum height=1em,fill=none,above=0.7cm of $(TnT.north)$, text centered, anchor=south, rotate=0, inner sep=0] (DAS) {\cref{subsec:method:DAS}: $(\Xn,p_n,A_n)$ is a \DAS~of $(\X,A)$};

\begin{scope}[on background layer]
\end{scope}

\draw[line] (regular) -- (stable) node[midway,above] {\cref{lem:RegularStable}};
\draw[line] (wT) -- (regular); 
\draw[line] (TnApprox) -- (regular);
\draw[line] (TnT) -- (TnApprox); 
\draw[line] (TnStable) -- (TnApprox); 
\draw[line] (DbEst) -- (wT);
\draw[line] (ddivEst) -- (wT);

\end{tikzpicture}}
   \end{center} 
   \caption{Roadmap for the analysis of the discrete problem \eqref{eq:discr:weakForm}.} \label{fig:overview}
\end{figure}

\subsection{Construction of $T$ and $T_n$}\label{subsec:CA:Tops}
Let us recall the construction of $T$ on the continuous level as considered in \cite{H23Hdiv}. In \cref{subsec:method:preliminaries}, we introduced $\q \coloneqq c_s^{-2} \rho^{-1} \nabla p$, yielding the reformulation \eqref{eq:aCont:q} of the sesquilinear form $a(\cdot,\cdot)$; recall also the definition of $\adivq(\cdot,\cdot)$ in \eqref{eq:aq:components}. 
Intuitively, the strategy to show that sesquilinear form $a(\cdot,\cdot)$ is \weaklyTcoercive~is to construct the operator $T$ to flip the sign in front of $\aconv(\cdot,\cdot)$ for elements in $\ker\{\adivq (\cdot,\cdot) \}$. 

To this end, we want to decompose the space $\X$ into a subspace associated with the perturbed divergence operator $(\div + \q \cdot)$ and its orthogonal complement. In essence, the following construction is a generalized Helmholtz decomposition, where we want to identify the kernel of $(\div + \q \cdot)$ instead of the kernel of $\div$. In particular, if the pressure $p$ is constant, we have that $\q = 0$ and we recover the classical Helmholtz decomposition. A similar (though less involved) argument is applied in \cite[Sec.~15.1]{SBH19} to the Helmholtz equation. 

Let $X_{\ast} := \{ u \in X \colon \spl u,1 \spr = 0 \}$ for any space $X \subset L^2$, with the special case $L^2_{0} \coloneqq L^2_\ast$.
For $\u \in H_0(\div) \coloneqq \{ \u \in H(\div) : \u \cdot \nv = 0 \text{ on } \partial \dom\}$, let $v \in H^2_{\ast}$ solve
\begin{equation}\label{eq:Xdecomp:cont}
   \begin{aligned}
      \divPqM \nabla v &= \divPqM \u  \text{ in } \dom,\\
      \nv \cdot \nabla v &= 0 \text{ on } \partial \dom,
   \end{aligned}
\end{equation}
where $P_{L^2_0}$ is the $L^2$-projection onto $L^2_0$ and $M$ is a suitable finite rank operator constructed below. The operator $M$ is necessary to ensure the well-posedness of the problem, since $\divPq \nabla$ might not be bijective.
It is, however, a compact perturbation of a bijective operator and therefore Fredholm with index zero. For any Fredholm operator with index zero, there exists a finite rank operator such that the sum of both operators is bijective, cf.~\cite[Thm.~5.3]{GGK90}. 

Since we exploit the specific structure of $M$ later on, we discuss an explicit construction of $M$. We set 
\begin{equation*}
   H^2_{\ast,\text{Neu}} \coloneqq \{ \phi \in H^2_{\ast}, \nv \cdot \nabla \phi = 0 \text{ on } \partial \dom \}, \quad \mathcal{N} \coloneqq \ker \left\{ \divPq \nabla \right\} \subset  H^2_{\ast,\text{Neu}}.
\end{equation*}
Let $L := \dim \mathcal{N}$ and $(\phi_{l})_{1 \le l \le L} \subset H^2_{\ast,\text{Neu}}$ be an orthonormal basis of $\mathcal{N}$ with respect to the inner product $\inner{\div \nabla \cdot, \div \nabla \cdot}$, which is equivalent to the canonical $H^2_{\ast,\text{Neu}}$-inner product since $\nv \cdot \nabla \phi = 0$ and $\inner{\phi,1} = 0$ for all $\phi \in H^2_{\ast,\text{Neu}}$. 

Let $(\psi_{l})_{1 \le l \le L} \subset H^2_{\ast,\text{Neu}}$ be an orthonormal basis of the $L^2_0$-orthogonal complement $\left( \divPq \nabla H^2_{\ast,\text{Neu}}\right)^{\perp}$. Then, we set 
\begin{equation}\label{eq:def:M}
   M := \sum_{l = 1}^L \psi_{l} \inner{\div \cdot, \div \nabla \phi_{l}}.
\end{equation}
By construction, $M$ can be applied to $H(\div)$-functions and is compact. Thus, the operator $\divPqM$ is a Fredholm operator with index zero, and hence it is bijective if and only if it is injective. However, the construction of $M$ ensures the injectivity of the operator and therefore the well-posedness of the problem \eqref{eq:Xdecomp:cont}. Thus, for any $\u \in \X \subset H_0(\div)$ there exists a unique $v \in H^2_{\ast}$ solving \eqref{eq:Xdecomp:cont}.

Thus, we can define a unique decomposition of $\v + \w = \u \in \X$ by setting $\v \coloneqq P_{V} \u \coloneqq \nabla v$ and $\w \coloneqq \u - \v$. In particular, this construction yields that $\divPq \w = -M \w$. The construction of $P_V : H_0(\div) \to \bH^1, \u \mapsto \nabla v$ allows us to use the compactness of the embedding $\bH^1 \hookrightarrow \bL^2$. 

If $\q = 0$, then $\div \nabla = \Delta$ is bijective and $M = 0$, so that we recover the standard Helmholtz decomposition into a gradient potential and a divergence-free part.

Further, we define a bijective operator $T \in L(\X)$ through $T \u \coloneqq \v - \w$. That $a(\cdot,\cdot)$ is indeed \weaklyTcoercive~with respect to this construction will be shown in \cref{thm:weakTcompatibilityFulfilled}. 

Now, we want to construct a similar decomposition of the discrete space $\Xn$. To account for the discontinuity of the discrete functions, we have to modify the right-hand side of \eqref{eq:Xdecomp:cont}. In particular, we replace the divergence operator and the operator $M$ with corresponding discrete counterparts.

For $\u_n \in \Xn$, let $\tilde{v} \in H^2_\ast$ be the solution to 
\begin{equation}\label{eq:Xdecomp:discr}
   \begin{aligned}
	   \divPqM \nabla \tilde{v} &= \ddivpqMn \u_n  \text{ in } \dom,\\  
      \nv \cdot \nabla \tilde{v} &= 0 \text{ on } \partial \dom,
   \end{aligned}
\end{equation}
where we interpret $\pi_n^l P_{L^2_0} \q \cdot \u_n = \pi_n^l P_{L^2_0} \q \cdot \PVol{\u_n}$ and define the operator $M_n$ similarly to \eqref{eq:def:M} as
\begin{equation}
   M_n \coloneqq \sum_{l = 1}^L \psi_{l} \inner{\ddiv \cdot, \div \nabla \phi_{l}}.
\end{equation}
Since we only changed the right-hand side of the problem, the well-posedness of the problem is not affected.
Then, we define the decomposition $\u_n = \v_n + \w_n$ where we choose $\v_n$ as the $H(\div)$-conforming HDG interpolation of $\nabla \tilde{v}$. To be precise, we recall the definition \eqref{eq:def:pinXd} of the projection operator $\pinXd$ and its properties studied in \cref{lem:pinXd:properties} and define  
\begin{equation}
   \v_n \coloneqq P_{V_n} \u_n \coloneqq \pinXd \nabla \tilde{v} = (\pi_n^d \nabla \tilde{v}, P_{\nv}(\pi_n^d \nabla \tilde{v}) + P_{\nv}^{\perp}(\pi_n^{\Fn}\nabla \tilde{v})), \quad \w_n \coloneqq \u_n - \v_n. 
\end{equation}

For later use, let $\Sol : \Xn \to \bH^1$, $\u_n \mapsto \nabla \tilde{v}$ be the solution operator of \eqref{eq:Xdecomp:discr} composed with $\nabla$. Then, we have that $P_{V_n} \u_n = \pinXd \Sol \u_n$.
Finally, we define the operator $T_n \colon \Xn \to \Xn$ through 
\begin{equation}\label{eq:def:Tn}
   T_n \u_n \coloneqq \v_n - \w_n,
   \quad \text{i.e., } T_n = 2P_{V_n} - \Id_{\Xn}.
\end{equation} 
Since \eqref{eq:Xdecomp:discr} is well-posed, we have the stability estimate 
\begin{equation}\label{eq:tildevStability}
   \Vert \Sol \u_n \Vert_{\bH^1} \lesssim \Vert \ddivpqMn \u_n \Vert_{L^2} \lesssim \Vert \u_n \Vert_{\Xn} \quad \text{ for all } \u_n \in \Xn.
\end{equation}
Furthermore, since $\text{ran}(\Sol) \subset \bH^1$, we can utilize the compact embedding $\bH^1 \hookrightarrow \bL^2$ to expose the weakly T-coercive structure of $A_n$ in \cref{thm:weakTcompatibilityFulfilled} below.

\begin{remark}[Alternative decomposition of $\Xn$]
   In the construction above, the normal jump is attributed to $\w_n$ and the $T_n$-operator flips its sign. Alternatively, we can isolate the normal jump through a suitably defined lifting operator, cf.~\cite{AHLS22,Thesis_vB23}. In this case, we would decompose $\u_n = \v_n + \w_n + \z_n$ with $\v_n$ as above and $\hdgjump{\u_n}_{\nv} = \hdgjump{\z_n}_{\nv}$. This construction is more natural, because since we associate $\w_n$ with the \emph{divergence free} part of the Helmholtz decomposition, we would expect the normal jump to be zero.  
   When defining $T_n$, we now have explicit control over the sign of the normal jump. The previous construction corresponds to $T_n \u_n \coloneqq \v_n - \w_n - \z_n$, but we could also define $T_n \u_n \coloneqq \v_n - \w_n + \z_n$. Note that for the latter construction, the stabilization term $s_n(\cdot,\cdot)$ would have to be redefined to have a positive sign in front of the normal contribution and in the forthcoming analysis, the stabilization parameter $\alpha$ would have to be chosen sufficiently large to ensure that $s_n(\cdot,\cdot)$ is positive definite. To avoid further technicalities, we do not further consider this alternative decomposition. 
\end{remark}

\subsection{Analysis of $T_n$}\label{subsec:CA:TnAnalysis}
We want to show that the sequence $\seq{T}$ is bounded, \stable, and \approximates~the operator $T$. By definition of $T_n$, we have that $T_n = 2P_{V_n} - \Id_{\Xn}$ and therefore we mainly have to focus on the properties of $P_{V_n}$.

\begin{lemma}\label[lemma]{lem:PVnBounded}
   There exists a constant $C > 0$ such that $\Vert T_n \Vert_{L(\Xn)} \le C$ for all $n \in \N$.
\end{lemma}

\begin{proof}
   It suffices to show the statement for $P_{V_n}$. Since $P_{V_n} = \pinXd \Sol$ and $\Sol \u_n \in \bH^1$, we obtain with \cref{lem:pinXd:properties} and \eqref{eq:tildevStability} that  
   \begin{equation}
      \Vert P_{V_n} \u_n \Vert_{\Xn} = \Vert \pinXd \Sol \u_n \Vert_{\Xn} \lesssim \Vert \Sol \u_n \Vert_{\bH^1} \lesssim \Vert \u_n \Vert_{\Xn}. 
   \end{equation}
   Thus, there exists a constant $C > 0$ such that $\Vert P_{V_n} \Vert_{L(\Xn)} \le C$ for all $n \in \N$.
\end{proof}

For $\q = 0$, the projection $P_{V_n}$ is idempotent, that is $P_{V_n}^2 = P_{V_n}$. In the case where $\q \not = 0$, we can still show that $P_{V_n}$ is asymptotically idempotent. 
\begin{lemma}\label[lemma]{lem:PvIdempotent}
   Let $O_n \coloneqq P_{V_n} P_{V_n} - P_{V_n}$. Then $\lim_{n \to \infty} \Vert O_n \Vert_{L(\Xn)} = 0$. 
\end{lemma}

\begin{proof}
   Let $\u_n \in \Xn$ and $\w_n \coloneqq (\Id_{\Xn} - P_{V_n}) \u_n$. Since $P_{V_n} = \pinXd \Sol$, \cref{lem:pinXd:properties} implies that $\Vert (P_{V_n} P_{V_n} - P_{V_n}) \u_n \Vert_{\Xn} \lesssim \Vert \Sol \w_n \Vert_{\bH^1}$. By construction of $P_{V_n}$, we have that $\hdgjump{P_{V_n} \u_n}_{\nv} = 0$ and $\ddiv P_{V_n} \u_n = \div \PVol{P_{V_n} \u_n} = \div \pi_n^d \Sol \u_n$, and therefore $\Sol(P_{V_n} \u_n) \in \bH^1_{\nv 0}$ solves 
   \begin{equation*}
      \divPqM \Sol (P_{V_n} \u_n) \overset{\eqref{eq:Xdecomp:discr}}{=} \divpqM \pi_n^d (\Sol \u_n).
   \end{equation*}
   Thus, we calculate that 
   \begin{equation}\label{eq:tildeV2}
      \begin{aligned}
         &\hspace*{-0.35cm}\divPqM \Sol \w_n \\
         \overset{\eqref{eq:Xdecomp:discr}}&{=} \divPqM \Sol \u_n-  \divpqM \pi_n^d (\Sol \u_n) \\
         &= \div \left(\Id_{\X} - \pi_n^d\right) \Sol \u_n \! + \! (P_{L^2_0} \q \cdot - \pi_n^l P_{L^2_0} \q \cdot \pi_n^d) \Sol \u_n \! + \! M \left(\Id_{\X} - \pi_n^d\right) \Sol \u_n \\ 
         &= \div \left(\Id_{\X} - \pi_n^d\right) \Sol \u_n + \left(\Id_{L^2_0} - \pi_n^l \right) P_{L^2_0} \q \cdot \Sol \u_n \\
         &\qquad + \pi_n^l P_{L^2_0} \q \cdot \left( \Id_{\X} - \pi_n^d \right) \Sol \u_n + M \left(\Id_{\X} - \pi_n^d\right) \Sol \u_n\\
         &= ( \Id_{L^2_0} - \pi_n^l )\divPq \Sol \u_n +  \pi_n^l P_{L^2_0} \q \cdot \left( \Id_{\X} - \pi_n^d \right) \Sol \u_n \\
         &\qquad + M  \left( \Id_{\X} - \pi_n^d \right) \Sol \u_n \eqqcolon -\tilde{O}_n \u_n,
      \end{aligned}
   \end{equation}
   where we use the commutation property $\div \pi_n^d = \pi_n^l \div$ in the last step. Consequently, $\Sol \w_n$ solves \eqref{eq:Xdecomp:discr} with right-hand side $-\tilde{O}_n \u_n$ and the stability estimate \eqref{eq:tildevStability} implies that $\Vert \Sol \w_n \Vert_{\bH^1} \lesssim \Vert \tilde{O}_n \u_n \Vert_{\Xn}$. We note that the minus in front of $\tilde{O}_n$ is purely for notational convenience in later calculations. 
 
   It remains to show that $\lim_{n \to \infty} \Vert \tilde{O}_n \Vert_{L(\Xn,L^2_0)} = 0$. Due to \eqref{eq:Xdecomp:discr}, we have that 
   \begin{equation}\label{eq:divPQnablaCalcs}
      \begin{aligned}
         \divPq \Sol \u_n &= \ddivpqMn \u_n - M \Sol \u_n, \\
		 \pi_n^l \divPq \Sol \u_n &= \ddivpqMn \u_n - \pi_n^l M \Sol \u_n \\
         &\quad+ (\pi_n^l - \Id_{L^2_0}) M_n \Sol \u_n.
      \end{aligned}
   \end{equation}
   Because $\pi_n^l$ converges to $\Id_{L^2_0}$ pointwise and the operators $M\Sol$ and $M_n \Sol$ are compact, it follows that
   \begin{align*}
      &\Vert (\Id_{L^2_0} - \pi_n^l )\divPq \Sol \Vert_{L(\Xn,L^2_0)} \\
      &\lesssim \Vert (\Id_{L^2_0} - \pi_n^l) M \Sol \Vert_{L(\Xn,L^2_0)} + \Vert (\pi_n^l - \Id_{L^2_0}) M_n \Sol \Vert_{L(\Xn,L^2_0)} \overset{n \to \infty}{\to} 0. 
   \end{align*}
   Furthermore, by construction of the operator $M$, the commutation property $\div \pi_n^d = \pi_n^l \div$, and \eqref{eq:divPQnablaCalcs} we have that
   \begin{align*}
      \Vert  M &( \Id_{\X} - \pi_n^d ) \Sol \Vert_{L^2(\Xn,L^2_0)} \lesssim \Vert \div \left( \Id_{\X} - \pi_n^d \right) \Sol \Vert_{L^2(\Xn,L^2_0)} \\
      &\lesssim \Vert (\Id_{L^2_0} - \pi_n^l )\divPq \Sol \Vert_{L(\Xn,L^2_0)} + \Vert (\Id_{L^2_0} - \pi_n^l) P_{L^2_0} \q \cdot \Sol \Vert_{L(\Xn,L^2_0)},
   \end{align*}
   where the second term converges to zero as well, since $\pi_n^l$ converges to $\Id_{L^2_0}$ pointwise and $P_{L^2_0} \q \cdot \Sol$ is compact. Finally, we estimate with \eqref{eq:pi:approximationEstimates}
   \begin{equation*}
      \Vert \pi_n^l P_{L^2_0} \q \cdot ( \Id_{\X} - \pi_n^d ) \Sol \Vert_{L(\Xn,L^2_0)} \lesssim \Vert ( \Id_{\X} - \pi_n^d ) \Sol \Vert_{L(\Xn,\bL^2)} \lesssim h_n \Vert \Sol \Vert_{L(\Xn,\bH^1)} \overset{n \to \infty}{\to} 0.
   \end{equation*}
   Altogether, we obtain that $\lim_{n \to \infty} \Vert \tilde{O}_n \Vert_{L(\Xn,L^2_0)} = 0$ and thus 
   \begin{equation*}
   \Vert ( P_{V_n} P_{V_n} - P_{V_n}) \u_n \Vert_{\Xn} \lesssim \Vert \Sol \w_n \Vert_{\bH^1} \stackrel{\eqref{eq:tildevStability}}{\lesssim} \Vert \tilde{O}_n \u_n \Vert_{\Xn} \lesssim \Vert \tilde{O}_n \Vert_{L(\Xn,L^2_0)} \Vert \u_n \Vert_{\Xn} \overset{n \to \infty}{\to} 0. 
   \end{equation*}
\end{proof}

The calculations in \eqref{eq:tildeV2} yield
\begin{equation}\label{eq:divwnEquals}
   \begin{aligned}
      (\ddiv \! + \pi_n^l P_{L^2_0} \q \cdot ) \w_n  \! \! \overset{\eqref{eq:Xdecomp:discr}}{=} \! (\div \! + \! P_{L^2_0} \q \cdot + M ) S_n \w_n \! - \! M_n \w_n \!\!\overset{\eqref{eq:tildeV2}}{=} \!\! -M_n \w_n \! - \! \tilde{O}_n \u_n
   \end{aligned}
\end{equation}
In particular, if $\q = 0$, then $M_n = 0$ and $\tilde{O}_n = 0$, so $\ddiv \w_n = 0$ and we recover the standard Helmholtz decomposition on the discrete level. From \eqref{eq:divwnEquals}, we observe that even in the case where $\q \not = 0$ the discrete perturbed divergence of $\w_n$ consists of the compact operator $M_n$ and $\tilde{O}_n$ which can be absorbed in the compact part of the \weaklyTcoercive~structure.

\begin{lemma}\label[lemma]{lem:TnStable}
   There exists an index $n_0 > 0$ and $C > 0$ such that $\Vert T_n^{-1} \Vert_{L(\Xn)} \le C$ for all $n > n_0$.
\end{lemma}

\begin{proof}
   We have that $T_n T_n = 4 P_{V_n} P_{V_n} - 4 P_{V_n} + \Id_{\X} = \Id_{\X} + 4 O_n$ with $O_n$ as defined in \cref{lem:PvIdempotent}. Since $\lim_{n \to \infty} \Vert O_n \Vert_{L(\Xn)} = 0$, there exists $n_0 > 0$ such that $\Vert 4 O_n \Vert_{L(\Xn)} < 1$ for all $n > n_0$ and thus there exists $C > 0$ such that $\Vert (\Id_{\X} + 4 O_n)^{-1} \Vert_{L(\Xn)} \le C$ for all $n > n_0$. Writing
   \begin{equation*}
      (\Id_{\X} + 4 O_n)^{-1} T_n = (T_n T_n)^{-1} T_n = T_n^{-1} ,
   \end{equation*}
   we conclude that $\Vert T_n^{-1} \Vert_{L(\Xn)} \le C \Vert T_n \Vert_{L(\Xn)}$ for all $n > n_0$, which proves the claim. 
\end{proof}

The next lemma shows that $T_n$ indeed \approximates~the operator $T$.

\begin{lemma}\label[lemma]{lem:TnToT}
   For each $\u \in \X$, it holds that $\lim_{n \to \infty} \Vert (T_n p_n - p_n T) \u \Vert_{\Xn} = 0$. 
\end{lemma}

\begin{proof}
   As before, we only have to show the statement for $P_{V_n}$. First, we estimate 
   \begin{equation*}
      \Vert (P_{V_n} p_n - p_n P_{V}) \u \Vert_{\Xn} \le d_n(P_V \u, p_n P_{V} \u) + d_n(P_V \u, P_{V_n} p_n \u),
   \end{equation*}
   and note that the first term converges to zero by \cref{lem:dnpnuToZero}. By definition, we have that $P_{V_n} p_n \u = \pinXd \Sol p_n \u$ and estimate for the second term
   \begin{equation}\label{eq:PVuPVnpnuToZero}
      d_n(P_{V} \u, P_{V_n} p_n \u) \overset{\eqref{eq:dn:triangle}}{\le} d_n(P_{V} \u, \pinXd P_V \u) + \Vert \pinXd (P_{V} \u - \Sol p_n \u) \Vert_{\Xn}. 
   \end{equation}
   Since $P_V \u \in \bH^1_{\nv 0}$, the first term converges to zero by \cref{lem:dnupinXdtoZero}. For the second term, we estimate 
   \begin{align*}
      \Vert \pinXd (&P_{V} \u - \Sol p_n \u) \Vert_{\Xn} \lesssim \Vert P_{V} \u - \Sol p_n \u \Vert_{\bH^1} \\
      &\lesssim \underbrace{\Vert \div \u - \ddiv p_n \u \Vert_{\Tn}}_{\text{(I)}} + \underbrace{\Vert P_{L^2_0} \q \cdot \u - \pi_n^l P_{L^2_0} \q \cdot \PVol{p_n \u} \Vert_{\Tn}}_{\text{(II)}} +\underbrace{ \Vert M \u - M_n p_n \u \Vert_{\Tn}}_{\text{(III)}}. 
   \end{align*} 
   By definition of $d_n(\cdot,\cdot)$, we have that $\text{(I)} \lesssim d_n(\u,p_n \u)$. We further estimate 
   \begin{align*}
      \text{(II)} &\lesssim \Vert (P_{L^2_0} \q \cdot - \pi_n^l P_{L^2_0} \q \cdot) \u \Vert_{L^2} + \Vert \pi_n^l P_{L^2_0} \q \cdot (\u - \PVol{p_n \u}) \Vert_{\Tn} \\
      &\lesssim \Vert (P_{L^2_0} \q \cdot - \pi_n^l P_{L^2_0} \q \cdot) \u \Vert_{L^2} + d_n(\u,p_n \u).
   \end{align*}
   By construction of $M$ and $M_n$ it holds that $\text{(III)}  \lesssim \Vert \div \u - \ddiv p_n \u \Vert_{\Tn} \lesssim d_n(\u,p_n \u)$. 

   Altogether, we obtain 
   \begin{align}
      \Vert \pinXd (P_{V} \u - \Sol p_n \u) \Vert_{\Xn} & \lesssim \Vert (P_{V} \u - \Sol p_n \u) \Vert_{\Xn} \nonumber \\ & \lesssim d_n(\u,p_n \u) + \Vert (P_{L^2_0} \q \cdot - \pi_n^l P_{L^2_0} \q \cdot) \u \Vert_{L^2} \overset{n \to \infty}{\to} 0, \label{eq:PVSnpn:conv}
   \end{align}
   where the first term converges to zero by \cref{lem:dnupinXdtoZero} and the second term converges to zero due to the pointwise convergence of $\pi_n^l$ to $\Id_{L^2_0}$, which proves the claim.
\end{proof}

\subsection{Weak T-compatibility}\label{subsec:CA:weakTcompatibility}
In the previous section, we have defined and analyzed the operators $T$ and $T_n$. To prepare for the application of \cref{thm:weakTcompatibility} in \cref{subsec:CA:convergenceEstimates}, we have to construct a characterization $A_n T_n = B_n + K_n$ that satisfies the conditions from \cref{thm:weakTcompatibility}. Before we do so in \cref{thm:weakTcompatibilityFulfilled}, we introduce the following notation and prove some auxiliary results. 

For $\u \in \bH^1_{\nv 0}$, we define the weighted $\bH^1$-seminorm through 
\begin{equation}
   \vert \u \vert^2_{\bH^1_{\csr}} \coloneqq \Vert (\csr)^{1/2} \nabla \u \Vert_{(L^2)^{d \times d}}^2.
\end{equation}

We show that the construction of $P_{V_n} \u_n := \pinXd \nabla \tilde{v}$, where $\tilde{v} \in H^2_{\ast,\text{Neu}}$ solves \eqref{eq:Xdecomp:discr}, allows us to bound the norm of the differential operators $\ddiffb$ and $\ddiv$ suitably. These estimates are crucial to show that the conditions from \cref{thm:weakTcompatibility} are satisfied.

\begin{lemma}\label[lemma]{lem:ddiffbEstimate}
   There exists constant $\DBconst > 0$ and $n_0 \in \N$ such that for all $\u_n \in \Xn$ and $n > n_0$, it holds that
   \begin{equation} 
      \Vert \rho^{1/2} \ddiffb (P_{V_n} \u_n) \Vert^2_{\Tn} \le \hypertarget{DBconst} \DBconst \Vert c_s^{-1} \bflow \Vert_{\bL^\infty}^2 \vert \Sol \u_n \vert_{\bH^1_{c_s^2\rho}}^2.
   \end{equation}
   In particular, $\DBconst = \AddDBconst (1 + O(h_n^2))$, where $\AddDBconst > 0$ depends on $C_{\text{dt}}$, the projection $\pi_n^d$, and the constant from \eqref{eq:pinXd:RBoundedness}.
\end{lemma}

\begin{proof}
   For $\u_n \in \Xn$, we set $\v_n \coloneqq P_{V_n} \u_n$ and denote by $\tilde{v}$ the solution to \eqref{eq:Xdecomp:discr} such that $\Sol \u_n = \nabla \tilde{v}$. For an element $\tau \in \Tn$ and $\eta \in W^{1,\infty}$ we use the notation $\eta_{\tau} = \eta \vert_{\tau}$ and estimate
   \begin{equation}\label{eq:ddiffbEstimate:proof1}
      \Vert \rho^{1/2} \ddiffb \v_n \Vert_{\bL^2(\tau)} \le \Vert \rho^{1/2} \diffb \PVol{\v_n} \Vert_{\bL^2(\tau)} + \Vert \rho^{1/2} \bRl \v_n \Vert_{\bL^2(\tau)}.
   \end{equation}
   By definition, $\PVol{\v_n} = \pi_n^d \nabla \tilde{v}$, and thus we can estimate the first term by  
   \begin{align*}
      \Vert \rho^{1/2} \diffb (\pi_n^d \nabla \tilde{v}) \Vert^2_{\bL^2(\tau)} &\le \Vert c_s^{-1} \bflow \Vert^2_{\bL^\infty} \overline{c_{s_\tau}}^2 \overline{\rho_{\tau}} \vert \pi_n^d \nabla \tilde{v} \vert^2_{\bH^1(\tau)} \le \vert \pi_n^d \vert^2_{L(\bH^1(\tau))} \Vert c_s^{-1} \bflow \Vert^2_{\bL^\infty} \overline{c_{s_\tau}}^2 \overline{\rho_{\tau}} \vert \nabla \tilde{v} \vert^2_{\bH^1(\tau)} \\
      &\le \vert \pi_n^d \vert^2_{L(\bH^1(\tau))} \Vert c_s^{-1} \bflow \Vert^2_{\bL^\infty} \overline{c_{s_\tau}}^2 \overline{\rho_{\tau}}   (\underline{c^2_{s_{\tau}}} \underline{\rho_{\tau}})^{-1} \vert \nabla \tilde{v} \vert^2_{\bH^1_{c_s^2 \rho}(\tau)}\\
      &\le \vert \pi_n^d \vert^2_{L(\bH^1(\tau))} \Vert c_s^{-1} \bflow \Vert^2_{\bL^\infty}  \!\! \underbrace{\left( 1 + h_n^2 \frac{1}{\underline{c_{s_\tau}}^2 \underline{\rho_{\tau}}} (C^L_{c_s \rho^{1/2}})^2 \right)}_{\eqqcolon C_n^L(\tau)}\! \vert \nabla \tilde{v} \vert^2_{\bH^1_{c_s^2 \rho}(\tau)}\!\!, 
   \end{align*}
   where we use the Lipschitz continuity of $c_s \rho^{1/2} \in W^{1,\infty}$ with constant $C^L_{c_s \rho^{1/2}}$ in the last step. Since $h_n \to 0$ for $n \to \infty$, there exists $n_0 \in \N$ such that $C_n^L(\tau) \le 2$ for all $n > n_0$ and all $\tau \in \Tn$. 
   For the second term in \eqref{eq:ddiffbEstimate:proof1}, we use \eqref{eq:liftingBoundedness} and same argumentation as in the proof of \cref{lem:pinXd:properties} to obtain
   \begin{align*}
      \Vert \rho^{1/2} \bRl \v_n \Vert^2_{\bL^2(\tau)} &\le C_{\text{dt}}^2 \overline{\rho_{\tau}} \Vert \h^{-1/2} \hdgjump{\v_n}_{\bflow} \Vert^2_{\bL^2(\partial \tau)} \le \Vert c_s^{-1} \bflow \Vert^2_{\bL^\infty} C_{\text{dt}}^2 \overline{c_{s_\tau}}^2 \overline{\rho_{\tau}} \Vert \h^{-1/2} \hdgjump{\v_n} \Vert^2_{\bL^2(\partial \tau)} \\
	  \overset{\eqref{eq:pinXd:RBoundedness}}&{\lesssim} \Vert c_s^{-1} \bflow \Vert^2_{\bL^\infty} C_{\text{dt}}^2 C_n^L(\tau) \vert \nabla \tilde{v} \vert_{\bH^1_{c_s^2 \rho}(\tau)},
   \end{align*}
   Inserting both estimates into \eqref{eq:ddiffbEstimate:proof1} and summing over all elements $\tau \in \Tn$ yields
   \begin{equation*}
      \Vert \rho^{1/2} \ddiffb \v_n \Vert^2_{\Tn} \le \DBconst \Vert c_s^{-1} \bflow \Vert^2_{\bL^\infty} \vert \nabla \tilde{v} \vert_{\bH^1_{c_s^2\rho}}^2,
   \end{equation*}
   where $\DBconst \coloneqq \AddDBconst (1 + O(h_n^2))$ with $\AddDBconst > 0$ depending on $C_{\text{dt}}$, the projection $\pi_n$, and the constant from \eqref{eq:pinXd:RBoundedness}. To ensure that $\DBconst$, we estimate $\vert \pi_n^d \vert^2_{L(\bH^1(\tau))} \le \sup_{n > n_0} \sup_{\tau \in \Tn} \vert \pi_n^d \vert^2_{L(\bH^1(\tau))}$ and use similar arguments for the estimates coming from \eqref{eq:pinXd:RBoundedness}. Since $\v_n = P_{V_n} \u_n$ and $\Sol \u_n = \nabla \tilde{v}$, the proof is finished. 
\end{proof}

Let us stress that the constant $\DBconst > 0$ only depends locally on the coefficients and their Lipschitz constant, and is independent of the ratio 
$(\underline{c_s}^2 \underline{\rho})/(\overline{c_s}^2 \overline{\rho})$. In particular, the quadratic factor $h_n^2$ mitigates the effects of large Lipschitz constants and asymptotically, the constant tends to $\AddDBconst$ with order $h_n^2$. 

In the following lemma, we show that the decomposition \eqref{eq:Xdecomp:discr} allows us to bound the norm of the discrete divergence operator from below. 

\begin{lemma}\label[lemma]{lem:ddivEstimate}
   For any $\delta \in (0,1)$, there exist $C_{\delta} > 0$, so that
   \begin{equation}
      \begin{aligned}
         \!\!\Vert c_s \rho^{1/2} \!\ddiv \! P_{V_n} \!\u_n \Vert^2_{\Tn} \! \ge \!\left(\! (1\!-\!\delta)^2 \vert \Sol \u_n \vert^2_{\bH^1_{c_s^2 \rho}} \! \! \! \! \! - C_{\delta} \Vert \Sol \u_n \Vert^2_{\bL^2} \!\right) \! + \! \spl \check{O}_n \u_n, \!\u_n \spr_{\Xn}\!,\!\!\!
      \end{aligned}
   \end{equation}
   for all $\bu_n\in\Xn$ and $\lim_{n \to \infty} \Vert \check{O}_n \Vert_{L(\Xn)} = 0$.
\end{lemma}

\begin{proof}
   For $\u_n \in \Xn$, let $\tilde{v} \in H^2_{\ast,\text{Neu}}$ be the solution to \eqref{eq:Xdecomp:discr} such that $\Sol \u_n = \nabla \tilde{v}$. Due to the properties of $\pinXd$ discussed in  \cref{lem:pinXd:properties}, we obtain for $P_{V_n} = \pinXd \Sol$ that $\ddiv P_{V_n} \u_n = \div \PVol{P_{V_n} \u_n} = \div \pi_n^d \nabla \tilde{v}$. 

   Using the commutation property $\pi_n^l \div = \div \pi_n^d$, we calculate 
   \begin{align*}
      \div \pi_n^d \nabla \tilde{v} &= \pi_n^l \div \nabla \tilde{v} \\
      \overset{\eqref{eq:Xdecomp:discr}}&{=} \pi_n^l \left( - (P_{L^2_0} \q \cdot + M) \nabla \tilde{v} + \ddivpqMn \u_n \right) \\
      &= - (P_{L^2_0} \q \cdot + M) \nabla \tilde{v} + \ddivpqMn \u_n \\
      &\qquad + (\Id - \pi_n^l)(P_{L^2_0} \q \cdot + M) \nabla \tilde{v} + (\pi_n^l - \Id)M_n \u_n \\
      \overset{\eqref{eq:Xdecomp:discr}}&{=} \div \nabla \tilde{v} + (\Id - \pi_n^l)(P_{L^2_0} \q \cdot + M) \Sol \u_n + (\pi_n^l - \Id) M_n \u_n \\
      &=: \div \nabla \tilde{v} + \hat{O}_n \u_n.
   \end{align*}
   Thus, we have that 
   \begin{equation}\label{eq:divPvnEquals}
      \begin{aligned}
         \Vert c_s \rho^{1/2} \ddiv P_{V_n} \u_n \Vert^2_{\Tn} &= \spl c_s^2 \rho \div \pi_n^d \nabla \tilde{v}, \div \pi_n^d \nabla \tilde{v} \spr \\
         &= \spl c_s^2 \rho \div \nabla \tilde{v}, \div \nabla \tilde{v} \spr + \spl \check{O}_n \u_n, \u_n \spr_{\Xn}, 
      \end{aligned}
   \end{equation} 
   where we define the operator $\check{O}_n$ through 
   \begin{equation*}
      \begin{aligned}
         \spl \check{O}_n \u_n, \u_n' \spr_{\Xn} \coloneqq &\spl c_s^2 \rho \div \pi_n^d \nabla \tilde{v}, \hat{O}_n \u_n' \spr + \spl c_s^2 \rho \hat{O}_n \u_n, \div (\pi_n^d \nabla \tilde{v})' \spr \\
         &+ \spl c_s^2 \rho \hat{O}_n \u_n, \hat{O}_n \u_n' \spr.
      \end{aligned}
   \end{equation*}
   With similar arguments as in \cref{lem:PvIdempotent}, we can show $\lim_{n \to \infty} \Vert \hat{O}_n \Vert_{L(\Xn,L^2_0)} = 0$ and thus $\lim_{n \to \infty} \Vert \check{O}_n \Vert_{L(\Xn)} = 0$.

   In the following, we use similar techniques as in the proof of \cite[Thm.~3.5]{HH21} to show that the first term can be estimated suitably by a weighted $\bH^1$-seminorm. By assumption, $\dom$ is a convex Lipschitz polyhedron and thus, we can apply \cite[Thm.~3.1.1.2]{G11} to estimate for any $\v \in \bH^1_{\nv 0}$
   \begin{equation}\label{eq:GrisvardEstimate}
      \Vert \div (\v) \Vert^2_{L^2} \! = \! \! \! \sum_{i,j = 1}^d \spl \partial_{x_j} \v_i, \partial_{x_i} \v_j \spr_{\bL^2} \! - \! \! \int_{\partial \dom} \! \! \! \! \mathcal{B}(P_{\nv}^\perp \v, P_{\nv}^\perp \v) \mathrm{d}s \ge \! \! \! \sum_{i,j = 1}^d \spl \partial_{x_j} \v_i, \partial_{x_i} \v_j \spr_{\bL^2}.
   \end{equation}
   where $\mathcal{B}(\bm{\tau},\bm{\tau}') := -\partial_{\bm{\tau}} \nv \cdot \overline{\bm{\tau}'}$ is the second fundamental quadratic form of $\partial \dom$ applied to tangential vector fields $\bm{\tau},\bm{\tau}'$ and $\partial_{\bm{\tau}}$ is the tangential derivative. The last estimate follows since the form $\mathcal{B}$ is non-positive for convex domains \cite[Sec.~3.1.1]{G11}.
   
   For any $\eta \in W^{1,\infty}(\dom)$, the product rule gives that $\eta \div \v = \div(\eta \v) - \v \cdot \nabla \eta$ and thus we estimate with the weighted Young's inequality for any $\delta \in (0,1)$ that 
   \begin{equation}\label{eq:etaDivVFirstEst}
      \Vert \eta \div \v \Vert^2_{L^2} \ge (1-\delta) \Vert \div (\eta \v) \Vert^2_{\bL^2} + (1-\frac{1}{\delta}) \Vert \v \cdot \nabla \eta \Vert_{\bL^2}^2. 
   \end{equation}
   Since $\eta \v \in \bH^1_{\nv 0}$, we can apply \eqref{eq:GrisvardEstimate} to the first term to obtain 
   \begin{equation*}
      \begin{aligned}
         \Vert \div (\eta \v) \Vert^2_{L^2} &\ge \sum_{i,j=1}^d \spl \partial_{x_j} (\eta \v_i), \partial_{x_i} (\eta \v_j) \spr_{\bL^2} \\
         &\ge \sum_{i,j=1}^d \big( \spl \eta \partial_{x_j} \v_{i}, \eta \partial_{x_i} \v_{j} \spr_{\bL^2} + \spl \v_i \partial_{x_j} \eta, \v_j \partial_{x_i} \eta \spr_{\bL^2} + \spl \v_i \partial_{x_j} \eta, \eta \partial_{x_i} \v_j \spr_{\bL^2} \\
         &\qquad+ \spl \eta \partial_{x_j} \v_i, \v_j \partial_{x_i} \eta \spr_{\bL^2} \big). 
      \end{aligned} 
   \end{equation*} 
   Applying this estimate to $\v = \nabla \tilde{v}$, we notice that $\partial_{x_j} \v_i = \partial_{x_i} \v_j$ and consequently
   \begin{equation*}
      \begin{aligned}
         \Vert \div (\eta \v) \Vert^2_{L^2} &\ge \Vert \eta \nabla \v \Vert^2_{\bL^2} + \sum_{i,j=1}^d \big( \spl \v_i \partial_{x_j} \eta, \v_j \partial_{x_i} \eta \spr_{\bL^2} + \spl \v_i \partial_{x_j} \eta, \eta \partial_{x_i} \v_j \spr_{\bL^2} \\
         &\qquad+ \spl \eta \partial_{x_j} \v_i, \v_j \partial_{x_i} \eta \spr_{\bL^2} \big) \\
         &\ge \Vert \eta \nabla \v \Vert^2_{\bL^2} - C \left( \Vert \v \cdot \nabla \eta \Vert^2_{\bL^2} + \Vert \eta \nabla \v \Vert_{\bL^2} \Vert \v \cdot \nabla \eta \Vert_{\bL^2} \right) \\
         &\ge  (1-\delta) \Vert \eta \nabla \v \Vert^2_{\bL^2} - (C+\frac{1}{4\delta}) \Vert \v \cdot \nabla \eta \Vert^2_{\bL^2},
      \end{aligned} 
   \end{equation*}
   where we use the Cauchy-Schwarz and the weighted Young's inequality in the last two lines. Inserting this into the estimate \eqref{eq:etaDivVFirstEst} yields 
   \begin{equation*}
      \Vert \eta \div \v \Vert^2_{L^2} \ge  (1-\delta)^2 \Vert \eta \nabla \v \Vert^2_{\bL^2} + \underbrace{\left( C(\delta-1) - \frac{5}{4\delta} + \frac{5}{4} \right)}_{=: - \tilde{C}_{\delta}} \Vert \v \cdot \nabla \eta \Vert^2_{\bL^2},
   \end{equation*}
   where we note that $\tilde{C}_{\delta} \ge 0$ for all $\delta \in (0,1)$. Since $\eta \in W^{1,\infty}$ is Lipschitz continuous, $\nabla \eta$ is bounded such that 
   $\Vert \v \cdot \nabla \eta \Vert^2_{\bL^2} \le C_{\nabla \eta} \Vert \v \Vert_{\bL^2}$. Inserting $\v = \nabla \tilde{v} = \Sol \u_n$ and $\eta = c_s \rho^{1/2}$ we obtain together with \eqref{eq:divPvnEquals} that 
   \begin{equation}
      \! \! \Vert c_s \rho^{1/2} \! \ddiv P_{V_n} \! \u_n \Vert^2_{\Tn} \ge (1\!-\!\delta)^2 \vert \Sol \u_n \vert^2_{\bH^1_{c_s^2 \rho}} \!  \! \! \! \!-\! C_{\delta} \Vert \Sol \u_n \Vert^2_{\bL^2} + \spl \check{O}_n \u_n, \!\u_n \spr_{\Xn}\!,
   \end{equation}
   where $C_{\delta} \coloneqq \tilde{C}_{\delta} C_{\nabla \eta}$. Thus, the claim is proven.
\end{proof}
 
To prove our main result, we have to assume that the Mach number $\Vert c_s^{-1} \bflow \Vert_{\bL^\infty}$ of the background flow is bounded suitably. Following the presentation in \cite{HH21,HLS22H1,H23Hdiv,Thesis_vB23}, we define the matrix $\matii \coloneqq - \rho^{-1} \hess(p) + \hess(\phi)$ and denote by $\lambda_{-}(\matii) \in L^\infty$ its smallest eigenvalue\footnote{Note that $\matii$ is symmetric.}. Then, we set for $\omega \not = 0$
\begin{equation} \label{eq:def:theta}
   C_{\matii} \coloneqq \max \left\{ 0, \sup_{x \in \dom} \frac{\lambda_{-}(\matii(x))}{\gamma(x)} \right\}, \qquad \theta \coloneqq \arctan(C_{\matii}/\vert \omega \vert) \in [0,\pi/2).
\end{equation}
\changed{Note that $\theta$ depends on the physical parameters and for large frequencies $\omega$ or damping parameters $\gamma$, it is close to zero.} On the technical site, its definition allows us to estimate 
\begin{equation}\label{eq:thetaEstimate}
    \spl \rho \matii \u_{\tau}, \u_{\tau} \spr \ge - \vert \omega \vert \tan(\theta) \Vert (\gamma \rho)^{1/2} \u_{\tau} \Vert_{\bL^2} \qquad \forall \u_{\tau} \in \XT,
\end{equation}
which we will use in the proof of \cref{thm:weakTcompatibilityFulfilled}. In preparation of \cref{thm:weakTcompatibilityFulfilled}, we make the following assumption. 

\begin{assumption}\label{assumption:MachNumber}
The background flow $\bflow \in \bW^{1,\infty}$ satisfies 
   \begin{equation}\label{eq:assumption:MachNumber}
      \Vert c_s^{-1} \bflow \Vert^2_{\bL^\infty} < \frac{1}{\DBconst (1 + C_{\matii}^2 / \vert \omega \vert^2)} = (\DBconst (1 + \tan^2(\theta)))^{-1},
   \end{equation}
   where $\DBconst = 1+O(h_n^2) > 0$ is the constant appearing in \cref{lem:ddiffbEstimate}. 
\end{assumption}

\begin{remark}\label[remark]{rem:MNsmall}
   The strict inequality in \cref{assumption:MachNumber} implies that the inequality holds even for a slightly smaller r.h.s., i.e. there is $\delta_0 \in (0,1)$ so that
   \begin{equation*}
      \Vert c_s^{-1} \bflow \Vert^2_{\bL^\infty} \!<\! \frac{(1-\delta_0)^2}{\DBconst (1\!+\!(1\!+\!\delta_0)^2 \tan^2(\theta) / \vert \omega \vert^2)} \Leftrightarrow \!
      \frac{(1-\delta_0)^2}{\DBconst \Vert c_s^{-1} \bflow \Vert^2_{\bL^\infty}} \!>\! 1 + (1+\delta_0)^2 \!\tan^2\!(\theta)
   \end{equation*}
   where we made use of the definition of $\theta$. Similarly, we can bound $\tan^2(\theta)$ from below by $\kappa^{-1} \tan^2(\theta + \tau)$ with $\kappa>1$ close to $1$ for $\tau>0$ sufficiently small. To be precise, there is $\tau_0 \in (0,\pi/2-\theta)$ and $\epsilon_0 \in (0,1/2)$ so that for all $\tau \in (0, \tau_0)$ and $\epsilon \in (0, \epsilon_0)$ we have that 
   \begin{equation}\label{eq:assumption:MachNumber:3}
      \frac{ (1-\delta_0)^2}{\DBconst \Vert c_s^{-1} \bflow \Vert^2_{\bL^\infty}} -1 >  (1+\delta_0)^2 \tan^2(\theta) > 
      \tan^2(\theta + \tau) (1-\epsilon)^{-1}(1-2\epsilon)^{-1}.
   \end{equation}
   Multiplying with $(1-\epsilon)$ and rearranging the terms, we obtain that for $\epsilon \in (0,\epsilon_0)$, $\delta \in (0,\delta_0)$, and $\tau \in (0,\tau_0)$, we have that 
   \begin{equation}\label{eq:MonsterConstant}
      \! \! C_{\theta,\tau,\epsilon,\delta} \!\coloneqq\! (1 \!- \epsilon)(1 \!- \delta)^2 \! \! - \! \DBconst \Vert c_s^{-1} \bflow \Vert^2_{\bL^\infty} \! \! \left( 1\! +\! \tan^2(\theta + \tau)(1 \! -2\epsilon)^{-1}\!\! \! -\epsilon \right) \! >\! 0.
   \end{equation}
   This constant appears in the proof of \cref{thm:weakTcompatibilityFulfilled} below and its positivity is essential to obtain \stability. In \cite{HLS22H1}, where an $\bH^1$-conforming discretization of \eqref{eq:cont:weakForm} was analyzed, the following smallness assumption was assumed:
   \begin{equation}\label{eq:MnAssump:H1}
      \Vert c_s^{-1} \bflow \Vert^2_{\bL^\infty} < \beta_h^2 \frac{\underline{c_s}^2 \underline{\rho}}{\overline{c_s}^2 \overline{\rho}} (1 + \tan^2(\theta))^{-1}. 
   \end{equation}
   Here $\beta_h$ is the discrete stability constant of the divergence operator. 
   This assumption is much more restrictive than \eqref{eq:assumption:MachNumber} because it depends on the ratio $\underline{c_s}^2 \underline{\rho} / \overline{c_s}^2 \overline{\rho}$, whereas the constant $\DBconst > 0$ is independent of this ratio. In particular, the constant $\DBconst$ tends to $\AddDBconst$ asymptotically, and thus \cref{assumption:MachNumber} is close to the boundedness assumption from the continuous analysis, cf.~\cite[Thm.~3.11]{HH21}.
   
   To avoid the dependence on the ratio, we used the weighted $\bH^1$-seminorm $\vert \cdot \vert_{\bH^1_{\csr}}$ in \cref{lem:ddiffbEstimate} and \cref{lem:ddivEstimate}, and the constants $\DBconst$ can be interpreted as the continuity of the operator $\ddiffb$ with respect to $\vert \cdot \vert_{\bH^1_{\csr}}$. We compare both assumptions numerically in \cref{subsec:numex:mach}.
\end{remark}

\begin{theorem}\label{thm:weakTcompatibilityFulfilled}
   Assume that \cref{assumption:MachNumber} is satisfied. Then, there exist sequences $\seq{B}, \seq{K}$, $B_n \in L(\Xn)$, $K_n \in L(\Xn)$, $n \in \N$ such that $A_n T_n = B_n + K_n$ with $\seq{B}$ being uniformly bounded. $\seq{B}$ is \stable, $\seq{K}$ is \compactop~and there exists a bijective operator $B \in L(\Xn)$ such that $B_n$ \approximates~$B$. 
\end{theorem}

\begin{proof}
   We split the proof of the statement into four steps. In the first step, we define the sequences $\seq{B}$ and $\seq{K}$ and argue that indeed $A_n T_n = B_n + K_n$. Afterwards, we show in the second and third step that the sequence $\seq{B}$ is \stable. In the last step, we show that there exists a bijective operator $B \in L(\X)$ and a compact operator $K \in L(\X)$ such that $AT = B + K$ and $B_n$ \approximates~$B$. In the following, we denote $\colorv{\v_n} \coloneqq P_{V_n} \u_n$, $\colorw{\w_n} \coloneqq \u_n - \colorv{\v_n}$ for an element $\u_n \in \Xn$ and $\colorv{\v_n'}, \colorw{\w_n'}$ defined analogously for an element $\u_n' \in \Xn$. \\[0.8ex]
   \emph{\underline{Step 1:} Definition of $B_n$ and $K_n$.} We want to define $B_n, K_n \in L(\Xn)$ such that $A_n T_n = B_n + K_n$, where $\seq{B}$ is uniformly coercive and $\seq{K}$ is \compactseq. In particular, we define $B_n = \BnI + \BnII$, where $\BnI$ is constructed to yield the essential control of the $\Vert \cdot \Vert_{\Xn}$-norm and $\BnII$ contains the remaining terms, which we will estimate in \emph{Step 3}. We add compact terms $\KnI$ and $\KnII$ to both operators which are subtracted again through $K_n = -\KnI - \KnII$.   
   
   We will consider a splitting $a_n(\cdot,\cdot) = a_n^{(1)}(\cdot,\cdot) + a_n^{(2)}(\cdot,\cdot)$ so that terms that directly help for ($T_n$-)coercivity will be collected in $a_n^{(1)}(\cdot,\cdot)$ and the remainder will be collected in $a_n^{(2)}(\cdot,\cdot)$.
   By construction, $T_n$ swaps the sign in front of \colorw{$\w_n$} such that $\spl A_n T_n \u_n, \u_n' \spr = a_n(\colorv{\v_n} - \colorw{\w_n}, \colorv{\v_n'} + \colorw{\w_n'})$. 
   
   We recall $\q \coloneqq c_s^{-2} \rho^{-1} \nabla p$, the associated rewriting of the sesquilinear form \eqref{eq:aCont:q}, and note that we can use a corresponding split on the discrete level: $a_n(\cdot,\cdot) = a_n^{(\div + \q \cdot)}(\cdot,\cdot) +s_n(\cdot,\cdot) - a_n^{\diffb}(\cdot,\cdot) + a_n^{(r - \q \cdot)}(\cdot,\cdot)$. 
   
   We have that \colorv{$\ddiv \v_n = \div \v_{\tau}$} and thus 
   \begin{equation}\label{eq:adiv:1}
      \begin{aligned}
         \hspace*{-2cm}a_n^{(\div + \q \cdot)}&(T_n \u_n,\u_n') =  
         \inner{\csr (\ddiv + \q \cdot) 
         \u_n,
         (\ddiv + \q \cdot) 
         \u_n'
         } 
         = 
         \textstyle{\sum_{J=\colorvv{I},\colorvw{II},\colorwv{III},\colorww{IV}}} (\ref{eq:adiv:1}_{\text{\tiny J\!}})
         \\
         = & \color{vv} \inner{\csr (\div + \q \cdot) \v_{\tau}, (\div + \q \cdot) \v_{\tau}'} \color{vw} + \inner{\csr (\div + \q \cdot) \v_{\tau}, (\ddiv + \q \cdot) \w_n'} \\
         & \color{wv} -\inner{\csr (\ddiv + \q \cdot) \w_n, (\div + \q \cdot) \v_{\tau}'} \color{ww}- \inner{\csr (\ddiv + \q \cdot) \w_n, (\ddiv + \q \cdot) \w_n'} 
      \end{aligned}
   \end{equation}
   The $\div$-parts of the first term in \colorvv{$(\ref{eq:adiv:1}_{\text{\tiny I\!}})$} will directly be used in $a_n^{(1)}(\cdot,\cdot)$ to control the divergence, below. 
   To rewrite \colorww{$(\ref{eq:adiv:1}_{\text{\tiny IV\!}})$}, we want to use that $(\ddiv + \pi_n^l P_{L^2_0} \q \cdot ) \w_n \stackrel{\eqref{eq:divwnEquals}}{=} -M_n \w_n - \tilde{O}_n \u_n$ for both arguments, so we shift in the terms involving the projection $\pi_n^l P_{L^2_0}$.   
   Then, we obtain\footnotemark \ that
   \begin{align}
      \colorww{(\ref{eq:adiv:1}_{\text{\tiny IV\!}})} =  &- \!\inner{\csr (M_n \w_{n} \!+\! \tilde{O}_n \u_n), M_n \w_n' \!+\! \tilde{O}_n \u_n'} \! \color{ww} - \! \inner{\csr \ddiv \w_n, (\Id \!- \pi_n^l P_{L^2_0}) \q \!\cdot\! \w_n'} \!\!\!\!\!\!\!\!\! \nonumber\\
      & \color{ww}- \inner{\csr (\Id-\pi_n^l P_{L^2_0} )\q \cdot \w_n, \ddiv \w_n'} \! + \! \inner{\csr \pi_n^l P_{L^2_0} \q \cdot \w_n,  \pi_n^l P_{L^2_0} \q \cdot \w_n'}  \!\!\!\!\!\!\!\!\! \label{eq:adiv:2} \\
	  & \color{ww} - \! \inner{\csr \q \otimes \q \cdot \w_{\tau},  \w_{\tau}'}  \!\!\!\!\!\!\!\!\! \nonumber.
   \end{align}
   
   \footnotetext{We use that $\spl a+b,a'+b' \spr = \spl a +c,a'+c' \spr + \spl a,b'-c' \spr + \spl b-c,a' \spr + \spl b,b' \spr - \spl c,c' \spr$ with $a = \ddiv \w_n$, $b = \q \cdot \w_n$, and $c = \pi_n^l P_{L^2_0} \q \cdot \w_n$; $a',b',c'$ analogously with $\w_n'$.}
   
   We apply the same argumentation to the mixed terms \colorvw{$(\ref{eq:adiv:1}_{\text{\tiny II\!}})$} and \colorwv{$(\ref{eq:adiv:1}_{\text{\tiny III\!}})$}, and note that the terms $\inner{\csr \q \otimes \q \cdot \z_{\tau},  \z_{\tau}'}$, $\z_{\tau} \in \{\v_{\tau}, \w_{\tau} \}$, $\z_{\tau}' \in \{\v_{\tau}',\w_{\tau}'\}$, cancel out with the $\q$-terms in $a_n^{(r - \q \cdot)}(T_n \u_n,\u_n')$. Thus, these terms do not appear in the following, and we are left with $\aremn(\cdot,\cdot)$ instead of $a_n^{(r - \q \cdot)}(\cdot,\cdot)$.

   For $a_n^{(1)}(\cdot,\cdot)$, we will only use the fourth term of \eqref{eq:adiv:2} and shift the remaining terms into $\anII(\cdot,\cdot)$. Hence, we define
   \begin{subequations}
      \begin{align}
         \tag{$\theequation{}|a_n^{(1)}$}\label{eq:anI}
         \anI(&\u_n,\u_n') \coloneqq  \\ 
         \tag{$\theequation{}a|a_n^{(1)}$}
            &\color{vv}\spl c_s^2 \rho \div \v_{\tau}, \div \v_{\tau}' \spr \color{ww}+ \spl c^2_s \rho \pi_n^l P_{L^2_0} \q \cdot \w_{\tau}, \pi_n^l P_{L^2_0} \q \cdot \w_{\tau}' \spr - s_n(\w_n,\w_n') \label{eq:anI:1} \\
            \tag{$\theequation{}b|a_n^{(1)}$}
            &\color{vv}- \spl \rho i \ddiffb \v_n, i \ddiffb \v_n' \spr \color{ww}+ \spl \rho \dopdSmall \w_n, \dopdSmall \w_n' \spr \label{eq:anI:2}\\
            \tag{$\theequation{}c|a_n^{(1)}$}
            &\color{wv}+ \spl \rho \dopdSmall \w_n, i \ddiffb \v_n' \spr 
            \color{vw}- \spl \rho i \ddiffb \v_n, \dopd \w_n' \spr \label{eq:anI:3}\\
            \tag{$\theequation{}d|a_n^{(1)}$}
            &\color{ww}+ \spl \rho (\matii + i \omega \gamma) \w_{\tau}, \w_{\tau}' \spr. \label{eq:anI:4}
      \end{align}
   \end{subequations}
   The $\div$-terms of \colorvv{$(\ref{eq:adiv:1}_{\text{\tiny I\!}})$} and the fourth term in \eqref{eq:adiv:2} from $a_n^{(\div + \q \cdot)}(T_n \u_n, \u_n')$ together with 
   $s_n(T_n \u_n, \u_n')$ form line \eqref{eq:anI:1}, where we note that $s_n(T_n \u_n, \u_n') = \colorww{ - s_n(\w_n,\w_n')}$ since $\hdgjump{P_{V_n} \u_n}_{\nv} = 0$ for all $\u_n \in \Xn$ by construction of $P_{V_n}$. 
   The terms in the lines \eqref{eq:anI:2} and \eqref{eq:anI:3} arise naturally from a selection of terms from $\aconvn(T_n \u_n, \u_n')$. By definition of $\matii = - \rho^{-1} \hess(p) + \hess(\phi)$, we can write $\aremn(\u_n,\u_n') = - \spl \rho(\matii + i \omega \gamma) \u_n, \u_n' \spr$, which is where the line \eqref{eq:anI:4} originates from. 
   
   Below, in \emph{Step 2}, we want to estimate these terms from below by applying \cref{lem:ddiffbEstimate,lem:ddivEstimate}. To this end, we add suitable compact terms to $B_n^{(1)}$ which we simultaneously subtract from $K_n$. For $C_1 > 0$ to be specified later on, we set
   \begin{align}
      \tag{$\theequation{}|K_n^{(1)}$}
	  \!\!\!\!\!\!\spl \KnI \u_n,\u_n' \spr_{\Xn} \!\!\coloneqq\! \colorvv{\spl \v_{\tau},\! \v_{\tau}' \spr} \! + \! C_1 \spl \Sol \u_n, \!\Sol \u_n' \spr \! + \! \colorww{\spl \csr M_n \w_{n},\! M_n \w_{n}' \spr}\!  + \! \spl \csr \tilde{O}_n \u_n,\! \tilde{O}_n \u_n' \spr.\!\!\!\!\!\!\!\!\! \label{eq:KnI}
   \end{align}
   Then, we define $\inner{\BnI \u_n, \u_n'}_{\Xn} \coloneqq \anI(\u_n,\u_n') + \spl \KnI \u_n,\u_n' \spr_{\Xn}$.
   To account for the remaining terms in $a_n(T_n \u_n, \u_n')$, 
   we define 
   \begin{equation}
      \tag{$\theequation{}|a_n^{(2)}$}
      \anII (\u_n, \u_n') \coloneqq a_n(T_n \u_n, \u_n') - \anI(\u_n,\u_n').
   \end{equation}
   To treat the terms in $\anII (\u_n, \u_n')$ we add suitable compact terms which together with $\anII (\u_n, \u_n')$ can be bounded from below. To this end, we  define, for $C_2 > 0$ to be specified later on,
   \begin{subequations}\label{eq:KnII}
      \begin{align}
         \tag{$\theequation{}a|K_n^{(2)}$}
         \spl \KnII &\u_n,\u_n'\spr_{\Xn} \coloneqq C_2 \big(\colorvv{\spl \v_{\tau}, \v_{\tau}'\spr} + \spl \Sol \u_n, \Sol \u_n' \spr + \spl \csr \tilde{O}_n \u_n, \tilde{O}_n \u_n' \spr \\
         \tag{$\theequation{}b|K_n^{(2)}$}
         & + \colorww{\spl \csr M_n \w_{n}, M_n \w_{n}' \spr} + \colorww{\spl \text{mean}(\q \cdot \w_{\tau}), \text{mean}(\q \cdot \w_{\tau}') \spr} \big)
      \end{align}
   \end{subequations}
   where $\text{mean}$ denotes the mean value operator $\frac{1}{\dom}\int_{\dom} \cdot d \mathbf{x}$,
   and set $\inner{\BnII \u_n, \u_n'}_{\Xn} \coloneqq \anII(\u_n, \u_n') + \spl \KnII \u_n, \u_n' \spr_{\Xn}$,  
   which we analyze in \emph{Step 3}. 

   With $K_n \coloneqq - \KnI - \KnII$ and $B_n \coloneqq \BnI + \BnII$, we then obtain that $A_n T_n = B_n + K_n$. The explicit expressions for the operators $\BnI$ and especially the lenghty one for $\BnII$ are written out in \cref{appendix:operators}.
   We note that the uniform boundedness of $B_n$, $n \in \N$, follows straightforwardly. Furthermore, it can be shown that the sequence $\seq{K}$ is indeed \compactop~with the same argumentation as in \cite[Lem.~17]{H23Hdiv}. In particular, we note that $\lim_{n \to \infty} \Vert \tilde{O}_n \Vert_{L(\Xn,L^2)} = 0$, that the operators $M_n$ and $\text{mean}(\cdot)$ give rise to compact terms, and that the construction of $\Sol$ allows us to use the compact embedding $\bH^1 \hookrightarrow \bL^2$. \\[0.8ex]
   \emph{\underline{Step 2:} Uniform coercivity of $\BnI$}. 
   First of all, we show that there exists an index $n_0 > 0$ such that $\BnI$ is uniformly coercive for all $n > n_0$.
   
   Let $\u_n \in \Xn$ be arbitrary and $\delta_0, \epsilon_0 \in (0,1)$, $\tau_0 \in (0,\pi/2-\theta)$ be such that the constant $C_{\theta,\tau,\epsilon,\delta}$ defined by \eqref{eq:MonsterConstant} is positive for all $\delta \in (0,\delta_0)$, $\epsilon \in (0,\epsilon_0)$ and $\tau \in (0,\tau_0)$. We recall that this is possible due to \cref{assumption:MachNumber}, as detailed in \cref{rem:MNsmall}. We further recall that by definition of $\theta$, the estimate \eqref{eq:thetaEstimate} holds. 

   Targeting at coercivity of $B_n^{(1)}$ in the sense of \eqref{eq:coercivity:C},
   we set $\xi \coloneqq  e^{-i (\theta + \tau) \text{sgn} \omega}$, $\vert \xi \vert =1$, so that $\Re( \xi (a + i b)) / \cos(\theta + \tau) = a + \text{sgn} (\omega) \tan(\theta + \tau) b$ for $a,b \in \mathbb{R}$. Using \eqref{eq:anI} and \eqref{eq:KnI}, we note that \eqref{eq:anI:3} becomes purely imaginary for $(\v_n',\w_n') = (\v_n,\w_n)$ and obtain
   \begin{align}
     \Re & \left( \xi \spl \BnI \u_n, \u_n \spr_{\Xn} \right) / \cos (\theta + \tau) \nonumber 
      \\
      &\hspace*{-0.35cm}
      = \colorv{\Vert c_s \rho^{1/2} \div \v_{\tau} \Vert^2_{L^2}} - \colorv{\Vert \rho^{1/2} \ddiffb \v_n \Vert^2_{\bL^2}} + \colorv{\Vert \v_{\tau} \Vert_{\bL^2}} + C_1 \Vert \Sol \u_n \Vert_{\bL^2}^2 
      \nonumber 
      + \colorw{\Vert c_s \rho^{1/2} M_n \w_n \Vert_{L^2}^2} \hspace*{-2cm}\\ 
      &+ \Vert c_s \rho^{1/2} \tilde{O}_n \u_n \Vert_{L^2}^2 + \colorw{\Vert c_s \rho^{1/2} \pi_n^l P_{L^2_0} \q \cdot \w_\tau \Vert^2_{L^2}} \nonumber 
      + \colorw{\Vert \rho^{1/2} \dopd \w_n \Vert^2_{\bL^2}} \hspace*{-2cm} \\ 
      &+ \colorvw{2 \tan(\theta+\tau) \text{sgn}(\omega) \Im \left( \spl \rho \dopd \w_n, i \ddiffb \v_n \spr \right)} - \colorw{s_n(\w_n,\w_n)} \hspace*{-2cm}\nonumber \\      
      & + \colorw{\spl \rho \matii \w_{\tau}, \w_{\tau} \spr_{\bL^2}}  \nonumber 
      + \colorw{\vert \omega \vert \tan(\theta+\tau)  \Vert (\gamma \rho)^{1/2} \w_{\tau} \Vert_{\bL^2}} \geq \dots \nonumber \\
      \intertext{
         Regrouping and applying a weighted Young's inequality ($2ab \leq (1-2\epsilon)^{-1} a^2 + (1-2\epsilon) b^2$)  on the mixed term (involving $\colorw{\w_n}$ and $\colorv{\v_n}$) and using \eqref{eq:thetaEstimate} yields
      }
      &\hspace*{-0.35cm} \dots \!\ge \color{v}\Vert c_s \rho^{1/2} \div \v_{\tau} \Vert^2_{L^2} - \Vert \rho^{1/2} \ddiffb \v_n \Vert^2_{\bL^2} (1\!+\!\color{black}\underbrace{ \color{v} \tan^2(\theta\!+\!\tau)(1\!-\!2\epsilon)^{-1}}_{\eqqcolon C_{\theta,\tau,\epsilon}}\color{v}) \!+\! \colorv{\Vert \v_{\tau} \Vert_{\bL^2}^2} \!\!\! \nonumber &&\markI
      \\[-3.5ex] \nonumber
      & + C_1 \Vert \Sol \u_n \Vert_{\bL^2}^2  &&\markII \\
      &+ \colorw{\Vert c_s \rho^{1/2} M_n \w_n \Vert_{L^2}^2} + \Vert c_s \rho^{1/2} \tilde{O}_n \u_n \Vert_{L^2}^2 + \colorw{\Vert c_s \rho^{1/2} \pi_n^l P_{L^2_0} \q \cdot \w_{\tau} \Vert^2_{L^2}} && \markIII 
      \nonumber
      \\ 
      & \hspace*{-0.19cm}
      \left.\begin{array}{ll}
            + \colorw{2 \epsilon \Vert \rho^{1/2} \dopd \w_n \Vert^2_{\bL^2}} \\[0.2cm]
            + \colorw{\vert \omega \vert \left( \tan(\theta + \tau) - \tan(\theta) \right) \Vert (\gamma \rho)^{1/2} \w_{\tau} \Vert_{\bL^2} - s_n(\w_n,\w_n)}
        \end{array} \hspace*{1cm}\right\}
        && \markIV \nonumber
   \end{align}
   $\markI$ and $\markII$ allow us to control $\colorv{\v_n}$ while $\markIII$ and $\markIV$ are responsible for $\colorw{\w_n}$. 
   We start with estimating $\markI$ from below. Splitting off an $\epsilon$-scaled $\X_n$-norm of $\colorv{\v_n}$ (note that $\colorv{\jnorm{\v_n}{}=0}$), and using \cref{lem:ddiffbEstimate,lem:ddivEstimate} and \eqref{eq:MonsterConstant} for the remainder, we obtain \vspace*{-0.5cm}
   \begin{align*}
      \markI\! &=\!  \colorv{\epsilon \Vert \v_n \Vert_{\X_n}^2\! \!+\! (1\!-\epsilon) \Vert c_s \rho^{1/2} \!\div \v_{\tau} \Vert^2_{L^2}}
      \colorv{- (1 \!+\! C_{\theta,\tau,\epsilon} \!-\! \epsilon) \Vert \rho^{1/2} \ddiffb \v_n \Vert^2_{\bL^2}
      \!+\!\!}\overbrace{\color{v}(1\!-\!\epsilon) \Vert \v_{\tau} \!\Vert_{\bL^2}^2}^{\geq 0}      
      \!\!\\
      &\ge\! \colorv{\epsilon \Vert \v_n \Vert_{\X_n}^2} + \underbrace{C_{\theta,\tau,\epsilon,\delta}  \vert \Sol \u_n \vert^2_{\bH^1_{c_s^2 \rho}}}_{\geq 0} + (1-\epsilon) \left(  \spl \check{O}_n \u_n, \u_n \spr -  C_{\delta} \Vert \Sol \u_n \Vert_{\bL^2}^2 \right) 
   \end{align*}
   ~\\[-5ex]
   We then have 
   \begin{align*}
      \markI\!+\!\markII \ge & \colorv{\epsilon \Vert \v_n \Vert_{\X_n}^2} 
      + \left(C_1 \!-\! (1\!-\!\epsilon)C_{\delta}\right) \! \Vert \Sol \u_n \Vert^2_{\bL^2}
	  - (1-\epsilon) \Vert \check{O}_n \Vert_{L(\Xn)} \Vert \u_n \Vert_{\Xn}^2.
   \end{align*}
   Furthermore, recalling \eqref{eq:divwnEquals}, i.e. $\colorw{(\ddiv + \pi_n^l P_{L^2_0} \q \cdot ) \w_n} = \colorw{-M_n \w_{n}} - \tilde{O}_n \u_n$, we see
  \begin{equation*}
   \begin{aligned}
      \markIII &= \colorw{\Vert c_s \rho^{1/2} M_n \w_{n} \Vert_{L^2}^2 + \Vert c_s \rho^{1/2} \tilde{O}_n \u_n \Vert_{L^2}^2 + \Vert c_s \rho^{1/2} \pi_n^l P_{L^2_0} \q \cdot \w_{\tau} \Vert_{L^2}^2}
       \\
      &
      \ge \frac14 \colorw{\Vert c_s \rho^{1/2} \ddiv \w_n \Vert_{\Tn}^2}.
   \end{aligned}
  \end{equation*}
  For the first term in $\markIV$ we 
  obtain from a weighted Young's inequality (as in \cite{HH21})
$$
\colorw{2 \epsilon \Vert \rho^{1/2} \dopd \w_n \Vert^2_{\bL^2}}
\geq 
\colorw{\epsilon \Vert \rho^{1/2} \ddiffb \w_n \Vert^2_{\bL^2}} - \colorw{C \epsilon
\Vert \rho^{1/2} \w_\tau \Vert^2_{\bL^2}}
$$
for $C\in\mathbb{R}$ only depending on $\omega$ and $\Omega$. For sufficiently small $\epsilon$ we can dominate the latter part by the second line of $\markIV$. Exploiting $\colorw{-s_n(\w_n,\w_n) \geq  \alpha \jnorm{\w_n}{2}}$ and choosing $\epsilon$ small enough (compared to $\alpha$, $\frac14$ and especially $\tau$), we get control of the $\Xn$-norm of $\colorw{\v_n}$:
$$
\markIII + \markIV \geq \epsilon \colorw{\Vert \w_n \Vert^2_{\Xn}}.
$$

Combining all estimates and $\Vert \u_n \Vert_{\X_n}^2 \leq 2 \Vert \v_n \Vert_{\X_n}^2 + 2 \Vert \w_n \Vert_{\X_n}^2$, we obtain 
  \begin{equation*}
   \begin{aligned}
      &\Re \left( \xi \spl \BnI \u_n, \u_n \spr_{\Xn} \right) / \cos(\theta + \tau)\ge \markI + \markII + \markIII + \markIV \\
      &\ge \frac{\epsilon}{2} \Vert \u_n \Vert^2_{\Xn} + \left( C_1 - (1-\epsilon)C_{\delta} \right) \Vert \Sol \u_n \Vert^2_{\bL^2} - (1-\epsilon) \Vert \check{O}_n \Vert_{L(\Xn)} \Vert \u_n \Vert_{\Xn}^2.
   \end{aligned}
  \end{equation*}
  Since $\lim_{n \to \infty} \Vert \check{O}_n \Vert_{L(\Xn)} = 0$, we can choose $n_1 > n_0$ large enough such that $(1-\epsilon) \Vert \check{O}_n \Vert_{L(\Xn)} < \epsilon/4$ and thus with $C_1 > (1-\epsilon)C_{\delta}$
  \begin{equation*}
   \begin{aligned}
      \Re \left( \xi \spl \BnI \u_n, \u_n \spr_{\Xn} \right)  / \cos(\theta + \tau)
      \ge \frac{\epsilon}{4} \Vert \u_n \Vert^2_{\Xn} 
   \end{aligned}
  \end{equation*}
  for $n > n_1$. 
  Thus, $\BnI$ is uniformly coercive (in the sense of \eqref{eq:coercivity:C}) for all $n > n_1$. \\[0.8ex]
  \emph{\underline{Step 3:} Uniform coercivity of $B_n$}. We have shown that $\BnI$ is uniformly coercive in \emph{Step 2}, so it remains to show that $B_n$ (after addition of $\BnII$) is uniformly coercive as well.
  To this end, we want to derive a bound for $\BnII$ of the form
  \begin{align}
   \Re ( \xi \spl \BnII \u_n, \!\u_n \spr_{\Xn}) /\! \cos(\theta + \tau) \!\ge\! -(\epsilon/8 \!+\! \zeta_{1,n}) \Vert \u_n \Vert^2_{\Xn} \!+ (C_2 \!-  \zeta_2) \vert \u_n \vert_{\KnII}^2 \!
   \label{BnII:lowerbound}
  \end{align}
  for $(\zeta_{1,n})_{n\in \mathbb{N}}$ and $\zeta_{1,n}, \zeta_2 \in \mathbb{R}_+$ with the semi-norm 
  \begin{align*}
   \vert \u_n \vert_{\KnII}^2 & \! \! \coloneqq  \! \colorv{\Vert \v_{\tau} \Vert^2_{\bL^2}}   \! \! + \! \Vert \Sol \u_n \Vert^2_{\bL^2}  \!+ \! \Vert c_s \rho^{1/2} \tilde{O}_n \u_n \Vert^2_{L^2}  \!+ \! \colorw{\Vert c_s \rho^{1/2} M_n  \! \w_{n} \Vert^2_{L^2}}  \!+ \! \colorw{\Vert \text{mean}(\q  \! \cdot  \! \w_{\tau}) \Vert^2_{L^2}} \\
   &= C_2^{-1} \KnII(\u_n,\u_n).
  \end{align*}
  We will show that $\zeta_{1,n} \to 0$ for $n \to \infty$ so that $- (\epsilon/8 + \zeta_{1,n})  \Vert \u_n \Vert^2_{\Xn}$ can be dominated by the $\BnI$-contribution for sufficiently large $n$. The bound \eqref{BnII:lowerbound} allows us to choose $C_2$ sufficiently large to compensate for the $- \zeta_2 \vert \u_n \vert_{\KnII}^2$ term and obtain the uniform coercivity of $B_n$.

  To prove \eqref{BnII:lowerbound}, it suffices to show the boundedness of $\anII$ in the following form
  \begin{align}
   | \spl \anII \u_n, \u_n \spr_{\Xn} | \le \eta_{1,n} \Vert \u_n \Vert^2_{\Xn} + \eta_2 \Vert \u_n \Vert_{\Xn} \vert \u_n \vert_{\KnII}. 
   \label{eq:anII:upperbound:I}
   \intertext{
      for $(\eta_{1,n})_{n \in \mathbb{N}} \to 0$ for $n \to \infty$ and $\eta_2$ bounded. Then, a weighted Young's inequality of the form $\eta_2 a b \leq \frac12 \frac{\epsilon}{4 \cos(\theta+\tau)} a^2 + \frac12 \frac{4 \cos(\theta+\tau) \eta_2^2}{\epsilon} b^2$ yields 
   }
      | \spl \anII \u_n, \u_n \spr_{\Xn} | \le \frac{\epsilon/8 + \zeta_{1,n}}{\cos(\theta+\tau)} \Vert \u_n \Vert^2_{\Xn} + 
      \frac{\zeta_2}{\cos(\theta+\tau)} \vert \u_n \vert_{\KnII}^2. 
      \label{eq:anII:upperbound}
      \end{align}
   with $\zeta_{1,n} = \eta_{1,n} \cos(\theta+\tau)$ and $\zeta_2 = 2/\epsilon \cos^2(\theta + \tau) \eta_2^2$ which implies \eqref{BnII:lowerbound}.
   It hence remains to show \eqref{eq:anII:upperbound:I}.

   $\anII(\cdot,\cdot)$ effectively contains all the terms of $a_n(\cdot,\cdot)$ that are not considered in $\anI(\cdot,\cdot)$. 
   Most terms are of the form that they pair a term that can be bounded by $\vert \cdot \vert_{\KnII}$ with another term that can be bounded by $\vert \cdot \vert_{\Xn}$. Those terms are thence directly suitable for \eqref{eq:anII:upperbound} and in the following, we only discuss the terms that do not match this pattern; for completeness we state the full expression $\anII(\cdot,\cdot)$ in \cref{appendix:operators}.

   The terms  of \eqref{eq:anII:upperbound} that do not match the pattern described above stem from contributions of $\colorw{\ddiv \w_n + \q \cdot \w_n}$. 
   We then shift in $\colorw{\pi_n^l P_{L^2_0} (\q \cdot \w_n)}$, cf.~\eqref{eq:adiv:2}, and make use of \eqref{eq:divwnEquals} to bound $\colorw{\ddiv \w_n + \pi_n^l P_{L^2_0} \q \cdot \w_n}$ by $\vert \u_n \vert_{\KnII}$. 
   The remaining terms are then of the form 
   $\spl \colorw{c_s^2 \rho (\Id - \pi_n^l  P_{L^2_0} ) (\q \cdot \w_{\tau})}, \ddiv \z_n \spr$ for $\z_n \in \{\colorv{\v_n'}, \colorw{\w_n'}\}$.
   To exploit the approximation properties of $\pi_n^l$ here, we first shift in another $\text{mean} (\q \cdot \w_n)$ term (which is itself bounded by constant times $\Vert \u_n \Vert_{\KnII}$) and only have to deal with the following expression where we abbreviate $\Pi^\ast = \Id - \text{mean} - \pi_n^l  P_{L^2_0}$: 
   \begin{align*}
   &\vert \spl \colorw{c_s^2 \rho \Pi^\ast (\q \cdot \w_{\tau})}, \ddiv \z_n \spr \vert
   = \vert \spl \colorw{\q \cdot \w_{\tau}}, \Pi^\ast (c_s^2 \rho \ddiv \z_n) \spr \vert \\
   &\le \Vert \q \Vert_{\bL^\infty} \colorw{\Vert \w_\tau \Vert_{\mathbf{L}^2}} \Vert \Pi^\ast (c_s^2 \rho \ddiv \z_n) \Vert_{L^2} \lesssim \Vert \u_n \Vert_{\Xn} \Vert \Pi^\ast (c_s^2 \rho \ddiv \z_n) \Vert_{L^2}, 
  \end{align*}
  We estimate the last term using a discrete commutator technique \cite{B99} to obtain
  \begin{align*}
   \Vert & \Pi^\ast (\csr  \ddiv \z_n )\Vert_{L^2}^2 \!=\!\! \sum_{\tau \in \Tn} \Vert \Pi^\ast (c_s^2 \rho \ddiv \z_n) \Vert^2_{L^2(\tau)}
   \!=\!\! \sum_{\tau \in \Tn} \Vert \Pi^\ast ((c_s^2 \rho \!-\! c_{\tau}) \ddiv \z_n) \Vert^2_{L^2(\tau)} \\
   & \le \sum_{\tau \in \Tn} \Vert c_s^2 \rho - c_{\tau} \Vert_{L^\infty(\tau)}^2 \Vert \ddiv \z_n \Vert_{L^2(\tau)}^2 
   \le (C^L_{c_s^2 \rho})^2 h_n^2 \Vert \z_n \Vert^2_{\Xn} \lesssim h_n^2 \Vert \u_n \Vert^2_{\Xn}.
  \end{align*}
  Here, we used that $\Pi^\ast r = 0$ for any piecewise polynomial $r$ of degree $k-1$
  and hence $\Pi^\ast c_\tau \ddiv \z_n = 0$ for any piecewise constant $c_\tau$.

  This reveals that all terms in $\anII$ are suitable for \eqref{eq:anII:upperbound:I} and hence \eqref{BnII:lowerbound} holds. 
  Together with  \emph{Step 2}, $C_2$ and $n$ sufficiently large we hence obtain 
  \begin{align}
   \Re ( \xi \spl B_n \u_n, \!\u_n \spr_{\Xn}) \!\ge\! 3\epsilon/16 \cos(\theta + \tau) \Vert \u_n \Vert^2_{\Xn} 
   \label{Bn:coercive}.
  \end{align}
  \emph{\underline{Step 4:} Asymptotic consistency of $B_n$}. In the last step,  we want to show that there exists a bijective operator $B \in L(\X)$ such that $B_n$ \approximates~$B$. 
  
  For $\u, \u' \in \X$, we define
   \begin{subequations}\label{eq:K}
      \begin{align}
         \spl K \u, \u' \spr_{\X} \coloneqq &-(1+C_2) \colorv{\spl \v, \v' \spr} - (C_1 + C_2) \colorv{\spl \v, \v' \spr} \label{eq:K:a} \tag{$K$-a} \\
         &- (1+C_2) \colorw{\spl \csr M \w, M \w' \spr} - C_2 \colorw{\spl \text{mean}(\q \cdot \w), \text{mean}(\q \cdot \w') \label{eq:K:b} \tag{$K$-b} \spr}
      \end{align}
   \end{subequations}
   and set $B \coloneqq AT - K$. Since $M$ is of finite rank and $\v = P_V \u \in \bH^1$, the compactness of the operator $K$ follows from the compact embedding $\bH^1 \hookrightarrow \bL^2$. Rewriting the operator $B$ with similar arguments as in \emph{Step 1}, cf.~\eqref{eq:B}, we can show that $B$ is coercive using the same arguments as in \emph{Step 2} and \emph{Step 3}. Thus, $B$ is bijective. It remains to show that $B_n$ \approximates~$B$, that is $\lim_{n \to \infty} \Vert (B_n p_n - p_n B) \u \Vert_{\Xn} = 0$. 
   
   To this end, we note that it suffices to show that $K_n$ \approximates~$K$, since we can use that $B_n = A_n T_n - K_n$, $B = AT-K$ and estimate 
   \begin{align*}
      \Vert &(B_n p_n - p_n B) \u \Vert_{\Xn} \le \Vert (K_n p_n - p_n K) \u \Vert_{\Xn} + \Vert (A_n T_n p_n - p_n A T ) \u \Vert_{\Xn} \\
      &\le \Vert (K_n p_n - p_n K) \u \Vert_{\Xn} + \Vert (A_n p_n - p_n A) T \u \Vert_{\Xn} 
      + \Vert A_n \Vert_{L(\Xn)} \Vert (T_n p_n - p_n T)  \u \Vert_{\X}.
   \end{align*}
   By \cref{thm:AnToA} and \cref{lem:TnToT}, $A_n$ \approximates~$A$ and $T_n$ \approximates~$T$ so the last two terms converge to zero. Thus, it indeed suffices to show that $K_n$ \approximates~$K$ to conclude that $B_n$ \approximates~$B$. 
   
   Similar to the proof of \cref{thm:AnToA}, we choose $\u_n ' \in \Xn$, $n \in \N$, $\Vert \u_n' \Vert_{\Xn} = 1$, such that $\Vert (K_n p_n - p_n K) \u \Vert_{\Xn} \le \vert \spl (K_n p_n - p_n K) \u, \u_n' \spr_{\Xn} \vert + 1/n$. For an arbitrary subsequence $\N' \subset \N$, we can choose $\u' \in \X$ and a subsequence $\N'' \subset \N$ such that $\u_n' \overset{\bL^2}{\rightharpoonup} \u'$, $\csr \ddiv \u_n' \overset{L^2}{\rightharpoonup} \csr \div \u'$ and $\rho \ddiffb \u_n' \overset{\bL^2}{\rightharpoonup} \rho \diffb \u'$ due to \cref{lem:weakConvergence}. 
   
   We compute that 
   \begin{align*}
      \inner{p_n K \u, \u_n'}_{\Xn} \overset{\eqref{eq:def:pn}}&{=} \inner{\div K \u, \csr \ddiv \u_n'} + \inner{K \u, \u_n'} + \inner{\diffb K \u, \rho \ddiffb \u_n'}\\
      \overset{n \in \N''}&{\to} \innercsr{\div K \u, \u'} + \inner{K \u, \u'} + \innerr{\diffb K \u, \diffb \u'} = \inner{K \u, \u'}_{\X}.
   \end{align*}

   Thus, we have to show that $\lim_{n \in \N''} \spl K_n p_n \u, \u_n' \spr_{\Xn} = - \spl K \u, \u' \spr_{\X}$. We recall that $K_n \coloneqq - \KnI - \KnII$ and therefore
   \begin{subequations}
      \begin{align}
         \inner{K_n p_n &\u, \u_n'}_{\Xn}\nonumber \\
         =  &-(1 + C_2) \inner{(\colorv{P_{V_n} p_n \u})_{\tau},\v_{\tau}'} - (C_1 + C_2) \inner{\Sol p_n \u, \Sol \u_n'} \label{eq:Knpnu:a} \tag{$K_n$-a} \\
         &-(1 + C_2) \inner{\csr \tilde{O}_n p_n \u, \tilde{O}_n \u_n'} - (1 + C_2) \inner{\csr M_n (\Id_{\Xn} - P_{V_n}) p_n \u, \colorw{\w_{n}'}} \label{eq:Knpnu:b} \tag{$K_n$-b} \\
         &-(1 + C_2) \inner{\text{mean}(\q \cdot (\Id_{\Xn} - P_{V_n})\PVol{p_n \u}), \colorw{\text{mean}(\q \cdot \w_{\tau}')}}. \label{eq:Knpnu:c} \tag{$K_n$-c}
       \end{align}
   \end{subequations} 
   In the following, we show that \eqref{eq:Knpnu:a} converges to $\eqref{eq:K:a}$ and $\eqref{eq:Knpnu:b} + \eqref{eq:Knpnu:c}$ converges to $\eqref{eq:K:b}$.

   \emph{Step 4a: Convergence of \eqref{eq:Knpnu:a}.} To show that $\eqref{eq:Knpnu:a}$ converges to $\eqref{eq:K:a}$, we exploit the convergence of $p_n \u$.
   Using the approximation properties of $\pi_n^d$ and the same argumentation as in \cref{lem:TnToT}, cf. especially \eqref{eq:PVSnpn:conv}, we have
   \begin{align*}
         \vert &\colorv{\spl \v, \v_{\tau}' \spr} - \spl \PVol{P_{V_n} p_n \u}, \colorv{\v_{\tau}'} \spr \vert = \vert (\spl P_{V} - \pi_n^d \Sol p_n) \u, \colorv{\v_{\tau}'} \spr \vert  \\
         &\lesssim \vert \spl (P_V - \pi_n^d P_{V}) \u, \colorv{\v_{\tau}'} \spr \vert  + \vert \spl \pi_n^d (P_{V} - \Sol p_n )\u, \colorv{\v_{\tau}'}\spr \vert \\
         &\lesssim h_{n} \Vert P_V \u \Vert_{\bH^1} + d_n(\u,p_n \u) + \Vert (\Id_{L^2_0} - \pi_n^l) P_{L^2_0} \q \cdot \u \Vert_{L^2} \overset{n \in \N''}{\to} 0, \\
      \text{and  }~ \vert &\spl \colorv{\v}, \Sol \u_n' \spr - \spl \Sol p_n \u, \Sol \u_n' \spr \vert \!=\! \vert \spl (P_V - \Sol p_n) \u, \Sol \u_n' \spr \vert 
      \\
      &\!\lesssim\! \Vert (P_{V} - \Sol p_n) \u \Vert_{\bL^2} 
      \lesssim d_n(\u,p_n \u) + \Vert ( \Id_{L^2_0} - \pi_n^l) P_{L^2_0} \q \cdot \u \Vert_{L^2}  \overset{n \in \N''}{\to} 0,
   \end{align*}
   where we used \cref{lem:dnpnuToZero} and the pointwise convergence of $\pi_n^l$ to $\Id_{L^2_0}$. Hence,
   \begin{equation*}
      \lim_{n \in \N''} \eqref{eq:Knpnu:a} = \lim_{n \in \N''} \left( -(1 + C_2) \colorv{\inner{\v, \v_{\tau}'}} - (C_1 + C_2) \inner{\colorv{\v},\Sol \u_n'} \right).
   \end{equation*}
   It remains to show that the right-hand side converges to $\eqref{eq:K:a}$, i.e. 
   $\inner{\colorv{\v},\Sol \u_n'} \overset{n \in \N''}{\to} \colorv{\inner{\v,\v'}}$. Let $S := \nabla ( (\div + P_{L^2_0}\q \cdot + M) \nabla)^{-1}$ such that
   \begin{align*}
      \spl &\colorv{\v}, \Sol \u_n' \spr - \colorv{\spl \v, \v' \spr}=  \spl \colorv{\v}, \Sol \u_n' - P_V \u' \spr \\
      &\!=\!  \spl \colorv{\v}, S (\ddivpqMn \u_n' - \divPqM \u') \spr \\
      &\!=\!  \spl S^\ast \colorv{\v}, \ddiv \u_n' \!-\! \div \u' \spr \!+\! \spl S^\ast \colorv{\v}, \pi_n^l P_{L^2_0} \q \cdot \PVol{\u_n'} \!-\! P_{L^2_0} \q \cdot \u' \spr 
      \!+\! \spl S^\ast \colorv{\v}, M_n \u_n' \!-\! M \u' \spr.
   \end{align*}
   By choice of the subsequence $\N'' \subset \N$ and \cref{lem:weakConvergence}, we have that $\csr \ddiv \u_n' \overset{L^2}{\rightharpoonup} \csr \div \u'$ (and hence $\ddiv \u_n' \overset{L^2}{\rightharpoonup} \div \u'$) and therefore $M_n \u_n' \to M \u'$ by construction of $M_n$ and $M$. Therefore, the first and the last term converge to zero as $S^\ast \colorv{\v} \in L^2$. 
   Furthermore, we have that 
   \begin{equation*}
      \begin{aligned}
         &\spl S^\ast \colorv{\v}, \pi_n^l P_{L^2_0} \q \cdot \PVol{\u_n'} - P_{L^2_0} \q \cdot \u' \spr \\
         = &\spl S^\ast \colorv{\v}, (\pi_n^l - \operatorname{id}) P_{L^2_0} \q \cdot \u \spr + \spl S^\ast \colorv{\v}, \pi_n^l P_{L^2_0} \q \cdot (\PVol{\u_n'} - \u') \spr \\
      \end{aligned}
   \end{equation*} 
   The first term converges to zero due to the pointwise convergence of $\pi_n^l$ to $\Id$. For the second term, we first notice that $\pi_n^l$ is uniformly bounded and by \cref{lem:weakConvergence} $\PVol{\u_n'} \overset{\bL^2}{\rightharpoonup} \u'$. Therefore $P_{L^2_0} \q \cdot \u_n' \overset{\bL^2}{\rightharpoonup} P_{L^2_0} \q \cdot \u'$ because the compact operator $P_{L^2_0} \q \cdot$ maps weakly convergent sequences onto weakly convergent sequences. We conclude $\inner{\colorv{\v},\Sol \u_n'} \overset{n \in \N''}{\to} \colorv{\inner{\v,\v'}}$ and hence $\lim_{n \in \N''} \eqref{eq:Knpnu:a} = \eqref{eq:K:a}$. 

   \emph{Step 4b: Convergence of \eqref{eq:Knpnu:b} \& \eqref{eq:Knpnu:c} to \eqref{eq:K:b}.}
   For \eqref{eq:Knpnu:b}, we first note that $\lim_{n \to \infty} \Vert \tilde{O}_n \Vert_{L(\Xn,L^2_0)} = 0$ due to the arguments in \cref{lem:PvIdempotent} so that we only have to consider the second term. 
   With similar arguments as in the proof of \cref{lem:TnToT} we have (recall that $\colorw{\w} = (\Id - P_{V})\u$)
   \begin{align*}
      \vert &\colorw{\spl M \w, M_n \w_{n}' \spr} - \spl M_n (\Id - P_{V_n})(p_n \u), M_n \colorw{\w_{n}'} \spr \vert \\
      &\lesssim \Vert \div (\u - P_V \u) - \ddiv (p_n \u - P_{V_n} p_n \u) \Vert_{L^2} \lesssim d_n(\u,p_n \u) + d_n(P_V \u, P_{V_n} p_n \u), 
   \end{align*}
   where the first term converges to zero due to \cref{lem:dnpnuToZero} and the second term converges to zero with the same argumentation as in the proof of \cref{lem:TnToT}, cf.~\eqref{eq:PVuPVnpnuToZero}. Applying \cref{lem:weakConvergence} yields $\lim_{n \in \N''} \eqref{eq:Knpnu:b} = \lim_{n \in \N''} (1 + C_2) \colorw{\inner{M \w, M_n \w_n'}} = (1 + C_2) \colorw{\inner{M \w, M \w'}}$.

   Finally, we consider \eqref{eq:Knpnu:c} and calculate that
   \begin{align*}
      \vert &\colorw{\spl \text{mean} (\q \cdot \w), \text{mean}(\q \cdot \w_{\tau}') \spr} - \spl \text{mean} (\q \cdot \PVol{(\Id-P_{V_n})(p_n \u)}), \colorw{\text{mean}(\q \cdot \w_{\tau}')} \spr \vert \\
      &\lesssim \Vert \u - P_{V} \u - \PVol{p_n \u - P_{V_n} p_n \u} \Vert_{\bL^2} \\
      &\lesssim \Vert p_n P_V \u - \PVol{ P_{V_n} p_n \u} \Vert_{\bL^2} + d_n(\u,p_n \u) + d_n (P_V \u, p_n P_{V_n} \u) \overset{n \in \N''}{\to} 0, 
   \end{align*}
   where we again apply \cref{lem:dnpnuToZero} and \cref{lem:TnToT}.

   Altogether, we obtain that $\lim_{n \in \N''} \eqref{eq:Knpnu:a} = \eqref{eq:K:a}$ and $\lim_{n \in \N''} (\eqref{eq:Knpnu:b} + \eqref{eq:Knpnu:c}) = \eqref{eq:K:b}$ and thus $K_n$ \approximates~$K$. Due to the argumentation above, we conclude that $B_n$ \approximates~$B$ which finishes the proof.
   
   To summarize, in \emph{Step 1} we have defined the decomposition $A_n T_n = B_n + K_n$, where the sequence $\seq{K}$ is compact. In \emph{Step 2} and \emph{Step 3}, we have shown that the sequence $\seq{B}$ is uniformly coercive and therefore \stable. Finally, in \emph{Step 4}, we have shown that there exists a bijective operator $B \in L(\X)$ such that $B_n$ \approximates~B.
   \end{proof}

\subsection{Convergence estimates}\label{subsec:CA:convergenceEstimates}
To conclude the analysis of the discrete problem, we show that the sequence of discrete solutions $\seq{\u}$ converges to the solution of the continuous problem. Further, if we assume additional regularity for the continuous solution, we obtain convergence with optimal order. 

\begin{theorem}\label{thm:convergence}
   Assume that \cref{assumption:MachNumber} holds. For $\bff \in \bL^2$, let $\u \in \X$ be the solution to \eqref{eq:cont:weakForm}. Then, there exists an index $n_0 > 0$ such that for all $n > n_0$ the problem \eqref{eq:discr:weakForm} has a unique solution $\u_n \in \Xn$ and $\lim_{n \to \infty} d_n(\u,\u_n) = 0$. 
\end{theorem}

\begin{proof}
   In \cref{thm:AnToA}, \cref{lem:TnToT}, and  \cref{thm:weakTcompatibilityFulfilled}, we have shown that the operators $A$, $\seq{A}$, $T$, and $\seq{T}$ fulfill the necessary conditions for the application of \cref{thm:weakTcompatibility} to conclude that the sequence $\seq{A}$ is \regular. To be able to apply \cref{lem:RegularStable}, which yields \stability~and \convergence, we still have to show that the continuous right-hand side converges to the discrete right-hand side in the sense of discrete approximation schemes. Let $\bm{g} \in \X$ be such that $\spl \bm{g}, \u \spr_{\X} = \spl \bff, \u \spr_{\bL^2}$ for all $\u \in \X$ and $\bm{g}_n \in \Xn$ be such that $\spl \bm{g}_n, \u_n \spr_{\Xn} = \spl \bff, \u_n \spr_{\bL^2}$ for all $\u_n \in \Xn$. Take $\u_n' \in \Xn$, $\Vert \u_n' \Vert_{\Xn} = 1$ such that $\Vert p_n \bm{g} - \bm{g}_n \Vert_{\Xn} \le \vert \spl p_n \bm{g} - \bm{g}_n, \u_n' \spr \vert + 1/n$ and for an arbitrary subsequence $\N' \subset \N$ choose $\N'' \subset \N'$ according to \cref{lem:weakConvergence}. Then, we have that with \eqref{eq:def:pn}
   \begin{align*}
      &\inner{p_n \bm{g} - \bm{g}_n, \u_n'}_{\Xn} = \inner{p_n \bm{g}, \u_n'}_{\Xn} - \inner{\bff, \u_n'}_{\bL^2} \\
      = &\innercsr{\div \bm{g}, \ddiv \u_n'} \! - \!\inner{\bm{g}, \u_n'} \!+ \!\innerr{ \diffb \bm{g}, \ddiffb \u_n'} \!- \!\inner{\bff, \u_n'}_{\bL^2} \!\! \overset{n \in \N''}{\to} \!\! \inner{\bm{g}, \u'}_{\X} \! - \!\inner{ \bff, \u'}_{\bL^2}\!\! = \! 0. 
   \end{align*}
   Thus, we can apply \cref{lem:RegularStable} to conclude that the sequence $\seq{A}$ is \stable, i.e.~there exists an index $n_0$ such that $A_n^{-1}$ exists and is bounded for all $n > n_0$ and problem \eqref{eq:discr:weakForm} has a unique solution for all $n > n_0$. Furthermore, it holds that $\lim_{n \to \infty} \Vert p_n \u - \u_n \Vert_{\Xn} = 0$. We estimate with \cref{lem:pnBoundedBypinX}
   \begin{equation}\label{eq:convergenceThm:estimate}
         d_n(\u,\u_n) \le d_n(\u,p_n \u) + \Vert p_n \u - \u_n \Vert_{\Xn} \le d_n(\u,\pinX \u) + \Vert p_n \u - \u_n \Vert_{\Xn},
   \end{equation}
   and apply \cref{lem:dnupinXdtoZero} to conclude that $\lim_{n \to \infty} d_n(\u,p_n \u) = 0$. 
\end{proof}

\begin{theorem}\label{thm:convEstimates}
   Let the assumptions from \cref{thm:convergence} be satisfied. If additionally $\u \in \X \cap \bH^{2+s}$, $s > 0$, $\rho \in W^{1+s,\infty}$, and $\bflow \in \bW^{1+s,\infty}$, then there exists a constant $C > 0$ independent of $n$ such that
   \begin{equation*}
      d_n(\u,\u_n) \le C h_n^{\min\{s,k,l\}}  \Vert \u \Vert_{\bH^{2+s}}.
   \end{equation*}
\end{theorem}

\begin{proof}
   To show the convergence rate, we continue to estimate \eqref{eq:convergenceThm:estimate}. For the first term,  \cref{lem:pinXRates} yields that $d_n(\u,\pinX \u) \lesssim h_n^{\min\{1+s,k\}}$. For the second term, we note that
   \begin{align*}
     \Vert p_n \u - \u_n \Vert_{\Xn} \le ( \sup_{n > n_0} \Vert A_n^{-1}  \Vert_{L(\Xn)}) \Vert A_n(p_n \u - \u_n) \Vert_{\Xn}
   \end{align*}
   and compute with similar arguments as in \cref{thm:AnToA} that 
   \begin{align*}
      \Vert A_n &(p_n \u - \u_n) \Vert_{\Xn} = \sup_{\u_n' \in \Xn  \Vert \u_n' \Vert_{\Xn} = 1} \vert a_n(p_n \u -  \u_n, \u_n') \vert \\
      &\le C d_n(\u,p_n \u) + \sup_{\u_n' \in \Xn  \Vert \u_n' \Vert_{\Xn} = 1} \vert \spl c_s^2 \rho \div \u, \ddiv \u_n' \spr \\
      &\quad - \spl \rho \opd \u, \dopd \u_n' \spr \\
      &\quad + \spl \div \u, \nabla p \cdot \u_n' \spr + \spl \nabla p \cdot \u, \ddiv \u_n' \spr + \spl (\hess(p) - \rho \hess(\phi)) \u, \u_n' \spr \\
      &\quad - i \omega \spl \gamma \rho \u, \u_n' \spr - \spl \bff, \u_n' \spr \vert. 
   \end{align*}
   For the first term, we again use the estimates from \cref{lem:pnBoundedBypinX} and \cref{lem:pinXRates}. For the remainder, we want to integrate by parts and use the fact that $\u$ solves \eqref{eq:Galbrun}. This requires that $\u \in \bH^2$ is regular enough, because a right hand-side $\bff \in \bL^2$ only grants $-\nabla(\csr \div \u) + \rho \diffb \diffb \u \in \bL^2$ which however does not imply that $-\nabla(\csr \div \u) \in \bL^2$ or $\rho \diffb \diffb \u \in \bL^2$.
   
   Let $\bm{\psi}_n \in [\mathbb{P}^l(\Tn)]^d$ be a suitable $H^1$-projection of $\opd \u$, for example as in \cite{EG16}. Then, we have that 
   \begin{equation*}
      \spl \opd \u, \rho \ddiffb \u_n' \spr = \spl \bm{\psi}_n, \rho \ddiffb \u_n' \spr + \spl \opd \u - \bm{\psi}_n, \rho \ddiffb \u_n' \spr, 
   \end{equation*}
   and similar to \eqref{eq:weakConvergence:partialInt} we calculate with the definition \eqref{eq:VecLift} and $\div(\rho \bflow) = 0$ that  
   \begin{equation*}
      \begin{aligned}
         \spl \bm{\psi_n}, &\rho \ddiffb \u_n' \spr = \spl \bm{\psi}_n, \rho \diffb \u_\tau' \spr_{\Tn} + \spl \bm{\psi}_n, \rho \bRl \u_n' \spr_{\Tn} \\
         &= - \spl \rho \diffb \bm{\psi}_n, \u_\tau' \spr_{\Tn} + \spl \bm{\psi}_n, \rho \hdgjump{\u_n'}_{\bflow} \spr_{\partial \Tn} + \spl \bm{\psi}_n, \rho \bRl \u_n' \spr_{\Tn} \\
         &= - \spl \rho \diffb \opd \u, \u_\tau' \spr_{\Tn} + \spl  \rho \diffb  (\opd \u - \bm{\psi}_n), \u_\tau' \spr_{\Tn}. 
      \end{aligned}
   \end{equation*}
    With similar techniques, cf.~\cite[Thm.~6.26]{Thesis_vB23}, we obtain for $\psi_n, \tilde{\psi}_n \in \mathbb{P}^l(\Tn)$ being suitable $H^1$-projections of $c_s^2 \rho \div \u$ and $\nabla p \cdot \u$ that 
   \begin{align*}
      \spl c_s^2 \rho \div \u, \ddiv \u_n' \spr = &- \spl \nabla(c_s^2 \rho \div \u), \u_n' \spr + \spl \nabla(c_s^2 \rho \div \u - \psi_n), \u_n' \spr \\
      &+ \spl c_s^2 \rho \div \u - \psi_n, \ddiv \u_n' \spr, \\ 
      \spl \nabla p \cdot \u, \ddiv \u_n' \spr = &- \spl \nabla(\nabla p \cdot \u), \u_n' \spr + \spl \nabla(\nabla p \cdot \u - \tilde{\psi}_n), \u_n' \spr \\
      &+ \spl \nabla p \cdot \u - \tilde{\psi}_n, \ddiv \u_n' \spr.
   \end{align*}
   Altogether, we obtain that 
   \begin{align*}
      \sup_{\u_n' \in \Xn  \Vert \u_n' \Vert_{\Xn} = 1} &\vert \spl c_s^2 \rho \div \u, \ddiv \u_n' \spr - \spl \rho \opd \u, \dopd \u_n' \spr \\
      &\ + \spl \div \u, \nabla p \cdot \u_n' \spr \! + \! \spl \nabla p \cdot \u, \ddiv \u_n' \spr \! + \! \spl (\hess(p) - \rho \hess(\phi)) \u, \u_n' \spr \\
      &\ - i \omega \spl \gamma \rho \u, \u_n' \spr - \spl \bff, \u_n' \spr \vert \\
      &\lesssim \Vert \rho \opd \u - \bm{\psi_n} \Vert_{\bH^1} + \Vert c_s^2 \rho \div \u - \psi_n \Vert_{H^1} \\
      &\  + \Vert \nabla p \cdot \u - \tilde{\psi}_n \Vert_{H^1} \\
      &\lesssim h^{\min\{ l,s\}}. 
   \end{align*}
   Combining all estimates, we obtain the desired result.
\end{proof}

\begin{remark}[The case $l = k-1$]\label[remark]{rem:CA:improvement}
   While the convergence estimate in \cref{thm:convEstimates} suggest the choice $l = k$ for the degree of the lifting operator $\bRl$, the numerical experiments in \cref{subsec:numex:conv} suggest that the choice $l = k-1$ (together with a reduced facet space $\XF = \mathbb{P}^{k-1}(\Fn)$ and $\XT = \mathbb{BDM}^k(\Tn)$) might be sufficient to obtain optimal convergence rates. Thus, in certain cases, it might be possible to improve the results presented in \cref{thm:convEstimates} to obtain the estimate $d_n(\u,\u_n) \lesssim h_n^{\min\{s,k,l+1\}}$. 
\end{remark}

\section{Numerical Experiments}\label{sec:numerics}
In this section, we study the discretization of Galbrun's equation with HDG methods numerically. After some preliminary discussions in \cref{subsec:numex:prelim} and implementational aspects in \cref{subsec:numex:compasp}, we study the convergence of different HDG methods towards an exact solution in \cref{subsec:numex:conv}. The main goal is to verify the convergence rates obtained in \cref{thm:convEstimates} numerically. In \cref{subsec:numex:mach}, we investigate the influence of the Mach number on the discretization error. Afterwards, we compare the proposed discretization of the convection term through lifting operators with a naive SIP discretization in \cref{subsec:numex:SIP}. To conclude, we consider numerical examples with physically relevant coefficients from the Sun in \cref{subsec:numex:sun}. For the implementation we use the finite element software \texttt{NGSolve} \cite{Sch97,Sch14}. Replication data is available in \cite{HLvB25}.

\subsection{Preliminaries}\label{subsec:numex:prelim}
In the numerical examples presented below, we want to compare different HDG discretizations. The most natural choices are a \emph{fully nonconforming HDG} method or an \emph{$H(\div)$-conforming HDG} method, cf.~\cref{rem:HdivHDG}. 
However, choosing polynomials of order $k$ for the facet unknowns might not be optimal in terms of computational efficiency. 
For different problems, optimal convergence rates have been obtained with facet unknowns of only order $k-1$, i.e. one order reduced compared to the volume unknowns. 
This can be achieved by involving only the $L^2$-projection on polynomials of degree $k-1$ for the hybrid DG jump operator on the facets. For a SIP discretization this has been achieved through the \emph{projected jumps} modification in \cite{LS16}. 

For our proposed discretizations, reducing the lifting order to $l=k-1$ implicitly includes a projection of the facet jumps onto the desired space as well and we can therefore also reduce the order of the facet unknowns to $k-1$ without further modifications. In the following, we call the simultaneous reduction of the facet and lifting degree to $k-1$ a \emph{reduced} method, e.g.~a \emph{reduced fully non-conforming HDG} and a \emph{reduced $H(\div)$-conforming HDG} method.
To obtain this improved efficiency also for the normal component, we also consider an \emph{optimized\footnotemark HDG} method. Here the finite element space, which we denote by $\mathbb{BDM}_{k}^{-}$ as discussed in \cref{rem:HdivHDG}, has so-called \emph{relaxed $H(\div)$-conformity}. In this case, we set $s_n(\cdot,\cdot) = 0$ to avoid an additional penalty on the highest order normal jump. 
\footnotetext{with respect to the computational complexity; we stress that these methods do not yield optimal convergence rates with \cref{thm:convEstimates}.}

An overview of the different discretizations together with their associated costs is given in \cref{table:discretizations}. The analysis from \cref{subsec:CA:convergenceEstimates} yields optimal converges rates for the fully-nonconforming and the $H(\div)$-conforming HDG method, but not for the reduced methods where we choose the lifting degree $l = k-1$.  

\begin{table}[!htbp]
   \begin{tabular}{lcccccc}
      \toprule
      HDG method & \multicolumn{3}{c}{discrete spaces} & \multicolumn{3}{c}{associated costs} \\%
      \cmidrule[0.4pt](r{0.125em}){1-1}%
      \cmidrule[0.4pt](lr{0.25em}){2-4}%
      \cmidrule[0.4pt](lr{0.25em}){5-7}%
      & $\XT$ & $\XF$ & lifting & ndofs & ncdofs& nze \\%
      \cmidrule[0.4pt](r{0.125em}){1-1}%
      \cmidrule[0.4pt](lr{0.125em}){2-2}%
      \cmidrule[0.4pt](lr{0.125em}){3-3}%
      \cmidrule[0.4pt](lr{0.125em}){4-4}%
      \cmidrule[0.4pt](lr{0.125em}){5-5}%
      \cmidrule[0.4pt](lr{0.125em}){6-6}%
      \cmidrule[0.4pt](lr{0.125em}){7-7}%
      full & $[\mathbb{P}^k(\Tn)]^d$ & $[\mathbb{P}^k(\Fn)]^d$ & $[\mathbb{P}^k(\Tn)]^d$ & 124 & 20 & 784 \\
      red.~full& $[\mathbb{P}^k(\Tn)]^d$ & $[\mathbb{P}^{k-1}(\Fn)]^d$ & $[\mathbb{P}^{k-1}(\Tn)]^d$ & 74 & 10 & \textbf{196} \\
      $H(\div)$& $\mathbb{BDM}^k(\Tn)$ & $[\mathbb{P}^{k,\text{tang}}(\Fn)]^d$ & $[\mathbb{P}^k(\Tn)]^d$ & 88 &  20 & 784\\
      red.~$H(\div)$& $\mathbb{BDM}^k(\Tn)$ & $[\mathbb{P}^{k-1,\text{tang}}(\Fn)]^d$ & $[\mathbb{P}^{k-1}(\Tn)]^d$ & 51 &  15 & 441 \\
      optimized & $\mathbb{BDM}^{-}_k(\Tn)$ & $[\mathbb{P}^{k-1,\text{tang}}(\Fn)]^d$ & $[\mathbb{P}^{k-1}(\Tn)]^d$ & 56 & 10 & \textbf{196}\\
      \bottomrule
   \end{tabular}
   \caption{We compare different HDG methods in terms of the associated computational costs measured through the number of degrees of freedom (ndofs), the number of coupling degrees of freedom (ncdofs), and the number of non-zero entries in the system matrices (nze) for polynomial degree $k = 1$ and a mesh with $6$ elements. The \emph{red. full} and the \text{optimzed} method have significantly fewer nzes than the \emph{full HDG} method. This reduction becomes less pronounced for higher polynomial degree.}
   \label{table:discretizations}
\end{table}

 For the experiments carried out below, we consider background flows of the following form
\begin{equation}\label{eq:numex:BFlows}
   \bflow_{\eta} \coloneqq \eta c_{\bflow} \begin{pmatrix}
      -y \\ x
   \end{pmatrix},
\end{equation}
where $\eta \in W^{1,\infty}$ is chosen in the specific experiment and the parameter $c_{\bflow} \in \mathbb{R}$ controls the Mach number of the flow. For all experiments, we restrict ourselves to the case where the gravitational potential $\phi$ is constant.

\subsection{Computational aspects}\label{subsec:numex:compasp}
In the next two remarks, we briefly address computational aspects of the implementation of the lifting.
\begin{remark}[Implementation of the lifting operator]\label[remark]{rem:numex:liftingImpl}
   In practice, we implement the lifting operator $\bRl$ through a mixed formulation adding an auxiliary equation and variable. 
   For all $\u_n, \u_n' \in \Xn$, we need to form
   \begin{equation}\label{eq:implementation:ddiffb}
      \begin{aligned}
         \spl \rho \ddiffb \u_n, \ddiffb \u_n' \spr_{\Tn} = &\spl \rho \diffb \u_n, \diffb \u_n' \spr_{\Tn} + \spl \rho \bRl \u_n, \diffb \u_n' \spr_{\Tn} + \spl \rho \diffb \u_n, \bRl \u_n' \spr_{\Tn} \\ 
         &+ \spl \rho \bRl \u_n, \bRl \u_n' \spr_{\Tn}.
      \end{aligned}
   \end{equation}
   For the mixed terms, we obtain by definition \eqref{eq:VecLift} that
   \begin{equation}
      \spl \rho \bRl \u_n, \diffb \u_n' \spr_{\Tn} \!+\! \spl \rho\diffb \u_n, \bRl \u_n' \spr_{\Tn} \!=\! - \spl\rho \hdgjump{\u_n}_{\bflow}, \diffb \u_n' \spr_{\partial \Tn} \!- \spl \rho \diffb \u_n, \hdgjump{\u_n'}_{\bflow} \spr_{\partial \Tn}\!.\! 
   \end{equation}
   For the remaining term, we introduce an auxiliary variable $\bm{r} = \bRl \u_n \in [\mathbb{P}^l(\Tn)]^d$ with the defining that $\bm{r}$ fulfills $\spl \rho \bm{r}, \bm{s} \spr_{\Tn} = - \spl \rho \hdgjump{\u_n}_{\bflow}, \bm{s} \spr_{\partial \Tn}$ for all $\bm{s} \in [\mathbb{P}^l(\Tn)]^d$. Then, we have   
   \begin{equation}
      \spl \rho \bRl \u_n, \bRl \u_n' \spr_{\Tn} = \spl \rho \bm{r}, \bRl \u_n' \spr_{\Tn} = \overline{\spl \rho \bRl \u_n', \bm{r} \spr_{\Tn}} = - \spl \rho \bm{r},  \hdgjump{\u_n'}_{\bflow} \spr_{\partial \Tn}, 
   \end{equation} 
   Thus, we can implement the term \eqref{eq:implementation:ddiffb} through a mixed formulation with the auxiliary variable $\bm{r}$ (and corresponding test function $\bm{s}$).
   As discussed in \cref{subsec:method:discretization} the scalar lifting operator $R^l$ is only introduced for notational convenience \emph{in the analysis} and by definition of the stabilization term $s_n$, we are considering a SIP method for the diffusion term. Thus, we do not have to implement the scalar lifting operator $R^l$ explicitly.
\end{remark}

\begin{remark}[Computational costs associated with the lifting operator]\label[remark]{rem:numex:LiftingCosts}
   In a DG setting, the implementation of lifting operators is usually associated with higher computational costs in the resulting linear systems because the lifting operator introduces new, less local couplings compared to classical SIP operators. 
   In contrast, the support of the HDG-lifting operator is local since the volume unknowns only couple through the facet unknowns indirectly.
   Note that especially $\bm{r}$ and $\bm{s}$ in \cref{rem:numex:liftingImpl} only occur locally on each element and can be eliminated locally.
   Thus, in an HDG setting, the implementation of the lifting operators leads to similar computational costs than the implementation of a corresponding SIP variant. To visualize the associated computational costs, we consider the sparsity pattern of the respective system matrices in \cref{fig:sparsitypattern}.
   \begin{figure}[!htbp]
      \begin{center}
         \begin{tikzpicture}[scale=0.89]  
            \node (A1) at (-6,0) { \includegraphics[width=0.21\textwidth,keepaspectratio]{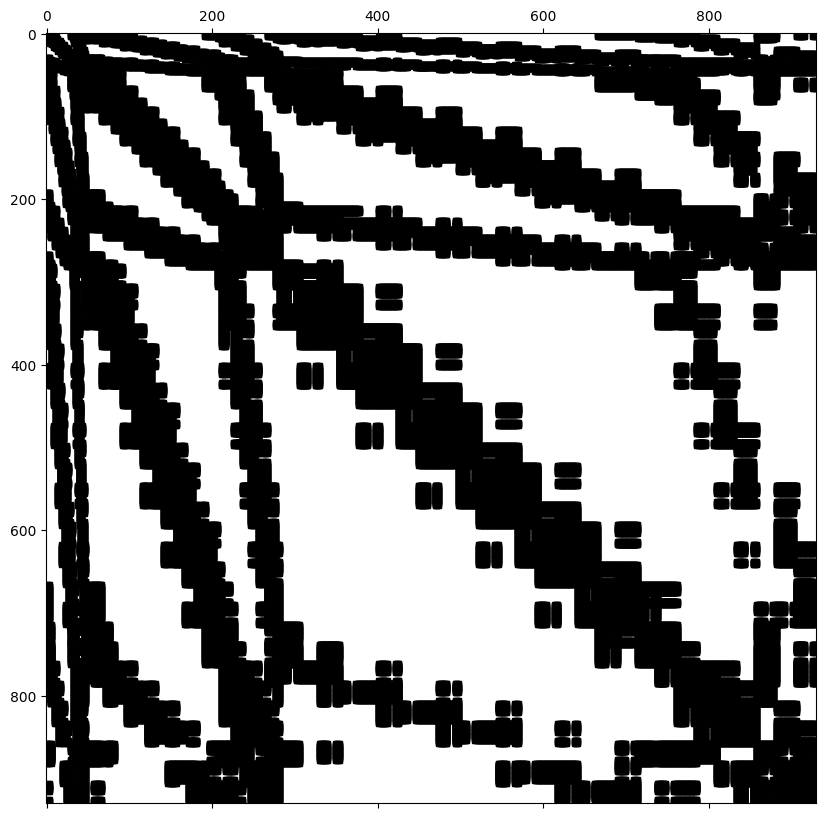}};
            \node (B1)at (-2,0) { \includegraphics[width=0.21\textwidth,keepaspectratio]{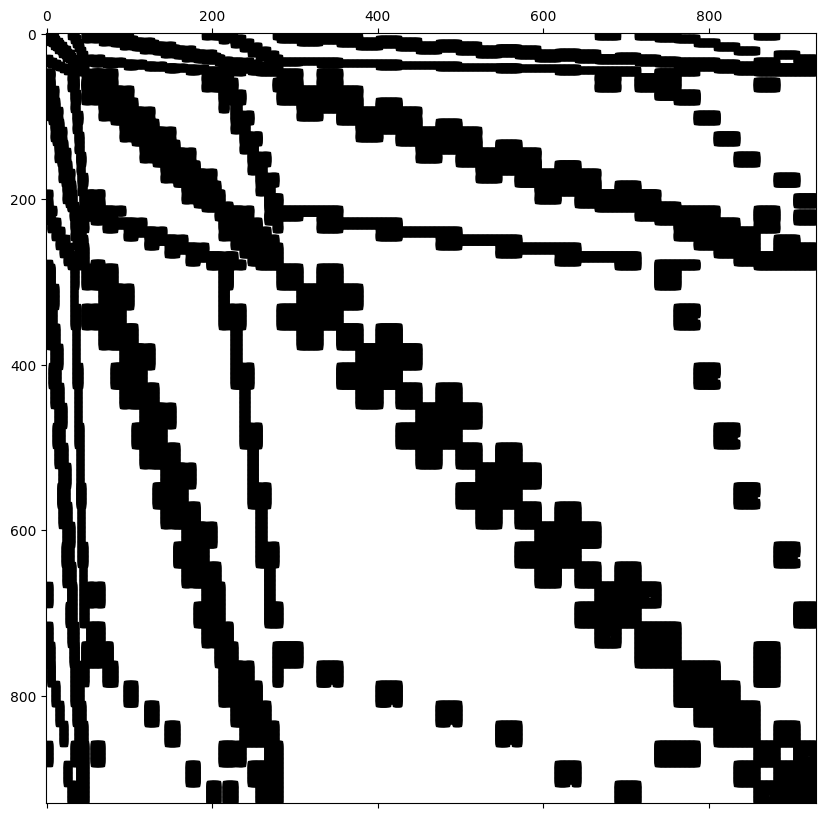}};
            \node (C1) at (2,0) { \includegraphics[width=0.21\textwidth,keepaspectratio]{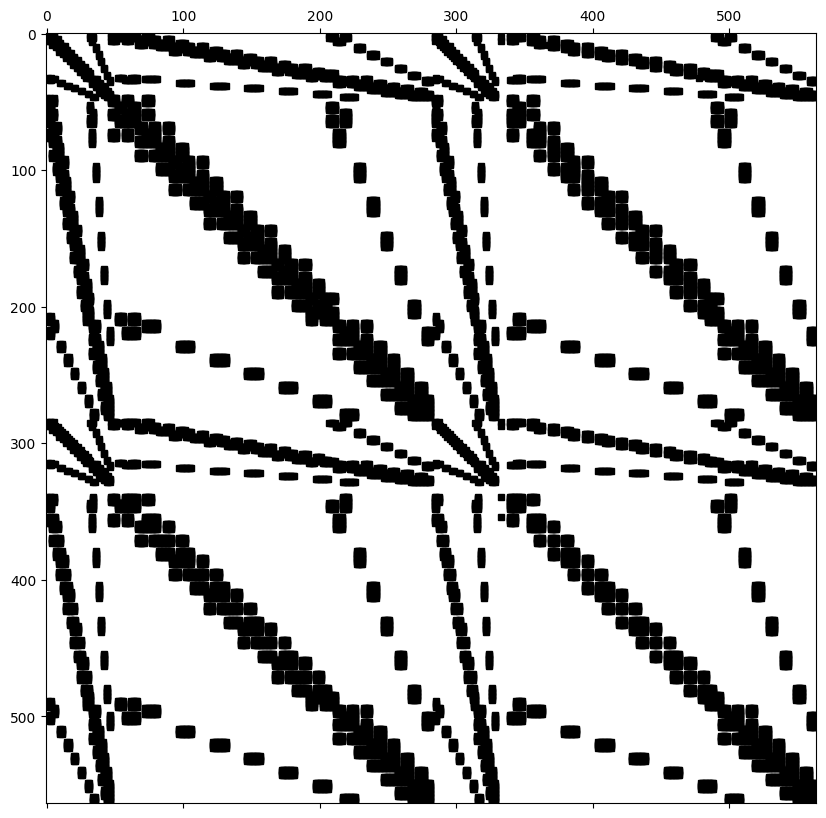}};
            \node (D1) at (6,0) { \includegraphics[width=0.21\textwidth,keepaspectratio]{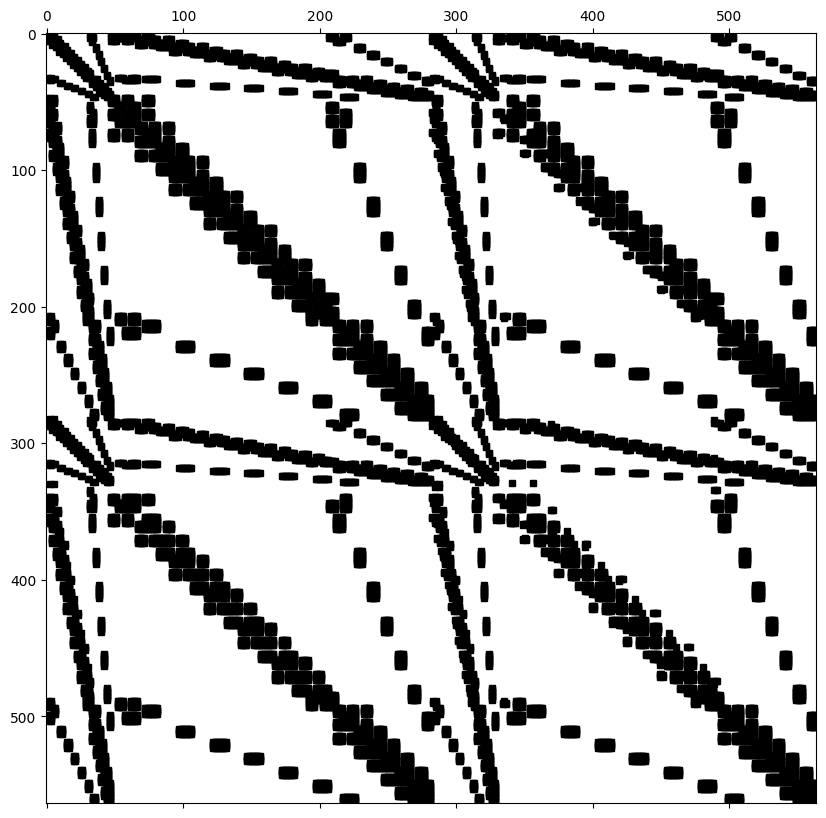}};
            \node[below=0.15cm of A1] (A) {\footnotesize $H(\div)$-DG+Lift};
            \node[below =0.1cm of A] {\footnotesize $\texttt{nzes:}\ 235,036$};
            \node[below=0.15cm of B1] (B) {\footnotesize $H(\div)$-DG+SIP};
            \node[below =0.1cm of B] {\footnotesize $\texttt{nzes:}\ 133,652$};
            \node[below=0.15cm of C1] (C) {\footnotesize $H(\div)$-HDG+Lift};
            \node[below =0.1cm of C] {\footnotesize $\texttt{nzes:}\ 30,096$};
            \node[below=0.15cm of D1] (D) {\footnotesize $H(\div)$-HDG+SIP};
            \node[below =0.1cm of D] {\footnotesize $\texttt{nzes:}\ 30,096$};
        \end{tikzpicture}
      \end{center}
      \caption{We compare the sparsity pattern of the stiffness matrix obtained with the following four methods: $H(\div)$-conforming DG with lifting stabilization  (left), $H(\div)$-conforming DG with SIP (middle-left), $H(\div)$-conforming HDG with lifting stabilization (middle-right), and $H(\div)$-conforming HDG with SIP (right). For the HDG methods, we use static condensation and for the DG method with the lifting operator, we apply the Schur complement to eliminate the unknowns associated with the lifting operator. In the HDG setting, both methods lead to the same number of non-zero entries ($\texttt{nzes}$) (even though the couplings differ slightly), whereas in the DG setting, the lifting operator almost doubles the number of non-zero entries.
      For the computations, we chose a mesh with $27$ elements and the polynomial degree $k = 5$.}
      \label{fig:sparsitypattern}      
   \end{figure}
\end{remark}

\subsection{Convergence studies}\label{subsec:numex:conv}
We consider the unit disk $\dom = \{ x \in \mathbb{R}^2 : \Vert x \Vert < 1 \}$ and choose the parameters 
\begin{subequations}\label{eq:Numex:parameters1}
   \begin{alignat}{3}
      \rho &= \sqrt{10/\pi} \exp(-10(x^2+y^2)), \quad &&c_s^2 = 1.44 + 0.25\rho, \quad &&\omega = 0.78 \times 2\pi, \\
      \gamma &= 0.1, &&\Omega =(0,0), && p =1.44\rho + 0.08\rho^2. 
   \end{alignat} 
\end{subequations}
While these parameters are artificially chosen, the density and the sound speed are modeled to behave similarly (though less extreme) than the respective parameters in the Sun. We consider the background flow $\bflow_{c_s}$ given by \eqref{eq:numex:BFlows} and note that $\div (\rho \bflow_{c_s}) = 0$, $\bflow_{c_s} \cdot \nv = 0$ and $\Vert c_s^{-1} \bflow_{c_s} \Vert_{\bL^\infty} = c_{\bflow}$. The source term $\bff$ is calculated such that the exact solution is given by
\begin{equation}\label{eq:num:refsol}
   \u_{\text{ex}} = \frac{1}{\rho} \sin(x^2+y^2) \sin((x^2+y^2)-1) \begin{pmatrix}
      +(1+i)g \\ -(1+i)g
   \end{pmatrix},
\end{equation}
where $g = \sqrt{\alpha/\pi} \exp(-\alpha(x^2+y^2))$, $\alpha = \log(10^9)$, is a Gaussian. 

In \cref{fig:ConvergenceStudy:RefSol}, we compare the discretization error in the $\Vert \cdot \Vert_{\X(\Tn)}$-norm of the methods from \cref{table:discretizations}. As expected by 
\cref{thm:convEstimates}, the fully non-conforming and $H(\div)$-conforming HDG methods converge with optimal order. Additionally, the reduced methods converge with optimal order as well, even though these cases are not covered by \cref{thm:convEstimates}.
\changed{The absolute error of the optimized HDG method is larger than the error of the other methods and the order of convergences seems to be reduced, however the degrees of freedom are reduced significantly\footnotemark.}
For the fully non-conforming methods, we observe that the stabilization parameter $\alpha = 100$ seems to be more robust than $\alpha = 10$, where we observe a longer preasymptotic phase.
Overall, the numerical results confirm the theoretical results from \cref{thm:convEstimates}, and suggest that the dependence of the convergence order on the lifting degree $l$ might be improved in certain cases, cf.~\cref{rem:CA:improvement}.

\footnotetext{For $k = 4$ on the finest mesh level, we have the following number of non-zero matrix entries for the system matrix: $\sim 5 \cdot 10^6$ (full HDG \& $H(\div)$-conforming HDG), $\sim 4\cdot 10^6$ (red. $H(\div)$-conforming HDG), $\sim 3 \cdot 10^6$ (optimized HDG).}


 \begin{figure}[!htbp]
   \begin{center}\hspace*{-0.3cm}
       \begin{tikzpicture}[scale=0.78, spy using outlines={circle, magnification=4, size=1cm, connect spies}]
           \begin{groupplot}[%
               group style={%
               group size=2 by 1,
               horizontal sep=0cm,
               vertical sep=0.1cm,
               },
           ymajorgrids=true,
           grid style=dashed,
           ymin=1e-7,ymax=2e1,
           ]       
           \nextgroupplot[width=9cm,height=7cm,domain=0:4,xmode=linear,ymode=log, 
                          xlabel={refinement level $L$}, 
                          title={$k = 3$}, 
                          ylabel={$\Vert \u_{\text{ex}} - \u_n \Vert_{\X(\Tn)}$}, 
                          cycle list name=ConvPlot, 
               xtick={0,1,2,3,4}, minor tick style = {white}, legend style={legend columns=6, draw=none,nodes={scale=.75}}, 
               legend to name=named
           ]
            \addplot+[discard if not={order}{3}, discard if not={method}{HdivHDG},line width=1pt] table [x=L, y=wxerror, col sep=comma] {data/ConvergenceStudy.csv};
            \addplot+[discard if not={order}{3}, discard if not={method}{redHdivHDG},line width=1pt] table [x=L, y=wxerror, col sep=comma] {data/ConvergenceStudy.csv};
            \addplot+[discard if not={order}{3}, discard if not={method}{fullHDG},discard if not={lamb}{-100},line width=1pt] table [x=L, y=wxerror, col sep=comma] {data/ConvergenceStudy.csv};
            \addplot+[discard if not={order}{3}, discard if not={method}{redfullHDG},discard if not={lamb}{-100},line width=1pt] table [x=L, y=wxerror, col sep=comma] {data/ConvergenceStudy.csv};
            \addplot+[discard if not={order}{3}, discard if not={method}{optHDG},line width=1pt] table [x=L, y=wxerror, col sep=comma] {data/ConvergenceStudy.csv};

            \addplot+[discard if not={order}{3}, discard if not={method}{fullHDG},discard if not={lamb}{-10},line width=1pt,dashed] table [x=L, y=wxerror, col sep=comma] {data/ConvergenceStudy.csv};
            \addplot+[discard if not={order}{3}, discard if not={method}{redfullHDG},discard if not={lamb}{-10},line width=1pt,dashed] table [x=L, y=wxerror, col sep=comma] {data/ConvergenceStudy.csv};

           \addplot[gray, dashed, domain=0:4] {(7.5*(1/2^(3))^(x+0.5))};
           \addplot[gray, dashed, domain=0:4] {(0.15*(1/2^(3))^(x+0.5))};

           
           \spy[gray,size=1.25cm,opacity=.8] on (5.30,1.23) in node [fill=white] at (6,4.5);

           \legend{$H(\div)$-HDG, red.~$H(\div)$-HDG, full HDG, red.~full HDG, opt.~HDG,,,$\mathcal{O}(h^k)$}

           \nextgroupplot[width=9cm,height=7cm,domain=0:4,xmode=linear,ymode=log, 
                          xlabel={refinement level $L$}, 
                          title={$k = 4$}, 
                          ylabel={}, 
                          cycle list name=ConvPlot, 
               xtick={0,1,2,3,4}, minor tick style = {white}, yticklabels={,,}
           ]
            \addplot+[discard if not={order}{4}, discard if not={method}{HdivHDG},line width=1pt] table [x=L, y=wxerror, col sep=comma] {data/ConvergenceStudy.csv};
            \addplot+[discard if not={order}{4}, discard if not={method}{redHdivHDG},line width=1pt] table [x=L, y=wxerror, col sep=comma] {data/ConvergenceStudy.csv};
            \addplot+[discard if not={order}{4}, discard if not={method}{fullHDG},discard if not={lamb}{-100},line width=1pt] table [x=L, y=wxerror, col sep=comma] {data/ConvergenceStudy.csv};
            \addplot+[discard if not={order}{4}, discard if not={method}{redfullHDG},discard if not={lamb}{-100},line width=1pt] table [x=L, y=wxerror, col sep=comma] {data/ConvergenceStudy.csv};
            \addplot+[discard if not={order}{4}, discard if not={method}{optHDG},line width=1pt] table [x=L, y=wxerror, col sep=comma] {data/ConvergenceStudy.csv};
            
            \addplot+[discard if not={order}{4}, discard if not={method}{fullHDG},discard if not={lamb}{-10},line width=1pt,dashed] table [x=L, y=wxerror, col sep=comma] {data/ConvergenceStudy.csv};
            \addplot+[discard if not={order}{4}, discard if not={method}{redfullHDG},discard if not={lamb}{-10},line width=1pt,dashed] table [x=L, y=wxerror, col sep=comma] {data/ConvergenceStudy.csv};
          
           \addplot[gray, dashed, domain=0:4] {(2.5*(1/2^(4))^(x+0.5))};
           \addplot[gray, dashed, domain=0:4] {(0.1*(1/2^(4))^(x+0.5))};


           \end{groupplot}
            \begin{axis}[
            ybar,
            footnotesize, 
            anchor=north east, 
            scale only axis, 
            cycle list name= barcolors,
            enlarge x limits=2.5,
            at={(2.75cm,2.15cm)}, 
            xticklabels={},
            xtick={},
            xmajorticks=false,
            ymin=4.5e4,ymax=8.7e4,
            bar width=0.2cm,
            axis line style=gray!85,
            width=2cm
            ]
               \addplot+[ybar,discard if not={order}{3}, discard if not={L}{4},discard if not={lamb}{0},,discard if not={method}{HdivHDG}] table [x=methodInd,y=ncdofs, col sep=comma] {data/ConvergenceStudy.csv};
               \addplot+[ybar,discard if not={order}{3}, discard if not={L}{4},discard if not={lamb}{0},,discard if not={method}{redHdivHDG}] table [x=methodInd,y=ncdofs, col sep=comma] {data/ConvergenceStudy.csv};
               \addplot+[ybar,discard if not={order}{3}, discard if not={L}{4},discard if not={lamb}{-100},,discard if not={method}{fullHDG}] table [x=methodInd,y=ncdofs, col sep=comma] {data/ConvergenceStudy.csv};
               \addplot+[ybar,discard if not={order}{3}, discard if not={L}{4},discard if not={lamb}{-100},,discard if not={method}{redfullHDG}] table [x=methodInd,y=ncdofs, col sep=comma] {data/ConvergenceStudy.csv};
               \addplot+[ybar,discard if not={order}{3}, discard if not={L}{4},discard if not={lamb}{0},,discard if not={method}{optHDG}] table [x=methodInd,y=ncdofs, col sep=comma] {data/ConvergenceStudy.csv};
         
            \end{axis}
            \begin{axis}[
               ybar,
               footnotesize, 
               anchor=north east, 
               scale only axis, 
               cycle list name= barcolors,
               enlarge x limits=2.5,
               at={(10.15cm,2.15cm)}, 
               xticklabels={},
               xtick={},
               xmajorticks=false,
               ymin=0.45e5,ymax=1.15e5,
               bar width=0.2cm,
               axis line style=gray!85,
               width=2cm
               ]
                  \addplot+[ybar,discard if not={order}{4}, discard if not={L}{4},discard if not={lamb}{0},,discard if not={method}{HdivHDG}] table [x=methodInd,y=ncdofs, col sep=comma] {data/ConvergenceStudy.csv};
                  \addplot+[ybar,discard if not={order}{4}, discard if not={L}{4},discard if not={lamb}{0},,discard if not={method}{redHdivHDG}] table [x=methodInd,y=ncdofs, col sep=comma] {data/ConvergenceStudy.csv};
                  \addplot+[ybar,discard if not={order}{4}, discard if not={L}{4},discard if not={lamb}{-100},,discard if not={method}{fullHDG}] table [x=methodInd,y=ncdofs, col sep=comma] {data/ConvergenceStudy.csv};
                  \addplot+[ybar,discard if not={order}{4}, discard if not={L}{4},discard if not={lamb}{-100},,discard if not={method}{redfullHDG}] table [x=methodInd,y=ncdofs, col sep=comma] {data/ConvergenceStudy.csv};
                  \addplot+[ybar,discard if not={order}{4}, discard if not={L}{4},discard if not={lamb}{0},,discard if not={method}{optHDG}] table [x=methodInd,y=ncdofs, col sep=comma] {data/ConvergenceStudy.csv};
            
               \end{axis}

       \end{tikzpicture}
       \pgfplotslegendfromname{named}
   \end{center}\vspace*{-0.1cm}
   \caption{Convergence of the methods listed in \cref{table:discretizations} against \eqref{eq:num:refsol} for polynomial degrees $k = 3$ and $k = 4$ with Mach number $\Vert c_s^{-1} \bflow_{c_s} \Vert^2_{\bL^\infty} = 0.25$. For the fully non-conforming methods, we consider the choices $\alpha \in \{10 k^2,100 k^2\}$ (dashed, solid). The error is measured in the $\X(\Tn)$-norm. In the embedded bar charts we show the number of coupled degrees of freedom for each method at the last refinement level $L=4$.}
   \label{fig:ConvergenceStudy:RefSol}
 \end{figure}

\subsection{Mach number robustness}\label{subsec:numex:mach}
As formalized in \cref{assumption:MachNumber}, the stability of the method depends on the Mach number of the background flow. Here, we want to study the influence of the Mach number on the error of the discretization. In particular, we want to compare the methods considered in this manuscript with the $\bH^1$-conforming discretization introduced in \cite{HLS22H1}, and therefore the assumptions on the Mach number from \eqref{eq:assumption:MachNumber} and \eqref{eq:MnAssump:H1}.

We consider the parameters given by \eqref{eq:Numex:parameters1}, but choose the right-hand side independent of the background flow: 
\begin{equation}
   \bff(x,y) \coloneqq \frac{1}{2} \sqrt{55/\pi} \exp(-55((x-0.35)^2 + (y-0.35)^2)) \begin{pmatrix}
      1 \\ 0
   \end{pmatrix}.
\end{equation}

To measure the discretization error, we calculate a reference solution on a fine mesh ($h \sim 0.5^5$) with high polynomial degree ($k = 7$) with the $H(\div)$-conforming HDG method, cf.~\cref{table:discretizations}. Then, we solve the problem on a coarser mesh ($h \sim 0.5^4$) with polynomial degree $k = 5$ using the $H(\div)$-conforming HDG and the $\bH^1$-conforming discretization. \changed{With the considered parameters, we have that $\theta = 0$\footnote{\changed{One of the eigenvalues of $\matii$ is negative on $\dom$ such that $C_{\matii} = 0$. The replication data provides a script that calculates the eigenvalues of the matrix symbolically.}} such that \eqref{eq:assumption:MachNumber} essentially only depends on the stability constant $\DBconst$.} To compare the resulting discretization error with the approximation quality of the discrete space, we calculate the best-approximation of the reference solution with respect to the $\X(\Tn)$-inner product, i.e.~we compute $\Pi^{\X_n} \u_{\text{ref}} \in \Xn$ such that $\spl \Pi^{\X_n} \u_{\text{ref}}, \v_n \spr_{\X(\Tn)} = \spl \u, \v_n \spr_{\X(\Tn)}$ for all $\v_n \in \Xn$. 

The results for three different background flows $\bflow_{\eta}$, $\eta \in \{1,c_s,c_s/\rho\}$, modeled by \eqref{eq:numex:BFlows} are displayed in \cref{fig:MNRobustness}. For the flows $\bflow_1$ and $\bflow_{c_s}$, we observe that the discretization error of the $H(\div)$-conforming HDG method is close to the best-approximation error until $\Vert c_s^{-1} \bflow \Vert_{\bL^\infty} \approx 1.0$. In contrast, the discretization error of the $\bH^1$-conforming method starts to deviate from the best-approximation error as soon as the Mach number approaches $0.3$. \changed{Curiously, both errors seem to increase for lower Mach numbers, but decreasing the mach number further does not lead to a blowup of the error\footnotemark. Since both methods behave similarly, this might be an artifact of the specific example, in particular, since we do not observe this behavior for the background flow $\bflow_{c_s/\rho}$.} Altogether, the $H(\div)$-conforming method seems to be more robust with respect to the increasing Mach number than the $\bH^1$-conforming method in those examples.

\footnotetext{\changed{We consider only the background flow $\bflow_{1}$: For the $H^1$-conforming method, the error is around $1.9\cdot 10^{-4}$ for $\Vert c_s^{-1} \bflow \Vert_{\bL^\infty} = 1\cdot 10^{-8}$ and $5.1\cdot 10^{-4}$ for $\Vert c_s^{-1} \bflow \Vert_{\bL^\infty} = 0.05$, whereas for the $H(\div)$-conforming method, the errors are $8.5\cdot 10^{-4}$ and $2.5\cdot 10^{-3}$.}}

For the background flow $\bflow_{c_s/\rho}$, the methods produce much more similar results and in particular, the error does not increase drastically once the Mach number exceeds $1.0$.

 \begin{figure}[!htbp]
   \begin{center}\hspace*{-0.3cm}
      \begin{tikzpicture}[scale=0.65]
          \begin{groupplot}[%
              group style={%
              group size=3 by 1,
              horizontal sep=0cm,
              vertical sep=2.5cm,
              },
          ymajorgrids=true,
          grid style=dashed,
          ymin=6e-6, ymax=5e-2,
          ]       
          \nextgroupplot[width=9cm,height=7cm,domain=0:4,xmode=linear,ymode=log, 
                         xlabel={Mach number $\Vert c_s^{-1} \bflow \Vert_{\bL^\infty}$}, 
                         title={$\bflow_{1}$}, 
                         ylabel={$\Vert \u_{\text{ref}} - \u_n \Vert_{\X(\Tn)}$}, 
                         cycle list name=MachNum, 
              minor tick style = {white}, legend style={legend columns=5, draw=none,nodes={scale=.8}}, 
              legend to name=named
          ]

          \addplot+[discard if not={order}{5}, discard if not={method}{H1}, discard if not={flow}{0},discard if not={L}{3},line width=1pt] table [x=Mach, y=wxerror, col sep=comma] {data/MNRobustness.csv};
          \addplot+[discard if not={order}{5}, discard if not={method}{H1}, discard if not={flow}{0},discard if not={L}{3},line width=1pt,dashed] table [x=Mach, y=wX-best, col sep=comma] {data/MNRobustness.csv};
          \addplot+[discard if not={order}{5}, discard if not={method}{HDivHDG}, discard if not={flow}{0},discard if not={L}{3},line width=1pt] table [x=Mach, y=wxerror, col sep=comma] {data/MNRobustness.csv};
          \addplot+[discard if not={order}{5}, discard if not={method}{HDivHDG}, discard if not={flow}{0},discard if not={L}{3},line width=1pt,dashed] table [x=Mach, y=wX-best, col sep=comma] {data/MNRobustness.csv};

           \draw[dashed] (1,1e-8) -- (1,1e0);

          \legend{$\bH^1$,$\X(\Tn)$-best ($\bH^1$), $H(\div)$-HDG, $\X(\Tn)$-best ($H(\div)$-HDG)}

         \nextgroupplot[width=9cm,height=7cm,domain=0:4,xmode=linear,ymode=log, 
           xlabel={Mach number $\Vert c_s^{-1} \bflow \Vert_{\bL^\infty}$}, 
           title={$\bflow_{c_s}$}, 
           ylabel={}, 
           cycle list name=MachNum, 
            minor tick style = {white}, 
            yticklabels={,,}
            ]

            \addplot+[discard if not={order}{5}, discard if not={method}{H1}, discard if not={flow}{1},discard if not={L}{3},line width=1pt] table [x=Mach, y=wxerror, col sep=comma] {data/MNRobustness.csv};
            \addplot+[discard if not={order}{5}, discard if not={method}{H1}, discard if not={flow}{1},discard if not={L}{3},line width=1pt,dashed] table [x=Mach, y=wX-best, col sep=comma] {data/MNRobustness.csv};
            \addplot+[discard if not={order}{5}, discard if not={method}{HDivHDG}, discard if not={flow}{1},discard if not={L}{3},line width=1pt] table [x=Mach, y=wxerror, col sep=comma] {data/MNRobustness.csv};
            \addplot+[discard if not={order}{5}, discard if not={method}{HDivHDG}, discard if not={flow}{1},discard if not={L}{3},line width=1pt,dashed] table [x=Mach, y=wX-best, col sep=comma] {data/MNRobustness.csv};

            \draw[dashed] (1,1e-8) -- (1,1e0);

            \nextgroupplot[width=9cm,height=7cm,domain=0:4,xmode=linear,ymode=log, 
            xlabel={Mach number $\Vert c_s^{-1} \bflow \Vert_{\bL^\infty}$}, 
            title={$\bflow_{c_s/\rho}$}, 
            ylabel={}, 
            cycle list name=MachNum, 
             minor tick style = {white}, 
             yticklabels={,,}
             ]
 
             \addplot+[discard if not={order}{5}, discard if not={method}{H1}, discard if not={flow}{2},discard if not={L}{3},line width=1pt] table [x=Mach, y=wxerror, col sep=comma] {data/MNRobustness.csv};
             \addplot+[discard if not={order}{5}, discard if not={method}{H1}, discard if not={flow}{2},discard if not={L}{3},line width=1pt,dashed] table [x=Mach, y=wX-best, col sep=comma] {data/MNRobustness.csv};
             \addplot+[discard if not={order}{5}, discard if not={method}{HDivHDG}, discard if not={flow}{2},discard if not={L}{3},line width=1pt] table [x=Mach, y=wxerror, col sep=comma] {data/MNRobustness.csv};
             \addplot+[discard if not={order}{5}, discard if not={method}{HDivHDG}, discard if not={flow}{2},discard if not={L}{3},line width=1pt,dashed] table [x=Mach, y=wX-best, col sep=comma] {data/MNRobustness.csv};
 
             \draw[dashed] (1,1e-8) -- (1,1e0);

          \end{groupplot}

      \end{tikzpicture}
      \pgfplotslegendfromname{named}
  \end{center}\vspace*{-0.1cm}
  \caption{$\X(\Tn)$-error of the $\bH^1$-conforming discretization, the $H(\div)$-conforming HDG discretization and the respective best-approximation error for the flows $\bflow_{\eta}$, $\eta \in \{1,c_s,c_s/\rho\}$ modeled after \eqref{eq:numex:BFlows}. We use the coefficient $c_{\bflow}$ to vary the Mach number in between $0.05$ and $1.25$.}
  \label{fig:MNRobustness}
 \end{figure}

\subsection{Comparison with SIP}\label{subsec:numex:SIP}
The lifting operator $\bRl$ allows us to stabilize the directional derivative $\diffb$ without balancing a stabilization parameter against the Mach number of the background flow $\bflow$. In this section, we compare the proposed method with a SIP version to assess its practical relevance. To avoid balancing two stabilization parameters against each other, we only consider the $H(\div)$-conforming HDG method. For the SIP version, we simply replace the term 
\begin{equation*}
   - \spl \rho \dopd \u_n, \dopd \u_n' \spr_{\Tn}
\end{equation*}
by
\begin{align*}
   - &\spl \rho \opd \u_{\tau}, \opd \u_{\tau}' \spr_{\Tn} - i \spl \rho \opd \u_{\tau}, \hdgjump{\u_n'}_{\bflow} \spr_{\partial \Tn} \\
   &- i \spl \hdgjump{\u_n}_{\bflow}, \rho \opd \u_{\tau}' \spr_{\partial \Tn} + \spl \rho \lambda \h^{-1} \hdgjump{\u_n}_{\bflow}, \hdgjump{\u_n'}_{\bflow} \spr_{\partial \Tn},
\end{align*}
where $\lambda > 0$ is a stabilization parameter that has to be chosen sufficiently large to ensure stability. We choose the same examples as considered in \cref{subsec:numex:conv}, where the parameters are given by \eqref{eq:Numex:parameters1} and the reference solution by \eqref{eq:num:refsol}. In \cref{fig:LiftvsSIP}, we compare the discretization error of the lifting stabilized method with the SIP method for stabilization parameters $\lambda \in \{1k^2,10k^2,100k^2\}$ and polynomial degree $k = 5$. We choose the background flow $\bflow_{c_s}$ and consider the Mach numbers $\Vert c_s^{-1} \bflow \Vert_{\bL^\infty} \in \{0.01,0.45\}$. 

The lifting stabilized discretization seems to be more stable and the error is (significantly) smaller. In particular, the choice of a suitable SIP stabilization parameter $\lambda$ seems to depend on the Mach number. It's also worth mentioning that the condition number of the system matrix grows with the stabilization parameter $\lambda$ in the SIP version, which is not the case for the lifting stabilized version. 
Altogether, we conclude that the lifting stabilized version is more robust and reliable than the SIP version, in particular because the computational costs are not significantly higher, cf.~\cref{rem:numex:LiftingCosts}.

\begin{figure}[!htbp]
   \begin{center}
      \begin{tikzpicture}[scale=0.74]
         \begin{groupplot}[%
             group style={%
             group size=2 by 1,
             horizontal sep=0cm,
             vertical sep=0.1cm,
             },
         ymajorgrids=true,
         grid style=dashed,
         ymin=5e-8, ymax=1.2e1,
         ]      

         \nextgroupplot[width=9cm,height=7cm,domain=0:4,xmode=linear,ymode=log, 
                        xlabel={refinement level}, 
                        title={$\Vert c_s^{-1} \bflow_{c_s} \Vert_{\bL^\infty} =0.01$}, 
                        ylabel={$\Vert \u_{\text{ex}} - \u_n \Vert_{\X(\Tn)}$}, 
                        cycle list name=LiftvsSIP, 
             minor tick style = {white}, legend style={legend columns=5, draw=none,nodes={scale=.8}}, 
             legend to name=named
         ]
         \addplot+[discard if not={method}{HdivHDG},discard if not={mn}{0.01},line width=1pt] table [x=L, y=wxerror, col sep=comma] {data/CompareWithSIP.csv};
         \addplot+[discard if not={method}{HdivSIPHDG},discard if not={mn}{0.01},discard if not={lambda}{1.0},line width=1pt] table [x=L, y=wxerror, col sep=comma] {data/CompareWithSIP.csv};
         \addplot+[discard if not={method}{HdivSIPHDG},discard if not={mn}{0.01},discard if not={lambda}{10.0},line width=1pt] table [x=L, y=wxerror, col sep=comma] {data/CompareWithSIP.csv};
         \addplot+[discard if not={method}{HdivSIPHDG},discard if not={mn}{0.01},discard if not={lambda}{100.0},line width=1pt] table [x=L, y=wxerror, col sep=comma] {data/CompareWithSIP.csv};
        
         \addplot[gray, dashed, domain=0:4] {(14*(1/2^(5))^(x+0.5))};

         \legend{lifting, SIP ($\lambda = 1$), SIP ($\lambda = 10$), SIP ($\lambda = 100$), $\mathcal{O}(h^5)$}

         \nextgroupplot[width=9cm,height=7cm,domain=0:4,xmode=linear,ymode=log, 
         xlabel={refinement level}, 
         title={$\Vert c_s^{-1} \bflow_{c_s} \Vert_{\bL^\infty} =0.45$}, 
         ylabel={}, 
         yticklabels={,,},
         cycle list name=LiftvsSIP, 
         minor tick style = {white}, 
         ]
         \addplot+[discard if not={method}{HdivHDG},discard if not={mn}{0.45},line width=1pt] table [x=L, y=wxerror, col sep=comma] {data/CompareWithSIP.csv};
         \addplot+[discard if not={method}{HdivSIPHDG},discard if not={mn}{0.45},discard if not={lambda}{1.0},line width=1pt] table [x=L, y=wxerror, col sep=comma] {data/CompareWithSIP.csv};
         \addplot+[discard if not={method}{HdivSIPHDG},discard if not={mn}{0.45},discard if not={lambda}{10.0},line width=1pt] table [x=L, y=wxerror, col sep=comma] {data/CompareWithSIP.csv};
         \addplot+[discard if not={method}{HdivSIPHDG},discard if not={mn}{0.45},discard if not={lambda}{100.0},line width=1pt] table [x=L, y=wxerror, col sep=comma] {data/CompareWithSIP.csv};
        
         \addplot[gray, dashed, domain=0:4] {(5*(1/2^(5))^(x+0.5))};
         \end{groupplot}

     \end{tikzpicture}
     \pgfplotslegendfromname{named}
 \end{center}\vspace*{-0.1cm}
 \caption{Discretization error measured in the $\Vert \cdot \Vert_{\X(\Tn)}$-norm for the lifting stabilized $H(\div)$-conforming method and the $H(\div)$-conforming SIP variant with $\lambda \in \{1k^2,10k^2,100k^2\}$ for polynomial degree $k = 5$ and Mach numbers $\Vert c_s^{-1} \bflow_{c_s} \Vert_{\bL^\infty} \in \{0.01,0.45\}$. The choice of a suitable penalty parameters $\lambda$ seems to depend on the Mach number and the error of the lifting stabilized method is smaller.}
 \label{fig:LiftvsSIP}
\end{figure}

\subsection{Sun parameters}\label{subsec:numex:sun}
Finally, let us consider a numerical example using the density, sound speed and pressure provided by the \texttt{modelS} \cite{modelS} for the Sun. Due to the extreme variation of these coefficients towards the boundary of the domain, we use special meshes that are finer towards the boundary but more coarse in the interior, see \cref{fig:SunCoefs}. In addition to the parameters given by the \texttt{modelS}, we follow \cite{CD18} and set 
\begin{equation*}
   \omega = 0.003 \cdot 2\pi \cdot R_{\circ}, \quad \gamma = \omega / 100, \quad \Omega = (0,0),
\end{equation*}
where $R_{\circ} \approx 1.0007126$ is the radius of the sun. We choose the right-hand side 
\begin{equation*}
   \bff = 10^7 \begin{pmatrix}
      g \\ 0
   \end{pmatrix}, 
\end{equation*}
where $g(x,y) = \sqrt{(\log(10^6)/0.1^2)/\pi} \exp((-\log(10^6)/0.1^2)((x-0.5)^2+(y-0.5)^2))$ is a Gaussian. We consider the case of a uniform and a non-uniform background flow. Specifically, we consider the flows $\bflow_{1/R_{\circ}}$ and $\bflow_{c_s/R_{\circ}}$ which are of the form \eqref{eq:numex:BFlows} such that we can use the parameter $c_{\bflow}$ to control the Mach number. We note that in both cases, we have that $\div(\rho \bflow) = 0$.

\begin{figure}[!htbp]       
   \centering
   \begin{tikzpicture}[scale=0.9]
    \pgfplotsset{
        scale only axis,
        xmin=-0.1, xmax=1.1,
        y axis style/.style={
            yticklabel style=#1,
            ylabel style=#1,
            y axis line style=#1,
            ytick style=#1
       },
       legend pos = south west,
       width=5cm,
    }
    
    \begin{axis}[
      axis y line*=left,
      y axis style=blue!75!black,
      xlabel={solar radius},
      ylabel={sound speed [m/sec]},
        xmode=linear,
        ymode=log,
        xtick={0,0.25,0.5,0.75,1},
    ]
    \addplot[
    color=blue!75!black,
    mark=none,
    thick,
    ] table[col sep=comma,x=radius,y=soundspeed]{data/modelS.csv}; \label{plot_cs}
    \end{axis}
    
    \begin{axis}[
      axis y line*=right, 
      axis x line=none,
      ytick={100000,1,0.00001},
        ymax=10e5,ymin=0,
      y axis style=red!75!black,
      ylabel={density [kg/m$^3$]},
      ymode=log,
    ]
    \addlegendimage{/pgfplots/refstyle=plot_cs}\addlegendentry{$c_s$}
   \addplot[
    color=red!75!black,
    mark=none,
    thick,
    ] table[col sep=comma,x=radius,y=density]{data/modelS.csv}; \addlegendentry{$\rho$}
    \end{axis}
    \node at (-4,2) {\includegraphics[width=4.5cm,keepaspectratio]{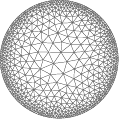}};
\end{tikzpicture}
\caption{The density and sound speed provided by the \texttt{modelS} (right) and an example mesh adapted to these coefficients (left).}
\label{fig:SunCoefs}
\vspace*{-0.2cm}
\end{figure}

To improve the computational efficiency, we use the \emph{optimized HDG} method as described in \cref{table:discretizations}. In \cref{fig:SunEx1,fig:SunEx2}, we display the real part of the $x$-components of the computed solutions for the two backgrounds flows. For the background flow $\bflow_{1/R_{\circ}}$, the computed solution seems to be stable, even when the Mach number exceed $1.0$. In contrast, we observe instabilities in the solutions computed for the second background flow $\bflow_{c_s/R_{\circ}}$ once the Mach number grows large, which is however still in agreement with the results from \cref{thm:weakTcompatibilityFulfilled}.

\begin{figure}[!htbp]
   \centering
   \begin{tikzpicture}
      \node at (-3,3) {\includegraphics[width=5cm,keepaspectratio]{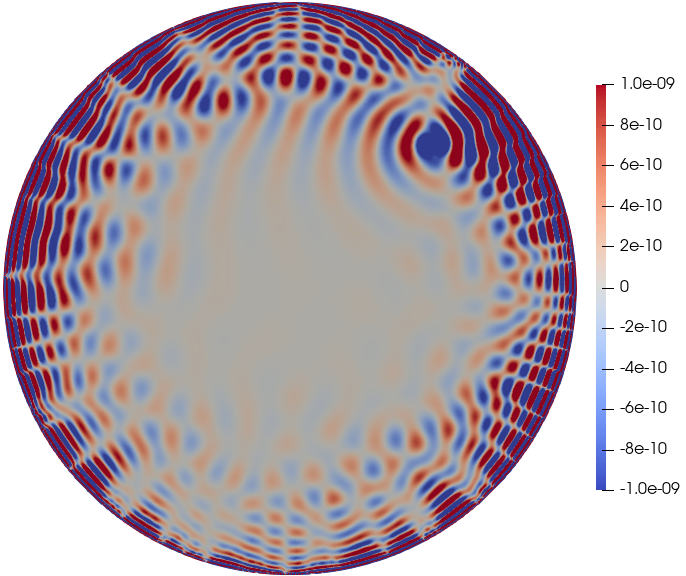}};
      \node at (3,3) {\includegraphics[width=5cm,keepaspectratio]{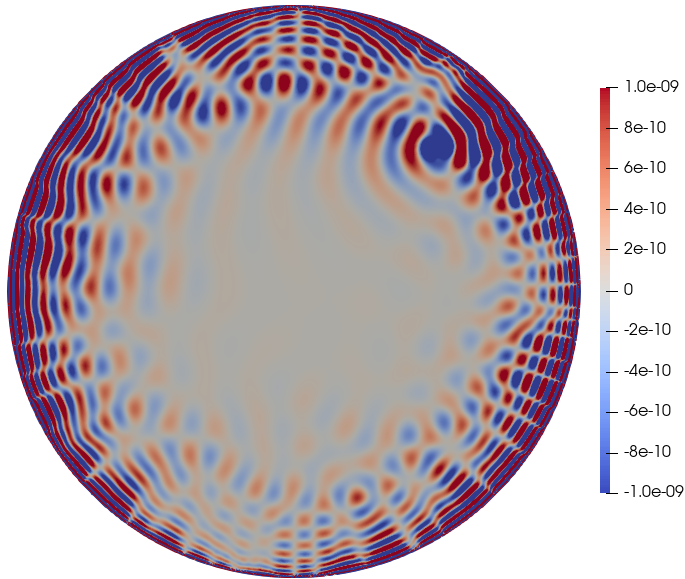}};
      \node at (-3,-2.5) {\includegraphics[width=5cm,keepaspectratio]{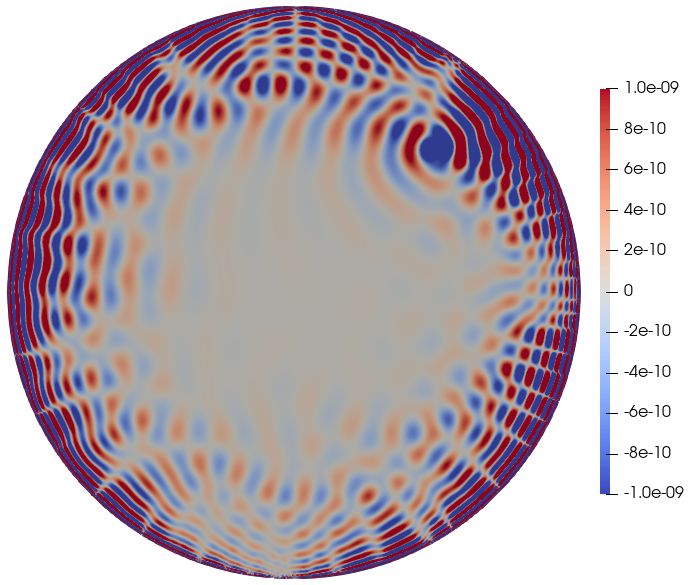}};
      \node at (3,-2.5) {\includegraphics[width=5cm,keepaspectratio]{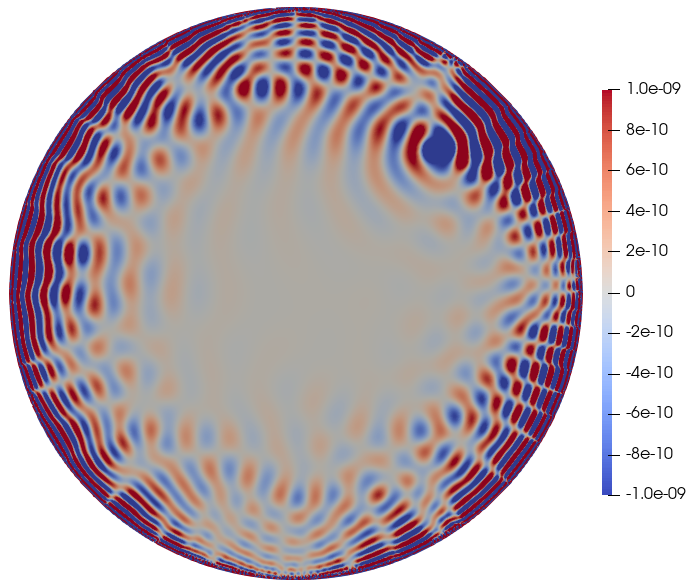}};
      \node at (-3,0.5) {\footnotesize $\Vert c_s^{-1} \bflow_{1/R_{\circ}} \Vert^2_{\bL^\infty} \approx 0.05$};
      \node at (3,0.5) {\footnotesize $\Vert c_s^{-1} \bflow_{1/R_{\circ}} \Vert^2_{\bL^\infty} \approx 0.50$};
      \node at (-3,-5.0) {\footnotesize $\Vert c_s^{-1} \bflow_{1/R_{\circ}} \Vert^2_{\bL^\infty} \approx 1.00$};
      \node at (3,-5.0) {\footnotesize $\Vert c_s^{-1} \bflow_{1/R_{\circ}} \Vert^2_{\bL^\infty} \approx 1.50$};
   \end{tikzpicture}
   \vspace*{-0.2cm}
   \caption{Real part of the $x$-component of the solution computed with the \emph{optimized HDG} method for polynomial degree $k = 7$ for the flow $\bflow_{1/R_{\circ}}$. We consider the Mach numbers $\Vert c_s^{-1} \bflow_{1/R_{\circ}} \Vert^2_{\bL^\infty} \in \{0.05,0.5,1.0,1.5\}$.}
   \label{fig:SunEx1}
   \vspace*{-0.2cm}
\end{figure}

\begin{figure}[!htbp]
   \centering
   \begin{tikzpicture}
      \node at (-3,3) {\includegraphics[width=5cm,keepaspectratio]{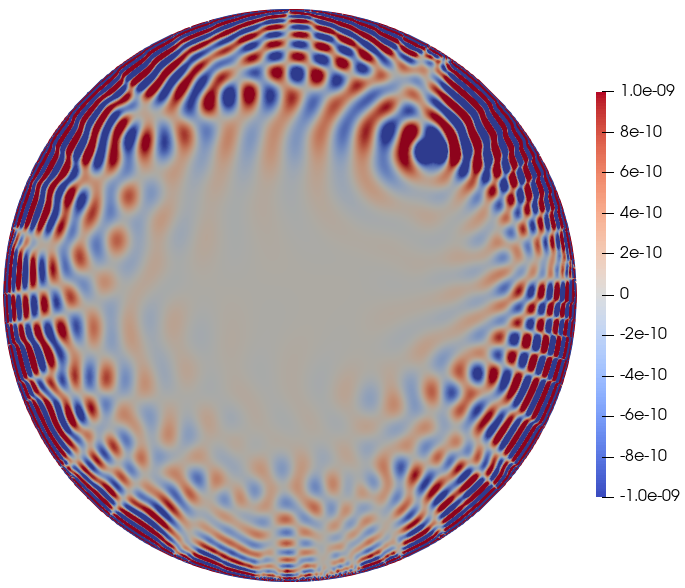}};
      \node at (3,3) {\includegraphics[width=5cm,keepaspectratio]{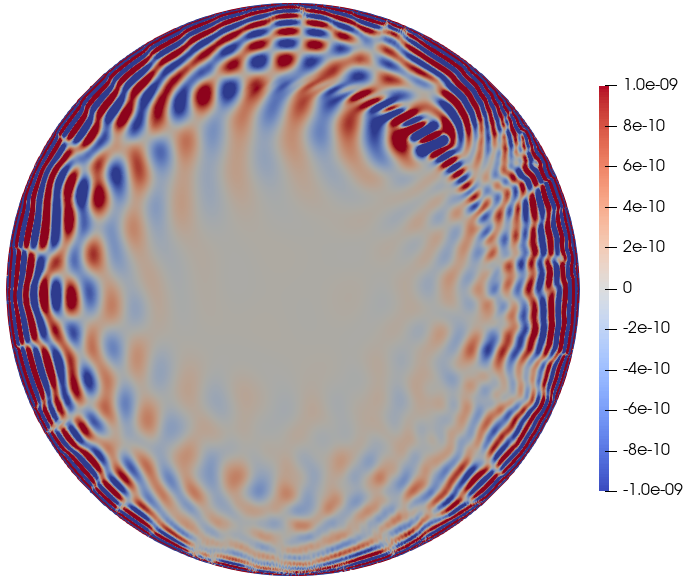}};
      \node at (-3,-2.5) {\includegraphics[width=5cm,keepaspectratio]{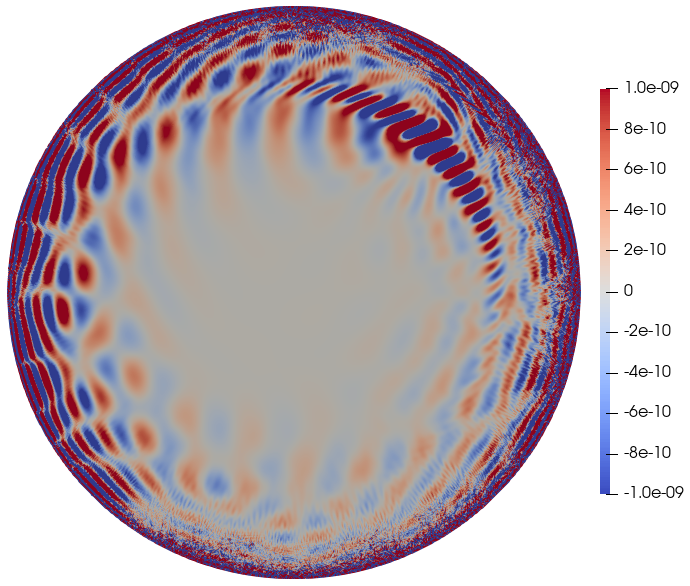}};
      \node at (3,-2.5) {\includegraphics[width=5cm,keepaspectratio]{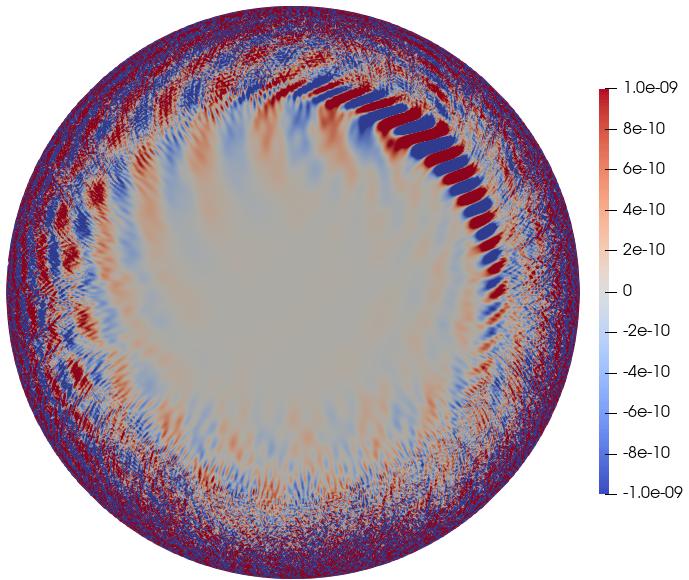}};
      \node at (-3,0.5) {\footnotesize $\Vert c_s^{-1} \bflow_{c_s/R_{\circ}} \Vert^2_{\bL^\infty} \approx 0.05$};
      \node at (3,0.5) {\footnotesize $\Vert c_s^{-1} \bflow_{c_s/R_{\circ}} \Vert^2_{\bL^\infty} \approx 0.50$};
      \node at (-3,-5.0) {\footnotesize $\Vert c_s^{-1} \bflow_{c_s/R_{\circ}} \Vert^2_{\bL^\infty} \approx 1.00$};
      \node at (3,-5.0) {\footnotesize $\Vert c_s^{-1} \bflow_{c_s/R_{\circ}} \Vert^2_{\bL^\infty} \approx 1.50$};
   \end{tikzpicture}
   \vspace*{-0.2cm}
   \caption{Real part of the $x$-component of the solution computed with the \emph{optimized HDG} method for polynomial degree $k = 7$ for the flow $\bflow_{c_s/R_{\circ}}$. We consider the Mach numbers $\Vert c_s^{-1} \bflow_{c_s/R_{\circ}} \Vert^2_{\bL^\infty} \in \{0.05,0.5,1.0,1.5\}$.}
   \label{fig:SunEx2}
   \vspace*{-0.2cm}
\end{figure}

\section*{Acknowledgements}
The first author acknowledges funding from Deutsche Forschungsgemeinschaft (DFG, German Research Foundation), projects 541433971 and 258734477 – SFB 1173 and that part of this work was conducted at the Johann Radon Institute for Computational and Applied Mathematics.
The second and third authors acknowledge funding from DFG SFB 1456 project 432680300.

\bibliographystyle{amsplain}
\bibliography{main}

@article{HLP2021,
  title = {Learned infinite elements},
  author = {T. Hohage and C. Lehrenfeld and J. Preu\ss},
  year = {2021},
  journal = {SIAM J. Sci. Comput.},
  volume = {43},
  number = {5},
  pages = {A3552-A3579},
  doi = {10.1137/20M1381757}
}

@article{HLS22H1,
    author = {Halla, M.  and Lehrenfeld, C. and Stocker, P.},
    title = {{A new T-compatibility condition and its application to the discretization of the damped time-harmonic Galbrun’s equation}},
    journal = {IMA Journal of Numerical Analysis},
    pages = {draf071},
    year = {2025},
    month = {08}
}

@misc{H23Hdiv,
      title={Convergence analysis of nonconform ${H}(\operatorname{div})$-finite elements for the damped time-harmonic {G}albrun's equation}, 
      author={Halla, M.},
      year={arXiv:2306.03496, 2023},
      eprint={2306.03496},
      archivePrefix={arXiv},
      primaryClass={math.NA}
}

@article{HH21,
  author = {Halla, M.  and Hohage, T.},
  year = {2021},
  month = {07},
  pages = {4068-4095},
  title = {On the {W}ell-posedness of the {D}amped {T}ime-harmonic {G}albrun {E}quation and the {E}quations of {S}tellar {O}scillations},
  volume = {53},
  journal = {SIAM Journal on Mathematical Analysis},
  doi = {10.1137/20M1348558}
}

@mastersthesis{Thesis_vB23,
  author = {{van Beeck}, T.},
  title = {{On stable discontinuous Galerkin discretizations for Galbrun’s equation}},
  school = {University of G{\"o}ttingen},
  year = {2023},
  doi = {10.25625/KGHQWV},
  url = {https://doi.org/10.25625/KGHQWV/F6PCXK},
  note = {doi:10.25625/KGHQWV/F6PCXK},
  month = {December}
}

@article{LLS18,
  title = {Hybrid {D}iscontinuous {Galerkin} methods with relaxed {H(div)}-conformity for incompressible flows. Part I},
  author = {Lederer, P. L. and Lehrenfeld, C. and Sch{\"o}berl, J.},
  journal = {SIAM J. Numer. Anal.},
  year = {2018},
  volume = {56},
  issue = {4},
  pages = {2070--2094},
  doi = {10.1137/17M1138078},
  notes = { NAM preprint 2017-05},
  url = {https://arxiv.org/abs/1707.02782}
}

@article{LLS19,
  title = {Hybrid {D}iscontinuous {Galerkin} methods with relaxed {H(div)}-conformity for incompressible flows. Part II},
  author = {Lederer, P. L. and Lehrenfeld, C. and Sch{\"o}berl, J.},
  journal = {ESAIM: M2AN},
  year = {2019},
  volume = {53},
  pages = {503--522},
  doi = {10.1051/m2an/2018054},
  notes = { NAM preprint 2017-05},
  url = {https://arxiv.org/abs/1707.02782}
}

@article{LS16,
  title = {High order exactly divergence-free {H}ybrid {D}iscontinuous {Galerkin} {M}ethods for unsteady incompressible flows},
  author = {C. Lehrenfeld and J. Schöberl},
  journal = {Computer Methods in Applied Mechanics and Engineering },
  volume = {307},
  number = {},
  pages = {339 -- 361},
  year = {2016},
  note = {},
  issn = {0045-7825},
  doi = {http://dx.doi.org/10.1016/j.cma.2016.04.025}
}

@article{Sch97,
	Author = {Sch{\"o}berl, J.},
	Da = {1997/07/01},
	Doi = {10.1007/s007910050004},
	Id = {Sch{\"o}berl1997},
	Isbn = {1432-9360},
	Journal = {Computing and Visualization in Science},
	Number = {1},
	Pages = {41--52},
	Title = {NETGEN An advancing front 2D/3D-mesh generator based on abstract rules},
	Ty = {JOUR},
	Url = {https://doi.org/10.1007/s007910050004},
	Volume = {1},
	Year = {1997}
}

@article{Sch14,
  title={C++ 11 Implementation of Finite Elements in {NGS}olve},
  author={Sch{\"o}berl, J.},
  year={2014},
  publisher={Institute of Analysis and Scientific Computing, TU Wien}
}

@article{BDM87,
author={Brezzi, F. and Douglas, J. and Dur{\'a}n, R. and Fortin, M.},
title={Mixed finite elements for second order elliptic problems in three variables},
journal={Numerische Mathematik},
year={1987},
month={Mar},
day={01},
volume={51},
number={2},
pages={237-250},
issn={0945-3245},
doi={10.1007/BF01396752},
url={https://doi.org/10.1007/BF01396752}
}

@book{Galbrun31,
  author          = {H. Galbrun},
  publisher         = {Gauthier-Villars, Paris},
  title           = {Propagation d'une onde sonore dans l'atmosphre et th\'eorie des zones de silence},
  year            = {1931}
}

@article{90yrsGalbrun,
author = {Maeder, M. and Gabard, G. and Marburg, S.},
year = {2020},
month = {10},
pages = {2050017},
title = {90 Years of {G}albrun’s {E}quation: {A}n {U}nusual {F}ormulation for {A}eroacoustics and {H}ydroacoustics in {T}erms of the {L}agrangian {D}isplacement},
volume = {28},
journal = {Journal of Theoretical and Computational Acoustics},
doi = {10.1142/S2591728520500176}
}

@article{BO09,
  author={Buffa, A. and Ortner, C.},
  journal={IMA Journal of Numerical Analysis}, 
  title={Compact embeddings of broken {S}obolev spaces and applications}, 
  year={2009},
  volume={29},
  number={4},
  pages={827-855},
  doi={10.1093/imanum/drn038}
}

@article{HB21,
  title = {On the well-posedness of {G}albrun's equation},
  journal = {Journal de Mathématiques Pures et Appliquées},
  volume = {150},
  pages = {112-133},
  year = {2021},
  issn = {0021-7824},
  doi = {https://doi.org/10.1016/j.matpur.2021.04.004},
  url = {https://www.sciencedirect.com/science/article/pii/S0021782421000593},
  author = {H\"{a}gg, L. and Berggren, M. },
}

@mastersthesis{Leh10,
  author = {C. Lehrenfeld},
  title = {{Hybrid Discontinuous Galerkin Methods for Incompressible Flow Problems}},
  school = {RWTH Aachen},
  year = {2010},
  doi = {10.25625/O4VBYH},
  url = {https://data.goettingen-research-online.de/file.xhtml?persistentId=doi:10.25625/O4VBYH/LNLFCJ&version=1.0},
  month = {May},
  note = {doi:10.25625/O4VBYH},
}

@article{CGL09,
  title={Unified hybridization of discontinuous {G}alerkin, mixed, and continuous {G}alerkin methods for second order elliptic problems},
  author={Cockburn, B. and Gopalakrishnan, J. and Lazarov, R.},
  journal={SIAM Journal on Numerical Analysis},
  volume={47},
  number={2},
  pages={1319--1365},
  year={2009},
  publisher={SIAM}
}

@article{KCR23,
  title={Convergence to weak solutions of a space-time hybridized discontinuous {G}alerkin method for the incompressible {N}avier-{S}tokes equations},
  author={Kirk, K. and Cesmelioglu, A. and Rhebergen, S.},
  journal={Mathematics of Computation},
  volume={92},
  number={339},
  pages={147--174},
  year={2023}
}

@article{BR97,
title = {A {H}igh-{O}rder {A}curate {D}iscontinuous {F}inite {E}lement {M}ethod for the {N}umerical {S}olution of the {C}ompressible {N}avier–{S}tokes {E}quations},
journal = {Journal of Computational Physics},
volume = {131},
number = {2},
pages = {267-279},
year = {1997},
issn = {0021-9991},
doi = {https://doi.org/10.1006/jcph.1996.5572},
url = {https://www.sciencedirect.com/science/article/pii/S0021999196955722},
author = {F. Bassi and S. Rebay},
}

@article{BMMPR00,
author = {Brezzi, F. and Manzini, G. and Marini, D. and Pietra, P. and Russo, A.},
title = {{Discontinuous Galerkin approximations for elliptic problems}},
journal = {Numerical Methods for Partial Differential Equations},
volume = {16},
number = {4},
pages = {365-378},
doi = {https://doi.org/10.1002/1098-2426(200007)16:4<365::AID-NUM2>3.0.CO;2-Y},
url = {https://onlinelibrary.wiley.com/doi/abs/10.1002/1098-2426%28200007%2916%3A4%3C365%3A%3AAID-NUM2%3E3.0.CO%3B2-Y},
year = {2000}
}

@book{GGK90,
place={Basel},
title={{Classes of linear operators. vol. I}}, 
publisher={Birkhauser}, 
author={Gohberg, I. and Goldberg, S. and Kaashoek, M. A.},
year={1990}}

@article{B99,
  title={The discrete commutator property of approximation spaces},
  author={Bertoluzza, S.},
  journal={Comptes Rendus de l'Acad{\'e}mie des Sciences-Series I-Mathematics},
  volume={329},
  number={12},
  pages={1097--1102},
  year={1999},
  publisher={Elsevier}
}

@article{AHLS22,
  title = {{Robust finite element discretizations for a simplified Galbrun's equation}},
  author = {Alemán, T. and Halla, M.  and Lehrenfeld, C. and Stocker, P.},
  journal = {ECCOMAS 2022},
  year = {2022},
  eprint = {2205.15650},
  url = {https://www.scipedia.com/public/Aleman_et_al_2022a},
  code = {https://doi.org/10.25625/9VR9WT},
  doi = {10.23967/eccomas.2022.206},
  note = {doi:10.23967/eccomas.2022.206}
}

@article{modelS,
author = {J. Christensen-Dalsgaard  and W. Däppen  and S. V. Ajukov  and E. R. Anderson  and H. M. Antia  and S. Basu  and V. A. Baturin  and G. Berthomieu  and B. Chaboyer  and S. M. Chitre  and A. N. Cox  and P. Demarque  and J. Donatowicz  and W. A. Dziembowski  and M. Gabriel  and D. O. Gough  and D. B. Guenther  and J. A. Guzik  and J. W. Harvey  and F. Hill  and G. Houdek  and C. A. Iglesias  and A. G. Kosovichev  and J. W. Leibacher  and P. Morel  and C. R. Proffitt  and J. Provost  and J. Reiter  and E. J. Rhodes  and F. J. Rogers  and I. W. Roxburgh  and M. J. Thompson  and R. K. Ulrich },
title = {{The Current State of Solar Modeling}},
journal = {Science},
volume = {272},
number = {5266},
pages = {1286-1292},
year = {1996},
doi = {10.1126/science.272.5266.1286},
URL = {https://www.science.org/doi/abs/10.1126/science.272.5266.1286},
}

@article{EG16,
url = {https://doi.org/10.1515/cmam-2015-0034},
title = {{Mollification in Strongly Lipschitz Domains with Application to Continuous and Discrete De Rham Complexes}},
author = {A. Ern and J. L. Guermond},
pages = {51--75},
volume = {16},
number = {1},
journal = {Computational Methods in Applied Mathematics},
doi = {doi:10.1515/cmam-2015-0034},
year = {2016}
}

@book{EG_FE1,
author = {Ern, A. and Guermond, J. L.},
year = {2021},
pages = {},
title = {{Finite Elements I: Approximation and Interpolation}},
isbn = {978-3-030-56340-0},
doi = {10.1007/978-3-030-56341-7},
publisher={Springer}
}

@book{EG_FE2,
  title={{Finite Elements II: Galerkin Approximation, Elliptic and Mixed PDEs}},
  author={Ern, A. and Guermond, J. L.},
  year={2021},
  pages={},
  isbn = {978-3-030-56924-2},
  publisher={Springer}
}

@article{Halla21,
	title = {{Galerkin approximation of holomorphic eigenvalue problems: weak T-coercivity and T-compatibility}},
	volume = {148},
	url = {https://doi.org/10.1007/s00211-021-01205-8},
	doi = {10.1007/s00211-021-01205-8},
	pages = {387--407},
	number = {2},
	journal = {Numerische Mathematik},
	author = {Halla, M.},
	year = {2021},
}

@article{Cia12,
	title = {{T-coercivity: Application to the discretization of Helmholtz-like problems}},
	volume = {64},
	doi = {10.1016/j.camwa.2012.02.034},
	pages = {22--34},
	journal = {Computers \& Mathematics with Applications},
	author = {Ciarlet, P.},
	year = {2012},
}

@article{BCZ10,
	title = {{Time harmonic wave diffraction problems in materials with sign-shifting coefficients}},
	volume = {234},
	issn = {0377-0427},
	url = {https://www.sciencedirect.com/science/article/pii/S0377042709005196},
	doi = {https://doi.org/10.1016/j.cam.2009.08.041},
	pages = {1912--1919},
	number = {6},
	journal = {Journal of Computational and Applied Mathematics},
	author = {Bonnet-Ben Dhia, A. S. and Ciarlet, P. and Zwölf, C. M.},
	year = {2010},
}

@article{BCS02,
	title = {{Boundary element methods for Maxwell's equations on non-smooth domains}},
	volume = {92},
	doi = {10.1007/s002110100372},
	pages = {679--710},
	journal = {Numerische Mathematik},
	author = {Buffa, A. and Costabel, M.  and Schwab, C.},
	year = {2002},
}

@article{CC13,
	title = {{T-coercivity and continuous Galerkin methods: application to transmission problems with sign changing coefficients}},
	volume = {124},
	issn = {0945-3245},
	doi = {10.1007/s00211-012-0510-8},
	pages = {1--29},
	number = {1},
	journal = {Numerische Mathematik},
	author = {Chesnel, L. and Ciarlet, P.},
  year={2013}
}

@article{DCC12,
	title = {{\textit{{T}}-coercivity for scalar interface problems between dielectrics and metamaterials}},
	volume = {46},
	issn = {0764-583X, 1290-3841},
	doi = {10.1051/m2an/2012006},
	number = {6},
	journal = {ESAIM: Mathematical Modelling and Numerical Analysis},
	author = {Bonnet-Ben Dhia, A. S. and Chesnel, L. and Ciarlet, P.},
	year = {2012},
	pages = {1363--1387},
}

@article{DCC14,
	title = {T-{Coercivity} for the {Maxwell} {Problem} with {Sign}-{Changing} {Coefficients}},
	volume = {39},
	issn = {0360-5302, 1532-4133},
	doi = {10.1080/03605302.2014.892128},
	number = {6},
	journal = {Communications in Partial Differential Equations},
	author = {Bonnet-Ben Dhia, A. S. and Chesnel, L. and Ciarlet, P.},
	year = {2014},
	pages = {1007--1031},
}

@misc{vBZ24,
	title={{An adaptive mesh refinement strategy to ensure quasi-optimality of finite element methods for self-adjoint Helmholtz problems}}, 
	doi = {10.48550/arXiv.2403.06266},
	number = {{arXiv}:2403.06266},
	publisher = {{arXiv}},
	author = {van Beeck, T. and Zerbinati, U.},
	year = {{arXiv}:2403.06266, 2024},
	eprinttype = {arxiv},
  primaryClass={math.NA}
}

@misc{HH024,
	title = {{A new numerical method for scalar eigenvalue problems in heterogeneous, dispersive, sign-changing materials}},
	doi = {10.48550/arXiv.2401.16368},
	number = {{arXiv}:2401.16368},
	publisher = {{arXiv}},
	author = {Halla, M.  and Hohage, T.  and Oberender, F. },
	year = {{arXiv}:2401.16368, 2024},
	eprinttype = {arxiv},
}

@article{HN15,
	title = {{Convergence of infinite element methods for scalar waveguide problems}},
	volume = {55},
	issn = {1572-9125},
	doi = {10.1007/s10543-014-0525-x},
	language = {en},
	number = {1},
	journal = {BIT Numerical Mathematics},
	author = {Hohage, T. and Nannen, L.},
	month = mar,
	year = {2015},
}

@book{G11,
  title={{Elliptic problems in nonsmooth domains}},
  author={Grisvard, P.},
  year={2011},
  publisher={SIAM}
}

@article{Stummel_I,
  title={{Diskrete Konvergenz linearer Operatoren. I}},
  author={Stummel, F.},
  journal={Mathematische Annalen},
  year={1970},
  volume={190},
  pages={45-92},
}

@book{Vainikko,
author         = {Vainikko, G.},
publisher      = {B. G. Teubner Verlag},
title          = {{Funktionalanalysis der Diskretisierungsmethoden}},
year           = {1976}
}

@techreport{CD18,
  TITLE = {{Solving time-harmonic Galbrun's equation with an arbitrary flow. Application to Helioseismology}},
  AUTHOR = {Chabassier, J. and Durufl{\'e}, M.},
  URL = {https://inria.hal.science/hal-01833043},
  TYPE = {Research Report},
  NUMBER = {RR-9192},
  INSTITUTION = {{INRIA Bordeaux}},
  YEAR = {2018},
  MONTH = Jul,
  HAL_ID = {hal-01833043},
  HAL_VERSION = {v1},
}

@article{GBS10,
author = {Gizon, L. and Birch, A. and Spruit, H.},
year = {2010},
month = {01},
title = {{Local Helioseismology: Three Dimensional Imaging of the Solar Interior}},
volume = {48},
journal = {Annual Review of Astronomy and Astrophysics},
doi = {10.1146/annurev-astro-082708-101722}
}

@article{BH24,
doi = {10.1088/1361-6420/ad2b9a},
year = {2024},
month = {mar},
publisher = {IOP Publishing},
volume = {40},
number = {4},
pages = {045016},
author = {Müller, B. and Hohage, T. and Fournier, D. and Gizon, L.},
title = {{Quantitative passive imaging by iterative holography: the example of helioseismic holography}},
journal = {Inverse Problems},
}

@article{LB97,
doi = {10.1086/304445},
year = {1997},
month = {aug},
publisher = {},
volume = {485},
number = {2},
pages = {895},
author = {Lindsey, C. and Braun, D. C.},
title = {{Helioseismic Holography}},
journal = {The Astrophysical Journal},
}

@book{SBH19,
  title={{Variational techniques for elliptic partial differential equations: Theoretical tools and advanced applications}},
  author={Sayas, F. J. and Brown, T. S. and Hassell, M. E.},
  year={2019},
  publisher={CRC Press}
}

@data{HLvB25,
author = {Halla, M. and Lehrenfeld, C. and van Beeck, T.},
publisher = {GRO.data},
title = {{Replication Data for: Hybrid Discontinuous Galerkin Discretizations for the damped time-harmonic Galbrun’s equation}},
year = {2025},
version = {DRAFT VERSION},
doi = {10.25625/0DQ9JQ},
url = {https://doi.org/10.25625/0DQ9JQ}
}

@misc{FvBZ25,
      title={{Analysis and numerical analysis of the Helmholtz-Korteweg equation}}, 
      author={Farrell, P. E. and van Beeck, T. and Zerbinati, U.},
      year={{arXiv}:2503.10771, 2025},
      eprint={2503.10771},
      archivePrefix={arXiv},
      primaryClass={math.NA},
      url={https://arxiv.org/abs/2503.10771}, 
}

@article{SV85,
  title={{Norm estimates for a maximal right inverse of the divergence operator in spaces of piecewise polynomials}},
  journal={ESAIM: Mathematical Modelling and Numerical Analysis},
  author = {L. R. Scott and M. Vogelius},
  volume={19},
  number={1},
  pages={111--143},
  year={1985},
  publisher={EDP Sciences}
}

@article{BE08,
  title={{Discontinuous Galerkin approximation with discrete variational principle for the nonlinear Laplacian}},
  author={Burman, E. and Ern, A.},
  journal={Comptes Rendus. Math{\'e}matique},
  volume={346},
  number={17-18},
  pages={1013--1016},
  year={2008}
}

@article{DPE10,
  title={{Discrete functional analysis tools for discontinuous Galerkin methods with application to the incompressible Navier--Stokes equations}},
  author={Di Pietro, D. A. and Ern, A.},
  journal={Mathematics of Computation},
  volume={79},
  number={271},
  pages={1303--1330},
  year={2010}
}

@article{DPEL14,
  title={{An arbitrary-order and compact-stencil discretization of diffusion on general meshes based on local reconstruction operators}},
  author={Di Pietro, D. A. and Ern, A. and Lemaire, S.},
  journal={Computational Methods in Applied Mathematics},
  volume={14},
  number={4},
  pages={461--472},
  year={2014}
}

@article{DPE15,
  title={{A hybrid high-order locking-free method for linear elasticity on general meshes}},
  author={Di Pietro, D. A. and Ern, A.},
  journal={Computer Methods in Applied Mechanics and Engineering},
  volume={283},
  pages={1--21},
  year={2015},
  publisher={Elsevier}
}

@book {CEP20,
    author = {Cicuttin, M. and Ern, A. and Pignet, N.},
    title = {{Hybrid high-order methods---a primer with applications to solid mechanics}},
    series = {SpringerBriefs in Mathematics},
    publisher = {Springer, Cham},
      year = {2021},
     pages = {viii+136},
}

@article{BRT23,
  title={{Construction and analysis of a HDG solution for the total-flux formulation of the convected Helmholtz equation}},
  author={Barucq, H. and Rouxelin, N. and Tordeux, S.},
  journal={Mathematics of Computation},
  volume={92},
  number={343},
  pages={2097--2131},
  year={2023}
}

@article{BFFGP20,
author = {Barucq, H. and Faucher, F.  and Fournier, D. and Gizon, L. and Pham, H.},
title = {{Efficient and Accurate Algorithm for the Full Modal Green's Kernel of the Scalar Wave Equation in Helioseismology}},
journal = {SIAM Journal on Applied Mathematics},
volume = {80},
number = {6},
pages = {2657-2683},
year = {2020},
doi = {10.1137/20M1336709}
}

@article{BFFGP21,
title = {{Outgoing modal solutions for Galbrun's equation in helioseismology}},
journal = {Journal of Differential Equations},
volume = {286},
pages = {494-530},
year = {2021},
issn = {0022-0396},
doi = {https://doi.org/10.1016/j.jde.2021.03.031},
author = {H. Barucq and F. Faucher and D. Fournier and L. Gizon and H. Pham},
}

@article{PFFBG24,
title = {{Assembling algorithm for Green's tensors and absorbing boundary conditions for Galbrun's equation in radial symmetry}},
journal = {Journal of Computational Physics},
volume = {519},
pages = {113444},
year = {2024},
issn = {0021-9991},
doi = {https://doi.org/10.1016/j.jcp.2024.113444},
author = {H. Pham and F. Faucher and D. Fournier and H. Barucq and L. Gizon},
}

@article{BMMPP12,
  title={{Time-harmonic acoustic scattering in a complex flow: a full coupling between acoustics and hydrodynamics}},
  author={Bonnet-Ben Dhia, A. S. and Mercier, J. F. and Millot, F. and Pernet, S. and Peynaud, E.},
  journal={Communications in Computational Physics},
  volume={11},
  number={2},
  pages={555--572},
  year={2012},
  publisher={Cambridge University Press}
}

@article{FHPG24,
	author = {{Fournier, D.} and {Hohage, T.} and {Preuss, J.} and {Gizon, L.}},
	title = {{Learned infinite elements for helioseismology - Learning transparent boundary conditions for the solar atmosphere}},
	DOI= "10.1051/0004-6361/202449611",
	url= "https://doi.org/10.1051/0004-6361/202449611",
	journal = {A\&A},
	year = 2024,
	volume = 690,
	pages = "A86",
}

@article{KirbySherwinCockburn2012,
	author = {{Kirby, R.M.} and {Sherwin, S.J.} and {Cockburn, B.}},
	title = {{To CG or to HDG}},
	DOI= "10.1007/s10915-011-9501-7",
	url= "https://doi.org/10.1007/s10915-011-9501-7",
	journal = {J. Sci. Comp.},
	year = 2012,
	volume = 51,
	pages = "183--212",
}

@article{AA2017,
  title={{Computational helioseismology in the frequency domain: acoustic waves in axisymmetric solar models with flows}},
  author={Gizon, L. and Barucq, H. and Durufl{\'e}, M. and Hanson, C. S. and Legu{\`e}be, M. and Birch, A. C. and Chabassier, J. and Fournier, D. and Hohage, T. and Papini, E.},
  journal={Astronomy \& Astrophysics},
  volume={600},
  pages={A35},
  year={2017},
  publisher={EDP Sciences}
}

\appendix
\allowdisplaybreaks 
\section{Supplementary material to \cref{thm:weakTcompatibilityFulfilled}}\label{appendix:operators}
This section contains the operators defined in the proof of \cref{thm:weakTcompatibilityFulfilled}. In \emph{Step 1}, we defined the operator $\BnI$ as the sum of $\anI$ and $\KnI$, i.e. 

\begin{subequations}\label{eq:appendix:BnI}
   \begin{align}
      \spl \BnI &\u_n, \u_n' \spr_{\Xn} \coloneqq \nonumber \\ 
         &\spl c_s^2 \rho \div \v_{\tau}, \div \v_{\tau}' \spr + \spl c^2_s \rho \pi_n^l P_{L^2_0} \q \cdot \w_{\tau}, \pi_n^l P_{L^2_0} \q \cdot \w_{\tau}' \spr - s_n(\w_n,\w_n') \\
         &- \spl \rho i D^n_{\bflow} \v_n, i D^n_{\bflow} \v_n' \spr + \spl \rho \dopdSmall \w_n, \dopdSmall \w_n' \spr \\
         &+ \spl \rho \dopdSmall \w_n, i D^n_{\bflow} \v_n' \spr - \spl \rho i D^n_{\bflow} \v_n, \dopd \w_n' \spr \\
         &+ \spl \v_{\tau}, \v_{\tau}' \spr \! + \! C_1 \spl \Sol \u_n, \Sol \u_n' \spr \! + \! \spl \csr M_n \w_{n}, M_n \w_{n}' \spr \!  + \! \spl \csr \tilde{O}_n \u_n, \tilde{O}_n \u_n' \spr \\
         &+ \spl \rho (\matii + i \omega \gamma) \w_{\tau}, \w_{\tau}' \spr. 
   \end{align}
\end{subequations} 
and the operator $\BnII$ as the sum of $\anII$ and $\KnII$.
With $\anII$ as follows
\begin{subequations}
   \begin{align*}
      \anII(&\u_n,\u_n') \coloneqq a_n(T_n \u_n , \u_n') - \anI(\u_n,\u_n') \\
       & = \anII(\v_n,\v_n') + \anII(\v_n,\w_n') + \anII(\w_n,\v_n') + \anII(\w_n,\w_n')\nonumber \text{ with} \\[-4ex]
      \intertext{}
      \anII(\v_n,\v_n') = &\spl c_s^2 \rho \q \cdot \v_{\tau}, \div \v_{\tau}' \spr + \spl c_s^2 \rho \div \v_{\tau}, \q \cdot \v_{\tau}' \spr + \inner{\csr \q \! \cdot \! \v_{\tau}, \q \! \cdot \! \v_{\tau}'} \\
      &- \spl \rho (\omega+i\Omega \times) \v_{\tau}, (\omega+i\Omega \times) \v_{\tau}' \spr  - \spl \rho (\matii + i \omega \gamma) \v_{\tau}, \v_{\tau}' \spr \\
      &- \spl \rho (\omega+i\Omega \times) \v_{\tau}, i \ddiffb \v_n' \spr - \spl \rho i \ddiffb \v_n, (\omega+i\Omega \times) \v_{\tau}' \spr \\[-4ex]
      \intertext{}
      \anII(\v_n,\w_n')=&-  \! \spl \rho  (\omega \! + \! i \Omega \times) \v_{\tau}, (\omega \! + \! i \ddiffb \! + \! i \Omega \times) \w_n' \spr  \! -  \! \spl \rho (\matii + \! i \omega \gamma) \v_{\tau}, \w_{\tau}' \spr + \inner{\csr \q \! \cdot \! \v_{\tau}, \q \! \cdot \! \w_{\tau}'} \\
      &- \! \spl c_s^2 \rho (\div + \pi_n^l P_{L^2_0} \q \cdot\! ) \v_{\tau},\! M_n \w_{n}' \! \! \! + \! \tilde{O}_n \u_n' \spr \! + \! \spl c_s^2 \rho (\Id - \pi_n^l  P_{L^2_0} ) (\q \cdot \v_{\tau}\! ), \ddiv \w_{n}' \spr \\
      &+ \inner{\csr \div \v_{\tau}, (\Id - \pi_n^l P_{L^2_0}) \q \cdot \w_n'} - \spl c_s^2 \rho \pi_n^l P_{L^2_0} (\q \cdot \v_{\tau}), \pi_n^l P_{L^2_0} (\q \cdot \w_{\tau}') \spr \\[-4ex]
      \intertext{}
      \anII(\w_n,\v_n')=& \spl \rho (\omega  \! +  \! i \ddiffb  \! + \! i \Omega \times) \w_n, (\omega  \! +  \! i \Omega \times) \v_{\tau}' \spr + \spl \rho (\matii  \! +  \! i\omega \gamma)\w_{\tau}, \v_{\tau}' \spr - \inner{\csr \q \! \cdot \! \w_{\tau}, \q \! \cdot \! \v_{\tau}'}  \\
      &+ \spl c_s^2 \rho (M_n \! \w_{n} \! +\! \tilde{O}_n \u_n), (\div + \pi_n^l P_{L^2_0} \q \cdot) \v_{\tau}' \spr \! -\! \spl c_s^2 \rho \ddiv \w_{n}, (\Id - \pi_n^l P_{L^2_0} ) (\q \cdot \v_{\tau}') \spr\\
      &- \inner{\csr (\Id-\pi_n^l P_{L^2_0})(\q \cdot \w_{\tau}), \div \v_{\tau}'} + \spl c_s^2 \rho \pi_n^l P_{L^2_0} (\q \cdot \w_{\tau}), \pi_n^l P_{L^2_0} (\q \cdot \v_{\tau}') \spr \\[-4ex]
      \intertext{}
      \anII(\w_n,\w_n')= &- \spl c_s^2 \rho (\Id-\pi_n^l  P_{L^2_0} )(\q \cdot \w_{\tau}), \ddiv \w_{n}' \spr 
      - \spl c_s^2 \rho \ddiv \w_{n}, (\Id  - \pi_n^l P_{L^2_0} )(\q \cdot \w_{\tau}') \spr \\
      &- \spl c_s^2 \rho (M_n \w_{n} + \tilde{O}_n \u_n), M_n \w_{n}' + \tilde{O}_n \u_n' \spr  
   \end{align*}
  \end{subequations}
we obtain
\begin{align*}
       \spl &\BnII \u_n, \u_n' \spr_{\Xn} \coloneqq  \\
       &C_2 \big(  \spl \v_{\tau}, \v_{\tau}' \spr + \spl \Sol \u_n, \Sol \u_n' \spr + \spl \csr \tilde{O}_n \u_n, \tilde{O}_n \u_n' \spr  \\
       & + \spl \csr M_n \w_{n}, M_n \w_{n}' \spr + \spl \text{mean}(\q \cdot \! \w_{\tau}), \text{mean}(\q \cdot \! \w_{\tau}') \spr \big) \\[-3ex]
       \intertext{}
       & + \spl c_s^2 \rho \q \cdot \! \v_{\tau}, \div \v_{\tau}' \spr + \spl c_s^2 \rho \div \v_{\tau}, \q \cdot \! \v_{\tau}' \spr + \inner{\csr \q \! \cdot \! \v_{\tau}, \q \! \cdot \! \v_{\tau}'}  \\
       &- \spl \rho (\omega+i\Omega \times) \v_{\tau}, (\omega+i\Omega \times) \v_{\tau}' \spr  - \spl \rho (\matii + i \omega \gamma) \v_{\tau}, \v_{\tau}' \spr \\
       &- \spl \rho (\omega+i\Omega \times) \v_{\tau}, i \ddiffb \v_n' \spr - \spl \rho i \ddiffb \v_n, (\omega+i\Omega \times) \v_{\tau}' \spr \\[-3ex]
       \intertext{}
       &- \! \spl \rho  (\omega \! +\! i \Omega \times) \v_{\tau}, (\omega \!+ \!i \ddiffb + i \Omega \times) \w_n' \spr \!-\! \spl \rho (\matii \!+\! i \omega \gamma) \v_{\tau}, \w_{\tau}' \spr  \!+\! \inner{\csr \q \! \cdot \! \v_{\tau}, \q \! \cdot \! \w_{\tau}'} \\
       &- \! \spl c_s^2 \rho (\div + \pi_n^l P_{L^2_0} \q \cdot \! ) \v_{\tau},\! M_n \w_{n}' \! \! \! + \! \tilde{O}_n \u_n' \spr \! + \! \spl c_s^2 \rho (\Id - \pi_n^l  P_{L^2_0} ) (\q \cdot \! \v_{\tau}\! ), \ddiv \w_{n}' \spr \\
       &+ \inner{\csr \div \v_{\tau}, (\Id - \pi_n^l P_{L^2_0}) \q \cdot \! \w_n'} - \spl c_s^2 \rho \pi_n^l P_{L^2_0} (\q \cdot \! \v_{\tau}), \pi_n^l P_{L^2_0} (\q \cdot \! \w_{\tau}') \spr\\[-3ex]
       \intertext{}
       &+ \spl \rho (\omega \! + \! i \ddiffb + i \Omega \times) \w_n,\! (\omega \!+\! i \Omega \times) \v_{\tau}' \spr \!+ \!\spl \rho (\matii \!+\! i\omega \gamma)\w_{\tau}, \v_{\tau}' \spr  \! - \! \inner{\csr \q \! \cdot \! \w_{\tau}, \q \! \cdot \! \v_{\tau}'} \\
       &+ \spl c_s^2 \rho (M_n \! \w_{n} \! +\! \tilde{O}_n \u_n), (\div + \pi_n^l P_{L^2_0} \q \cdot \!) \v_{\tau}' \spr \! -\! \spl c_s^2 \rho \ddiv \w_{n}, (\Id - \pi_n^l P_{L^2_0} ) (\q \cdot \! \v_{\tau}') \spr\\
       &- \inner{\csr (\Id-\pi_n^l P_{L^2_0})(\q \cdot \! \w_{\tau}), \div \v_{\tau}'} + \spl c_s^2 \rho \pi_n^l P_{L^2_0} (\q \cdot \! \w_{\tau}), \pi_n^l P_{L^2_0} (\q \cdot \! \v_{\tau}') \spr \\[-3ex]
       \intertext{}
       &- \spl c_s^2 \rho (\Id-\pi_n^l  P_{L^2_0} )(\q \cdot \! \w_{\tau}), \ddiv \w_{n}' \spr 
       - \spl c_s^2 \rho \ddiv \w_{n}, (\Id  - \pi_n^l P_{L^2_0} )(\q \cdot \! \w_{\tau}') \spr \\
       &- \spl c_s^2 \rho (M_n \w_{n} + \tilde{O}_n \u_n), M_n \w_{n}' + \tilde{O}_n \u_n' \spr.  
\end{align*}
The terms added with $\KnI$ and $\KnII$ are subtracted through the operator $K_n$, which is given by 
\begin{subequations}
   \begin{align}
         \spl K_n &\u_n, \u_n' \spr_{\Xn} \coloneqq \nonumber \\
         &- (1+C_2) \spl \v_{\tau}, \v_{\tau}' \spr - (C_1 + C_2) \spl \Sol \u_n, \Sol \u_n' \spr \label{eq:Kn:a}\\
         &- (1+C_2) \spl \csr M_n \w_{n}, M_n \w_{n}' \spr - C_2 \spl \text{mean}(\q \cdot \w_{\tau}), \text{mean}(\q \cdot \w_{\tau}') \spr \label{eq:Kn:b}\\ 
         &- (1+C_2) \spl \csr \tilde{O}_n \u_n, \tilde{O}_n \u_n' \spr. \label{eq:Kn:c}
   \end{align}
\end{subequations}
In \emph{Step 4}, we defined the compact operator $K$ in \eqref{eq:K:a}-\eqref{eq:K:b} and set $B \coloneqq AT - K$, i.e. 
\begin{subequations}\label{eq:B}
   \begin{align}
      \spl B \u&, \u' \spr_{\X} \coloneqq \nonumber \\
      &\spl c_s^2 \rho \div \v, \div \v' \spr - \spl \rho i \diffb \v, i \diffb \v' \spr + \spl c_s^2 \rho P_{L^2_0}(\q \cdot \w), P_{L^2_0}(\q \cdot \w') \spr \\
      &- \spl \rho i \diffb \v, (\omega + i\diffb + i\Omega \times) \w' \spr + \spl \rho (\omega + i \diffb + i \Omega \times) \w, i \diffb \v' \spr \\
      &+ \spl \rho (\omega + i \diffb + i \Omega \times) \w, (\omega + i \diffb + i \Omega \times) \w' \spr + \spl \rho (i \omega \gamma + \matii) \w, \w' \spr \\
      &+ \spl \v, \v' \spr + C_1 \spl \v, \v' \spr + \spl \csr M \w, M\w' \spr \\[.5cm]
      &+ C_2 \Big(  \spl \v, \v' \spr + \spl \v, \v' \spr + \spl \csr M \w, M \w' \spr + \spl \text{mean}(\q \cdot \w), \text{mean}(\q \cdot \w') \spr \Big)\\[.5cm]
      &+ \spl c_s^2 \rho \q \cdot \v, \div \v' \spr + \spl c_s^2 \rho \div \v, \q \cdot \v' \spr - \spl \rho (\omega + i \Omega \times) \v, (\omega + i \Omega \times) \v' \spr \\ 
      &- \spl \rho (\omega + i\Omega \times) \v, i \diffb \v' \spr - \spl \rho i \diffb \v, (\omega + i \Omega \times) \v' \spr - i \omega \spl \gamma \rho \v, \v' \spr - \spl \rho \matii \v, \v' \spr \\[.5cm] 
      &- \spl \rho \matii \v, \w' \spr - i \omega \spl \gamma \rho \v, \w' \spr - \spl c_s^2 \rho P_{L^2_0} (\q \cdot \v), P_{L^2_0} (\q \cdot \w') \spr \\
      &- \spl \rho (\omega + i \Omega \times) \v, (\omega + i \diffb + i \Omega \times) \w' \spr - \spl c_s^2 \rho (\div + P_{L^2_0} \q \cdot) \v, M \w' \spr \\ 
      &+ \spl c_s^2 \rho \text{mean}(\q \cdot \v), \div \w' \spr \\[.5cm]
      &+ \spl  \rho \matii \w, \v' \spr + i \omega \spl \gamma \rho \w, \v' \spr + \spl c_s^2 \rho P_{L^2_0} (\q \cdot \w), P_{L^2_0} (\q \cdot \v') \spr \\
      &+ \spl \rho (\omega + i \diffb + i \Omega \times) \w, (\omega + i \Omega \times) \v' \spr + \spl c_s^2 \rho M \w, (\div + P_{L^2_0} \q \cdot) \v' \spr \\
      &- \spl c_s^2 \rho \div \w, \text{mean}(\q \cdot \v') \spr \\[.5cm]
      &- \spl c_s^2 \rho \text{mean}(\q \cdot \w), \div \w' \spr - \spl c_s^2 \rho \div \w, \text{mean}(\q \cdot \w') \spr - \spl c_s^2 \rho M \w, M \w' \spr.
   \end{align}
\end{subequations}

\end{document}